\documentclass[11pt]{amsart}
\usepackage{amsmath,amscd,amssymb,amsfonts,amsthm,hyperref}
\usepackage[all]{xy}

 \newtheorem{theorem}{Theorem}[section]
\newtheorem{corollary}[theorem]{Corollary}
\newtheorem{proposition}[theorem]{Proposition}

\newtheorem{lemma}[theorem]{Lemma}
\theoremstyle{remark}
\newtheorem{remark}[theorem]{Remark}
\newtheorem{definition}[theorem]{Definition}

\newtheorem{example}[theorem]{Example}
\newtheorem{examples}[theorem]{Examples}
\newcommand\A{\mathbb{A}}

\renewcommand\L{\mathcal{L}}

\newcommand{\TM}{\mathbb{T}M}

\renewcommand{\TH}{\mathbb{T}H}

\newcommand{\T}{\mathbb{T}}

\newcommand{\n}{\mathfrak{n}}

\newcommand{\pr}{\on{pr}}

\newcommand\lie[1]{\mathfrak{#1}}
\renewcommand{\k}{\lie{k}}
\newcommand{\h}{\lie{h}}

\newcommand{\g}{\lie{g}}

\newcommand{\q}{\lie{q}}
\newcommand{\on}{\operatorname}

\newcommand{\Ad}{ \on{Ad} }
\newcommand{\ad}{ \on{ad} }

\newcommand{\End}{ \on{End} }

\renewcommand{\ker}{ \on{ker}}

\newcommand{\Mult}{{\on{Mult}}}

\newcommand{\da}{\dasharrow}

\renewcommand\a{\mathsf{a}}

\newcommand{\hra}{\hookrightarrow}
\newcommand{\xra}{\xrightarrow}

\newcommand{\rra}{\rightrightarrows}

\renewcommand{\d}{{\mbox{d}}}
\newcommand{\dd}{\mf{d}}

\newcommand{\ol}{\overline}

\newcommand\sig{\sigma}
\newcommand\eps{\epsilon}
\newcommand\Om{\Omega}
\newcommand\om{\omega}

\newcommand{\f}{\tfrac}

\newcommand{\p}{\partial}
\renewcommand{\l}{\langle}
\renewcommand{\r}{\rangle}
\newcommand{\hh}{{ \f{1}{2}}}
\newcommand{\ti}{\tilde}

\newcommand\pt{\on{pt}}

\newcommand\beqn{\begin{equation}}
\newcommand\eeqn{\end{equation}}
\newcommand{\ca}{\mathcal}
\newcommand{\wh}{\widehat}
\newcommand{\wt}{\widetilde}
\newcommand{\mf}{\mathfrak}
\newcommand{\beq}{\begin{eqnarray*}}
\newcommand{\eeq}{\end{eqnarray*}}

\newcommand{\Cour}[1]      {[\![#1]\!]}

\begin{document}

\title[]{Dirac Lie groups}

\author{D. Li-Bland}
\address{University of Toronto, Department of Mathematics,
40 St George Street, Toronto, Ontario M4S2E4, Canada }
\email{dbland@math.toronto.edu}

\author{E. Meinrenken}
\address{University of Toronto, Department of Mathematics,
40 St George Street, Toronto, Ontario M4S2E4, Canada }
\email{mein@math.toronto.edu}

\date{\today}

\begin{abstract}
A classical theorem of Drinfel'd states that the category of simply connected Poisson Lie groups $H$
is isomorphic to the category of Manin triples $(\dd,\g,\h)$, where $\h$ is the Lie algebra of $H$.
In this paper, we consider \emph{Dirac Lie groups}, that is, Lie groups $H$ endowed with a multiplicative Courant algebroid $\A$ and a Dirac structure $E\subseteq \A$ for which the multiplication is a Dirac morphism. 
It turns out that the simply connected Dirac Lie groups are classified by so-called \emph{Dirac
Manin triples}. We give an explicit construction of the Dirac Lie group structure defined by a Dirac Manin 
triple, and develop its basic properties. 
\end{abstract}

\maketitle
\setcounter{tocdepth}{2}

{\small \tableofcontents \pagestyle{headings}}

\setcounter{section}{-1}
\vskip.3in
\section{Introduction}\label{sec:intro}

Dirac structures were introduced by T. Courant \cite{Courant:1990uy} as a common framework for closed 2-forms 
and Poisson structures on manifolds. He showed that the integrability condition $\d\om=0$ for 2-forms and 
$[\pi,\pi]=0$ for bivector fields admits a common generalization to an integrability condition on Lagrangian subbundles   
$E\subseteq \TM=TM\oplus T^*M$ relative to a certain bracket on $\Gamma(\TM)$. Liu-Weinstein-Xu \cite{ManinTriplesBi}
generalized Courant's original set-up, replacing $\TM$ with a more general notion of a \emph{Courant algebroid}
$\A\to M$. The theory of Courant algebroids and Dirac structures was clarified and simplified in the work of 
Dorfman \cite{Dorfman:1993us}, \v{S}evera \cite[Letter no.7]{LetToWein}, Roytenberg \cite{Roytenberg99}, 
Uchino \cite{Uchino02}, and others. It has recently gained attention through the development 
of generalized complex geometry \cite{Gualtieri:2004wh,Hitchin:2003kx}, and it provides a unified setting for 
various types moment maps \cite{Alekseev:2009tg,Bursztyn:2009wi}.

A Poisson Lie group is a Lie group $H$, equipped with a Poisson
structure such that the multiplication map is Poisson. To extend this
definition to Dirac geometry, it is required that the Courant
algebroid $\A$ itself carries a multiplicative structure. As suggested
by Mehta \cite{Mehta:2009js} and further explored in
\cite{LiBland:2010wi}, we require that $\A$ carries a
$\ca{VB}$-groupoid structure $\A\rra\g$ over the group $H\rra \pt$, in
such a way that the groupoid multiplication is a Courant morphism
$\Mult_\A\colon\A\times\A\da \A$. (For the standard Courant algebroid
$\A=\TH$ this structure is automatic, with $\g=\h^*$.) One then has a
notion of a \emph{multiplicative} Dirac structure $E\subseteq \A$. In
the case of $\TH$ these were classified in the work of Ortiz
\cite{Ortiz:2008bd} and Jotz \cite{Jotz:2009va}, independently. While
\cite{Jotz:2009va,Ortiz:2008bd} refer to multiplicative Dirac
structures as Dirac Lie group structures, we will reserve this latter
term for the case that the multiplication map is a Dirac morphism
(i.e~ a morphism of Manin pairs as in \cite{Bursztyn:2009wi}). For
$\A=\TH$, only the Poisson Lie group structures are Dirac Lie group
structures in our sense, but many more examples are obtained by
considering more general Courant algebroids. These include the well known
Cartan-Dirac structure (cf.~ \cite{Alekseev:2009tg} and references therein), 
and the examples in 
Section 5 of \cite{Klimczyk:1997op}. One of the goals of
this paper is to develop the theory of Dirac Lie groups in this
setting. The super-geometric interpretation of Dirac Lie group
structures was previously studied in \cite{LiBland:2010wi} under the
name \emph{MP Lie groups}.

By Drinfel'd's theorem \cite{Drinfeld83}, the category of simply connected Poisson Lie groups $H$
is canonically equivalent to the category of Manin triples $(\dd,\g,\h)$. That is, $\dd$ is a Lie algebra with a
vector space decomposition into two Lie subalgebras $\g,\h$, and equipped with a non-degenerate invariant symmetric bilinear form (`metric') for which $\g,\h$ are Lagrangian. According to a recent result of Michal Siran \cite{Siran11}, 
the non-simply connected Poisson Lie groups are similarly classified by $H$-equivariant Manin triples. 

We will show that Dirac Lie groups $H$ are classified by $H$-equivariant \emph{Dirac Manin triples} $(\dd,\g,\h)_\beta$. These consist of a Lie algebra $\dd$ with a vector space direct sum decomposition into two Lie subalgebras $\g,\h$, together with an invariant symmetric bilinear form $\beta$ on the dual space $\dd^*$ such that $\g$ is $\beta$-coisotropic, i.e.~ the restriction of $\beta$ to 
$\on{ann}(\g)$ is zero. Here $\beta$ may be degenerate or even zero, and there is no compatibility requirement between $\beta$ and $\h$. We will prove:
\begin{theorem}\label{th:main}
The category of Dirac Lie groups and the category of equivariant Dirac Manin triples are canonically 
equivalent.   
\end{theorem}
Theorem \ref{th:main} may be deduced from the classification of \emph{MP Lie groups} in \cite{LiBland:2010wi}, but 
we will give a direct proof, not using super geometry.

The functor from Dirac Manin triples $(\dd,\g,\h)_\beta$ to Dirac Lie groups 
is constructed as follows. As a first step, we use a reduction procedure to construct a new Dirac Manin triple 
$(\q,\g,\mf{r})_\gamma$, with a Lie algebra homomorphism $f\colon \q\to \dd$ taking 
$\mf{r}$ to $\h$ and restricting to the identity on $\g$. The new Dirac Manin triple is such that 
$\gamma$ is non-degenerate and $\g$ is Lagrangian in $\q$. The corresponding Dirac Lie group $(\A,E)$ is described using a  `left trivialization' 
\[ \A=H\times\q,\ \ E=H\times\g,\] 
where $H\times\q$ is an \emph{action Courant
algebroid} \cite{LiBland:2009ul}. An explicit description of the groupoid structure in terms
of this trivialization is given in Theorem \ref{th:mainB}. It is rather
cumbersome, however, to verify the compatibility properties from these explicit formulas. 
Therefore, we show that one can also obtain $(\A,E)$ by co-isotropic 
reduction of the direct product of the multiplicative Manin pairs $(\TH,TH)$ and 
$(\ol{\q}\oplus \q,\g\oplus \g)$ (where $\ol{\q}\oplus \q\rra\q$ carries the pair groupoid structure).

Of particular interest are the Dirac Lie group structures $(\A,E)$ over $H$ for which the underlying Courant algebroid is \emph{exact}, in the sense of \v{S}evera. We prove that this is the case 
if and only if $\beta$ is non-degenerate and $\g$ is Lagrangian. In this case we construct a \emph{canonical} isomorphism with the Courant algebroid $\TH_\eta$, with twisting 3-form the Cartan 3-form over $H$. 
We hence obtain another concrete description of the Dirac Lie group structure, in terms of differential 
forms and vector fields. 


The organization of the paper is as follows. In Section~\ref{sec:prel} we summarize the basic theory of Courant algebroids, Dirac structures, and their morphisms. In Section~\ref{sec:DLG} we introduce and motivate our definition of Dirac Lie groups. Next, in Sections~\ref{sec:classDLG} and~\ref{sec:mor} we classify Dirac Lie groups and their morphisms in terms of Lie theoretic data. Then, in Section~\ref{sec:explF} we summarize the structural formulas for Dirac Lie groups obtained in the previous two sections, and use them to describe some examples explicitly. Following this, in Section~\ref{sec:compl} we relate Dirac Lie groups to the theory of quasi-Poisson geometry \cite{Alekseev99} and the language of quasi-Lie bialgebroids \cite{Roytenberg02,KosmannSchwarzbach:2005wc,PonteXu:08}. Finally, in Section~\ref{sec:exact} we study those Dirac Lie groups for which the underlying Courant algebroid is exact. 

\vskip.1in

\noindent{\bf Acknowledgements.} We thank Henrique Bursztyn and Pavol \v{S}evera for 
discussions and helpful comments. David Li-Bland was supported by an NSERC CGS-D Grant; Eckhard Meinrenken was supported by
an NSERC Discovery Grant.

\section{Preliminaries}\label{sec:prel}
We begin with a quick review of Courant algebroids and Dirac structures. A more slow-paced overview and 
further references can be found in our paper \cite{LiBland:2009ul}. 
\subsection{Basic definitions}\label{subsec:basic}
A \emph{Courant algebroid} over a manifold $M$ is a vector bundle $\A\to M$, together with a bundle 
map $\mf{a}\colon \A\to TM$ called the \emph{anchor}, a bundle metric\footnote{In this paper, we take 
`metric' to mean a non-degenerate symmetric bilinear form.} $\l\cdot,\cdot\r$, and 
 a bilinear bracket $\Cour{\cdot,\cdot}$ on its space of sections $\Gamma(\A)$. These are required 
to satisfy the following axioms, for all sections $x_1,x_2,x_3\in\Gamma(\A)$:
\begin{enumerate}
\item[c1)] $\Cour{x_1,\Cour{x_2,x_3}}=\Cour{\Cour{x_1,x_2},x_3}
+\Cour{x_2,\Cour{x_1,x_3}}$, 
\item[c2)] $\a(x_1)\l x_2,x_3\r=\l \Cour{x_1,x_2},\,x_3\r+\l x_2,\,\Cour{x_1,x_3}\r$,
\item[c3)] $\Cour{x_1,x_2}+\Cour{x_2,x_1}=\a^*(\d \l x_1,x_2\r)$.
\end{enumerate}
Here $\a^*\colon T^*M\to \A^*\cong\A$ is the dual map to $\a$. The axioms c1)-c3) imply various other properties, 
in particular
\begin{enumerate}
\item[c4)] $\Cour{x_1,fx_2}=f\Cour{x_1,x_2}+\a(x_1)(f)x_2$,
\item[c5)] $\Cour{fx_1,x_2}=f\Cour{x_1,x_2}-\a(x_2)(f)x_1+\l x_1,x_2\r \a^*(\d f)$ 
\item[c6)] $\a(\Cour{x_1,x_2})=[\a(x_1),\a(x_2)]$,
\end{enumerate}
for sections $x_i\in\Gamma(\A)$ and functions $f\in C^\infty(M)$. We
will refer to the bracket $\Cour{\cdot,\cdot}$ as the \emph{Courant
bracket} (it is also know as the \emph{Dorfman bracket}). 

Following \v{S}evera \cite{LetToWein}, the Courant algebroid is called
\emph{exact} if the sequence
\[ 0\to T^*M\to \A\to TM\to 0\]
is exact. In particular, the bundle metric of $\A$ is of split signature, and 
$T^*M$ is a Lagrangian subbundle. Any choice of a Lagrangian splitting $\mathsf{l}\colon TM\to \A$ 
gives an isomorphism $\A\xra{\cong} TM\oplus T^*M$, with inverse map $v+\alpha\mapsto \mathsf{l}(v)+\a^*(\alpha)$. 
Under this identification, the anchor map $\a$ becomes projection 
to the first summand, the bilinear form is
$\l v_1+\alpha_1,v_2+\alpha_2\r=\l\alpha_2,v_1\r+\l \alpha_1,v_2\r$, and the Courant bracket reads
\[ \Cour{v_1+\alpha_1,v_2+\alpha_2}=[v_1,v_2]+\L_{v_1}\alpha_2-\iota(v_2)\d\alpha_1+\iota(v_1)\iota(v_2)\eta,\]
for vector fields $v_i\in\mf{X}(M)$ and 1-forms $\alpha_i\in\Om^1(M)$. 
Here  $\eta\in \Om^3(M)$ is the closed 3-form obtained as 
\begin{equation}\label{eq:3form}
 \iota(v_1)\iota(v_2)\iota(v_3)\eta
=\l \Cour{\mathsf{l}(v_1),\mathsf{l}(v_2)},\mathsf{l}(v_3)\r.\end{equation}
Conversely, given a closed 3-form $\eta$, these formulas define a Courant algebroid structure 
on $TM\oplus T^*M$. 
We will denote this Courant algebroid by $\TM_\eta$, or simply $\TM$ if $\eta=0$.
The set of Lagrangian splittings is an affine space, with $\Om^2(M)$ as its space of motions, 
and a change of the splitting by a 2-form $\om$ changes $\eta$ to $\eta+\d\om$. 

Another class of examples of Courant algebroids is as follows.
Suppose $\g$ is a Lie algebra equipped with an invariant metric. 
Given a Lie algebra action $\varrho\colon \g\to \mf{X}(M)$ on a manifold $M$, 
let $\A=M\times\g$ with anchor map $\a(m,\xi)=\varrho(\xi)_m$, and with
the bundle metric coming from the metric on $\g$. As shown in
\cite{LiBland:2009ul}, the Lie bracket on constant sections $\g\subseteq
C^\infty(M,\g)=\Gamma(\A)$ extends to a Courant bracket \emph{if and only if}
the stabilizers $\g_m\subseteq \g$ are coisotropic, i.e.~ $\g_m\supseteq \g_m^\perp$. Explicitly, for $\xi_1,\xi_2\in \Gamma(\A)=C^\infty(M,\g)$ the
Courant bracket reads (see \cite[$\mathsection$ 4]{LiBland:2009ul})
\begin{equation}
\label{eq:actioncourant} \Cour{\xi_1,\xi_2}=[\xi_1,\xi_2]+\L_{\varrho(\xi_1)}\xi_2-\L_{\varrho(\xi_2)}\xi_1+\varrho^*\l\d\xi_1,\xi_2\r.\end{equation}
Here $\varrho^*\colon T^*M\to M\times\g$ is the dual map to the action map $\varrho\colon M\times\g\to TM$, 
using the metric to identify $\g^*\cong \g$. 
We refer to $M\times\g$ with bracket \eqref{eq:actioncourant} as an \emph{action Courant algebroid}. 

For any Courant algebroid $\A$, we denote by $\ol{\A}$ the Courant algebroid with the same 
bracket and anchor, but with the opposite bilinear form.  

\subsection{Involutive subbundles}
Let $\A\to M$ be a Courant algebroid. 
A subbundle $E\subseteq \A$ along a submanifold $S\subseteq M$ is called \emph{involutive} if it has the 
property 
\[ x_1|_S,\ x_2|_S\in\Gamma(E) 
\Rightarrow \Cour{x_1,x_2}|_S\in \Gamma(E).\]  
We stress that this property need not define a bracket on $\Gamma(E)$, in general.
Using the properties c4 and c5 of Courant algebroids, one finds that if $E\to S$ is an involutive sub-bundle, 
with $0<\on{rank}(E)<\on{rank}(\A)$, then 
\[ \a(E)\subseteq TS,\ \ \a(E^\perp)\subseteq TS.\]
Note that the second property is not preserved under intersections of bundles, and indeed 
a sub-bundle given as the intersection of involutive sub-bundles need not be involutive (unless these subbundles are defined over the same submanifold). 
An involutive Lagrangian sub-bundle $E\subseteq \A$  along $S\subseteq M$ is called
a \emph{Dirac structure along $S$}.   For instance, if $\A=\T M$ is the standard Courant algebroid, 
then $T^*M|_S$ and $TS\oplus \on{ann}(TS)$ are Dirac structures along $S$. 

A Dirac structure along $S=M$ is
simply called a Dirac structure. Dirac structures were introduced by
Courant \cite{Courant:1990uy} and Liu-Weinstein-Xu
\cite{ManinTriplesBi}; the notion of a Dirac structure along a
submanifold goes back to \v{S}evera \cite{LetToWein} and was developed
in \cite{Alekseev:2002tn,Bursztyn:2009wi,PonteXu:08} .

\subsection{Courant relations}
A smooth relation $S\colon M_0\da M_1$ between manifolds is an immersed submanifold 
$S\subseteq M_1\times M_0$. We will write $m_0\sim_S m_1$ if $(m_1,m_0)\in S$. 
Given smooth relations $S\colon M_0\da M_1$ and 
$S'\colon M_1\da M_2$, the set-theoretic composition $S'\circ S$ is the image of 
\begin{equation}\label{eq:intersection}
 S'\diamond S=(S'\times S)\cap (M_2\times (M_1)_\Delta \times M_0)
 \end{equation}
under projection to $M_2\times M_0$. We say that the two relations \emph{compose cleanly} if 
\eqref{eq:intersection} is a clean intersection in the sense of Bott (i.e. it is smooth, and the intersection of the tangent bundles is the tangent bundle of the intersection), and the map from
$S'\diamond S$ to $M_2\times M_0$ has constant rank. In this case, the composition 
$S'\circ S\colon M_0\da M_2$ is a well-defined smooth relation. See Appendix~\ref{app:compos} for more information on the composition of smooth relations. For background on clean intersections of manifolds, see 
e.g.~ \cite[page 490]{ho:an3}. 
 
Specializing to vector bundles, Lie algebroids 
and Courant algebroids, we define
\begin{definition}
\begin{enumerate}
\item 
A \emph{vector bundle relation} (\emph{$\ca{VB}$-relation}) $R\colon V_0\da V_1$ between vector bundles $V_i\to M_i$ is a subbundle $R\subseteq V_1\times V_0$ along a submanifold $S\subseteq M_1\times M_0$. 
\item 
A \emph{Lie algebroid relation} (\emph{$\ca{LA}$-relation}) $R\colon E_0\da E_1$ between Lie algebroids $E_i\to M_i$ is a Lie subalgebroid $R\subseteq E_1\times E_0$ along a submanifold $S\subseteq M_1\times M_0$. 
\item 
A \emph{Courant relation} (\emph{$\ca{CA}$-relation}) $R\colon \A_0\da \A_1$ between Courant algebroids $\A_i\to M_i$ is a Dirac structure $R\subseteq \A_1\times \ol{\A_0}$ along a submanifold $S\subseteq M_1\times M_0$. 
\end{enumerate}
For a $\ca{VB}$-relation $R\colon V_0\da V_1$, with underlying relation $S\colon M_0\da M_1$, we 
define $\ker(R)\subseteq p_{M_0}^*V_0,\ \on{ran}(R)\subseteq p_{M_1}^*V_1$ to be the kernel and range of the bundle map $R\to p_{M_1}^*V_1,\ 
(v_1,v_0)\mapsto v_1$ (where $p_{M_i}\colon S\to M_i,\ (m_1,m_0)\mapsto m_i$).  
\end{definition}
We will sometimes depict $\ca{VB}$-relations by diagrams as follows:
\[ \xymatrix{ {V_0} \ar@{-->}[r]^{R}\ar[d] & {V_1}\ar[d] \\
M_0 \ar@{-->}[r]_{S} & M_1
}\]
The dashed arrow $\da$ for $S$ is replaced by a solid arrow if $S$ is the graph of a map, and 
similarly for $R$. Given a $\ca{VB}$-relation  $R\colon V_0\da V_1$ and sections $\sig_i\in\Gamma(V_i)$, we will write 
$\sig_0\sim_R \sig_1$ if $(\sig_1,\sig_0)$ restricts to a section of $R$. Given a relation $S\colon M_0\to M_1$ 
and functions $f_i\in C^\infty(M_i)$, we write $f_0\sim_S f_1$ if $f_0(m_0)=f_1(m_1)$ for all 
$(m_1,m_0)\in S$. The following is clear from the definitions: 

\begin{proposition}
Suppose $\A_0,\A_1$ are Courant algebroids and $R\colon \A_0\to \A_1$ is a $\ca{VB}$-relation
with underlying relation $S\colon M_0\da M_1$.  
Suppose $\sig_0\sim_R \sig_1$ and $\tau_0\sim_R \tau_1$. Then 
\begin{enumerate}
\item If $R$ is Lagrangian, $\l\sig_0,\tau_0\r\sim_S \l\sig_1,\tau_1\r$.  
\item If $R$ is involutive, $\Cour{\sig_0,\tau_0}\sim_R \Cour{\sig_1,\tau_1}$. 
\end{enumerate}
\end{proposition}
The composition $R'\circ R$ of two $\ca{VB}$-relations is called \emph{clean} if it is clean as a composition of submanifolds. It is then automatic that $R'\diamond R$ and $R'\circ R$ are smooth subbundles along
$S'\diamond S$ and $S'\circ S$, respectively, where $S',S$ are the base manifolds of $R',R$. Conversely, if the base manifolds compose cleanly, and the pointwise fibers of $R'\diamond R,\ R'\circ R$
have constant rank, then the subbundles compose cleanly. 
\begin{remark}\label{rem:grab}
Here it is convenient to work with the characterization of vector bundles and their morphisms in terms of scalar multiplication, due to Grabowski and Rotkiewicz \cite{Grabowski:2009dc}. Specifically, 
a smooth submanifold of a vector bundle is a vector subbundle if and only if it is invariant under scalar multiplication \cite[Theorem 2.3]{Grabowski:2009dc}, and a smooth map between vector bundles is a vector bundle homomorphism if and only if it intertwines scalar multiplication \cite[Theorem 2.4]{Grabowski:2009dc}.  
\end{remark}

The following proposition shows that the clean composition of $\ca{CA}$-relations is again a 
$\ca{CA}$-relation. There is a parallel statement for $\ca{LA}$-relations, with a similar proof. 
\begin{proposition}\label{prop:dirrel}
Suppose $\A_i\to M_i$ are Courant algebroids, and that 
$R\colon \A_0\da \A_1$ and $R'\colon \A_1\da \A_2$ are $\ca{VB}$-relations 
with clean composition. 
\begin{enumerate}
\item 
If $R,R'$ are involutive then so is $R'\circ R$.
\item 
If $R,R'$ are Lagrangian then so is $R'\circ R$.
\end{enumerate}
In particular, if $R,R'$ are Courant relations then so is $R'\circ R$. 
\end{proposition}
\begin{proof}
(a)
Let $p\colon M_2\times M_1\times M_1\times M_0\to M_2\times M_0$ be the projection, and 
let 
\[ Q\colon \A_2\times\ol{\A_1}\times \A_1\times\ol{\A_0}\da \A_2\times \A_0\] 
be the relation 
defined by $(\A_2)_\Delta\times (\A_1)_\Delta\times (\A_0)_\Delta$. Under this relation, 
$\ti{\sig}\sim_Q \sig$ if and only if the restriction of $\ti{\sig}-p^*\sig$ to 
$M_2\times (M_1)_\Delta\times M_0$ takes values in $0\times (\A_1)_\Delta\times 0$. 
Since $Q$ is involutive, we have 
\[ \ti{\sig}\sim_Q \sig,\ \ti{\tau}\sim_Q\tau\Rightarrow \Cour{\ti{\sig},\ti{\tau}}\sim_Q \Cour{\sig,\tau}.\]
Suppose $R,R'$ are involutive. 
Let $\sig$ be a section of $\A_2\times \ol{\A_0}$ whose restriction to $S'\circ S$ takes values 
in $R'\circ R$. Since $R'\diamond R\to R'\circ R$ is a surjective vector bundle homomorphism covering a 
submersion $S'\diamond S\to S'\circ S$, the restriction 
$\sig|_{S'\circ S}$ admits a lift to a section $\ti{\sig}_{S'\diamond S}$ of $ R'\diamond R$. 
By definition, $\ti{\sig}_{S'\diamond S}-p^*\sig\rvert_{S'\diamond S}$ takes values in 
$0\times (\A_1)_\Delta\times 0$. Since the bundles $R'\times R\to S'\times S$ and $\A_2\times (\A_1)_\Delta\times \A_0\to M_2\times (M_1)_\Delta\times M_0$ intersect cleanly, we may choose  
$\ti{\sig}\in\Gamma(\A_2\times\ol{\A_1}\times \A_1\times \ol{\A_0})$ such that
\begin{enumerate}
\item[(i)] $\ti{\sig}\rvert_{S'\diamond S}=\ti{\sig}_{S'\diamond S}$,
\item[(ii)]
$\ti{\sig}\rvert_{S'\times S}$ takes values in $R'\times R$, 
\item[(iii)] $(\ti{\sig}-p^*\sig)\rvert_{M_2\times (M_1)_\Delta\times M_0}$ takes values in 
$0\times(\A_1)_\Delta\times 0$, i.e.~ $\ti{\sig}\sim_Q \sig$. 
\end{enumerate}
Note that (iii) implies that $\ti{\sig}\rvert_{M_2\times (M_1)_\Delta\times M_0}$
takes values in $\A_2\times(\A_1)_\Delta\times\ol{\A_0}$.

Given another section $\tau$ of $\A_2\times \ol{\A_0}$ whose restriction $\tau|_{S'\circ S}$ takes values in $R'\circ R$, let $\ti{\tau}$ be constructed similarly. Since $R'\times R$ and 
$\A_2\times(\A_1)_\Delta\times\ol{\A_0}$ are involutive, 
the restriction of $\Cour{\ti{\sig},\ti{\tau}}$ to $S'\times S$ takes values in $R'\times R$, 
while the restriction to $M_2\times (M_1)_\Delta\times M_0$ takes values in $\A_2\times(\A_1)_\Delta\times\ol{\A_0}$. Hence $\Cour{\ti{\sig},\ti{\tau}}\rvert_{S'\diamond S}$ 
takes values in  $R'\diamond R$. Since $\Cour{\ti{\sig},\ti{\tau}}\sim_Q \Cour{\sig,\tau}$, this shows 
that $\Cour{\sig,\tau}\rvert_{S'\circ S}$ takes values in $R'\circ R$.

Part (b) follows from the well-known statement that the composition of Lagrangian relations of vector spaces is again Lagrangian (Lemma~\ref{lem:lagRel}).  

\end{proof}
A \emph{Courant morphism} \cite{LetToWein} is a Courant relation $R\colon \A_0\da \A_1$ such that the 
underlying relation $S\colon M_0\da M_1$ is the graph of a map $\Phi\colon M_0\to M_1$. 
(In contrast with vector bundle morphisms or Lie algebroid morphisms, one does 
not require that $R$ be a graph.) As a special case of Proposition \ref{prop:dirrel}, the composition of 
Courant morphisms is again a Courant morphism. 

\begin{example}
Any smooth map $\Phi\colon M_0\to M_1$ 
has a \emph{standard lift} to a Courant morphism $R_\Phi\colon \TM_0\da \TM_1$, given by 
\begin{equation}\label{eq:standardLift}
v_0+\alpha_0\sim_{R_\Phi} v_1+\alpha_1 \quad\Leftrightarrow\quad v_1=T\Phi(v_0),\text{ and }\alpha_0=T\Phi^*\alpha_1.
\end{equation}
%
More generally, suppose $\eta_i\in\Omega^3(M_i)$ are closed three forms, and $\omega\in\Omega^2(M_0)$ with $\eta_0=\Phi^*\eta_1+d\omega$. Then there is a Courant morphism $R_{\Phi,\omega}\colon (\TM_0)_{\eta_0}\da (\TM_1)_{\eta_1}$  given by \cite[Example~2.11]{Bursztyn:2009wi}
$$v_0+\alpha_0\sim_{R_{\Phi,\omega}} v_1+\alpha_1 \quad\Leftrightarrow\quad v_1=T\Phi(v_0),\text{ and }\alpha_0=T\Phi^*\alpha_1-\iota(v_0)\omega.$$
\end{example}

\subsection{Manin pairs}
A \emph{Manin pair} $(\A,E)$ is a Courant algebroid $\A\to M$ together with a Dirac structure $E\subseteq \A$. 
If $M=\pt$, this reduces to the usual notion of a Manin pair of Lie algebras. A \emph{morphism of Manin pairs} \cite{Bursztyn:2009wi}
\[ R\colon (\A,E)\da (\A',E'),\] 
with underlying map $\Phi\colon M\to M'$, is a morphism of Courant algebroids with the property that for all $m\in M$, any element of $E'_{\Phi(m)}$ is $R$-related to a unique element of $E_m$. Equivalently, in terms of composition of relations, 
\[ \Phi^* E'=R\circ E,\ \ \ \ \ker(R)\cap E=0.\]
One obtains a bundle map $\Phi^*E'\to E$, associating to each $x'\in E'_{\Phi(m)}$ 
the unique $x\in E_m$ to which it is $R$-related. This bundle map is a comorphism of Lie algebroids
\cite{Mackenzie05}, thus in particular the map $\Phi^*\colon \Gamma(E')\to \Gamma(E)$ preserves Lie brackets.  

\begin{example}\label{ex:canonicalmorphism}
For any Manin pair $(\A,E)$ over $M$, there is a morphism of Manin pairs 
\[ R\colon (\TM,TM)\da (\A,E)\]
where $v+\alpha\sim_R x$ if and only if $v=\a(x)$ and $x-\a^*(\alpha)\in E$.
\end{example}

\begin{example}
Suppose $M,M'$ are Poisson manifolds with bivector fields $\pi,\pi'$. Let $\Phi\colon M\to M'$ be a smooth map. Then the standard lift $R_\Phi\colon \T M\da \T M'$ (cf.~  \eqref{eq:standardLift}) defines a morphism of Manin pairs $R_\Phi\colon (\T M,\on{Gr}_\pi)\da (\T M',\on{Gr}_{\pi'})$ if and only if $\Phi$ is a Poisson map. 
\end{example}

\section{Dirac Lie groups}\label{sec:DLG}
The definition of Dirac Lie group structures (Definition~\ref{def:dirlie} below) requires that the ambient Courant algebroid  itself be multiplicative, in the sense that it carries the structure of a $\ca{CA}$-groupoid.
\subsection{$\ca{CA}$-groupoids}
For any groupoid $H\rra H^{(0)}$ and $k>0$ denote $H^{(k)}=\{(g_1,\ldots,g_k)\in H^k|\ s(g_i)=t(g_{i+1}),\ i=1,\ldots,k-1\}$. Let 
$\on{Mult}_H\colon H^{(2)}\to H,\ (X,Y)\to
 X\circ Y$ denote the groupoid multiplication, and 
 \[ \on{gr}(\on{Mult}_H)=\{(X\circ Y,X,Y)|\ \ (X,Y)\in H^{(2)}\}\subseteq H^3\] 
 its graph. 
\begin{definition} \label{def:CALAVBgroupoid}
Let $H\rra H^{(0)}$ be a Lie groupoid.
\begin{enumerate}
\item 
 A \emph{$\ca{VB}$-groupoid} over $H$ is a vector bundle $V\to H$, equipped with a 
 groupoid structure such that $\on{gr}(\on{Mult}_V)\subseteq V^3$ 
is a vector subbundle along $\on{gr}(\on{Mult}_H)$. 
\item An \emph{$\ca{LA}$-groupoid} over $H$ is a Lie algebroid $E\to H$, equipped with a 
 groupoid structure such that $\on{gr}(\on{Mult}_E)\subseteq E^3$ 
is a Lie subalgebroid along $\on{gr}(\on{Mult}_H)$.
\item A \emph{$\ca{CA}$-groupoid} over $H$ is a Courant algebroid $\A\to H$, equipped with a
groupoid structure such that
$\on{gr}(\on{Mult}_\A)\subseteq \A\times \ol{\A}\times \ol{\A}$  is a Dirac structure along  
$\on{gr}(\on{Mult}_H)$.
\end{enumerate}
\end{definition}
In other words, we require that the groupoid multiplication is a $\ca{VB}$-relation, $\ca{LA}$-relation or 
$\ca{CA}$-relation, respectively. It is common to indicate a $\mathcal{VB}$-groupoid $V$ by a diagram 
\[ \xymatrix{ {V} \ar[r]\ar@<-4pt>[r] \ar[d] &{V^{(0)}}\ar[d]\\
H \ar[r]\ar@<-4pt>[r] & H^{(0)} }\]  
\begin{remark}
\begin{enumerate}
\item The definition  of $\ca{VB}$-groupoids given above is somewhat shorter 
than Pradines' original definition \cite{Pradines:1988td}, which requires that  all the groupoid structure maps of $V$ are morphisms of vector bundles. The equivalence of the two definitions follows  
from Grabowski-Rotkiewicz's Remark~\ref{rem:grab}. For instance, since  $V^{(0)}\subseteq V$ is  a smooth submanifold invariant under scalar multiplication, it is a vector subbundle. Similarly, since $s_V,\,t_V\colon V\to V^{(0)}$ 
are smooth maps intertwining scalar multiplication they are vector bundle morphisms. 
\item $\ca{LA}$-groupoids are due to Mackenzie \cite{Mackenzie:2003wj,Mackenzie:2008tz}. The definition above implies that 
$V^{(0)}$ is a Lie subalgebroid along $H^{(0)}$, and that all the groupoid structure 
maps are morphisms of Lie algebroids.
\item The concept of a $\ca{CA}$-groupoid (also called Courant groupoid) was suggested by  Mehta
\cite[Example 3.8]{Mehta:2009js} and Ortiz \cite[Section 7.3]{Ortiz:2009ux}, and developed in detail in
\cite{LiBland:2010wi}.
\end{enumerate}
\end{remark}  
A \emph{relation of $\ca{CA}$-groupoids} 
$R\colon \A_0\da \A_1$ is a $\ca{CA}$-relation such that $R\subseteq \A_1\times\ol{\A_0}$ is a 
Lie subgroupoid. If the underlying relation $S\colon H_0\da H_1$ is the graph of a groupoid 
homomorphism, then $R$ is called a \emph{morphism of $\ca{CA}$-groupoids}. Relations and morphisms 
of $\ca{VB},\ca{LA}$-groupoids are defined similarly. 
\begin{proposition}\label{prop:units}\label{prop:alongunit}
Let $\A\to H$ be a $\ca{CA}$-groupoid.
Then the set of units $\A^{(0)}\subseteq \A$ is a Dirac structure along $H^{(0)}\subseteq H$.  
Furthermore, the groupoid inversion defines a morphism of Courant algebroids 
$\on{Inv}_\A\colon \A\da \ol{\A}$ over $\on{Inv}_H\colon H\to H$.
\end{proposition}
\begin{proof}
Define a relation $D\colon \ol{\A}\times \A\da \A$, where $(x_1,x_2)\sim_D x$ if and only if 
$x=x_1^{-1}\circ x_2$. Since 
\[ D=\{(x_1^{-1}\circ x_2,x_1,x_2)|\ t(x_1)=t(x_2)\}\subseteq \A\times\A \times \ol{\A}.\]
is obtained from $\on{gr}(\Mult_\A)$ by a cyclic permutation of components (and an overall sign change of the metric), it is a Dirac structure along the graph of the relation $H\times H\da H,\ (g_1,g_2)\sim g_1^{-1} g_2$. On the other hand, we may think of the diagonal in $\ol{A}\times \A $ as a Courant relation 
$\A_\Delta\colon 0\da \ol{\A}\times \A$, with underlying relation $H_\Delta:\pt\da H\times H$. Observe $\A^{(0)}=D\circ \A_\Delta$, where the composition is clean. Hence $\A^{(0)}$ is a Dirac structure along 
$H^{(0)}$. Similarly, the graph of the groupoid inversion 
$\on{gr}(\on{Inv}_\A)\subseteq \ol{\A}\times \ol{\A}$ is a clean composition of Courant relations
$\on{gr}(\on{Inv}_\A)=\A^{(0)}\circ \on{gr}(\on{Mult}_\A)$.
\end{proof}
Note that $D$ and $\A_\Delta$ are relations of Courant groupoids, if we take $\ol{\A}\times \A$ with the pair groupoid structure.

\subsection{Multiplicative Manin pairs and Dirac Lie group structures}
\begin{definition}\cite{LiBland:2010wi,Ortiz:2009ux,Bursztyn:2009wi,Mehta:2009js}
A \emph{multiplicative Manin pair}  is a Manin pair 
$(\A,E)$, where $\A\rra \A^{(0)}$ is a $\ca{CA}$-groupoid over $H\rra H^{(0)}$, and  
$E\rra E^{(0)}$ is a $\mathcal{VB}$-subgroupoid of $\A$. A \emph{morphism of 
multiplicative Manin pairs} $R\colon (\A_0,E_0)\da (\A_1,E_1)$ is a morphism of Manin pairs
which is also a morphism of $\ca{CA}$-groupoids $R\colon \A_0\da \A_1$.
\end{definition}
The involutivity condition implies that $E$ inherits the structure of an $\ca{LA}$-groupoid. 

As shown in Proposition~\ref{prop:alongunit}, for any $\ca{CA}$-groupoid structure $\A\rra \A^{(0)}$, 
the space $\A^{(0)}$ of units is a Dirac structure along $H^{(0)}$.  In this paper, we are mainly concerned with the case that $H^{(0)}=\pt$, 
such that $H$ is a group. In this case, the 
groupoid multiplication defines a Courant morphism 
$R=\on{gr}(\Mult_\A)$ 
covering the group multiplication,
\[ \xymatrix{ {\A\times \A} \ar@{-->}[r]^{\ \ \ R}\ar[d] & {\A}\ar[d] \\
{H\times H}\ar[r]_{\ \ \ \Mult_H} & H
}\]%
\begin{definition}\label{def:dirlie}
A \emph{Dirac Lie group structure} on a Lie group $H$ is a multiplicative Manin pair $(\A,E)$ 
over $H$ such that the multiplication morphism 
$\on{gr}(\Mult_\A)\colon (\A,E)\times (\A,E)\da (\A,E)$ is a morphism of Manin pairs.  

Given Dirac Lie group structures $(\A,E),(\A',E')$ on Lie groups $H,H'$, a morphism of multiplicative Manin pairs
$(\A,E)\da (\A',E')$ is called a \emph{morphism of Dirac Lie groups}.
\end{definition}

\begin{remark}
In other words, we define the category of Dirac Lie groups to be the subcategory of \emph{group like} objects in the category of Manin pairs. Meanwhile, the category of multiplicative Manin pairs is the subcategory of \emph{groupoid like} objects in the category of Manin pairs.
Our definition is more restrictive than that of Ortiz \cite{Ortiz:2009ux,Ortiz:2008bd} and Jotz \cite{Jotz:2009va}, where Dirac Lie group structures are taken to be arbitrary multiplicative Manin pairs over $H$. 
Note that \cite{Jotz:2009va,Ortiz:2008bd} only explore the case $\A=\TH$. In an earlier paper,  
Milburn \cite{Milburn:2007} gives a `categorical' definition of what he calls \emph{Dirac groups}, 
similar to the Ortiz-Jotz definition.  
\end{remark}

\begin{proposition}\label{prop:dirm}
A multiplicative Manin pair $(\A,E)$ defines a Dirac Lie group structure if and only if 
$E$ is a wide subgroupoid of $\A$, i.e~ $\A^{(0)}=E^{(0)}$. 
\end{proposition}
\begin{proof}
Suppose $(\A,E)$ is a multiplicative Manin pair over $H$, and that $\A^{(0)}=E^{(0)}$. 
We will show that the multiplication morphism is a morphism of Manin pairs.  

By Proposition~\ref{prop:units}, $\g=\A^{(0)}$ is Lagrangian, as is $E_e$. Since $E$ contains the units, it follows that $E_e=\g$. More generally, for any 
$h\in H$ the source and target map give isomorphisms $s_h,t_h\colon E_h\to \g$. 
Hence if $h_1,h_2\in H$ are given, then any $x\in E_{h_1h_2}$ can be uniquely written as a product 
$x=x_1\circ x_2$ with $x_i\in E_{h_i}$: $x_1$ is uniquely determined by $t(x_1)=t(x)$, and 
then $x_2=x_1^{-1}\circ x$. This shows that $\on{Mult}_\A$ gives a morphism of Manin pairs. 

Conversely, suppose $E^{(0)}$ is a proper subspace of $A^{(0)}$. Then $\dim E^{(0)}<\on{rank}(E)=\dim A^{(0)}$. 
In particular, $\ker(s|_E)$ is non-trivial. Since $$\{(x^{-1},x)\mid x\in \ker(s|_E)\}\subseteq \ker(\Mult_\A)\cap (E\times E),$$ 
this shows that $\Mult_\A$ does not define a morphism of Manin pairs.  
\end{proof}

\subsection{Examples}

\begin{example}
For any Courant algebroid $\A\to M$, the direct product $\ol{\A}\times \A\to M\times M$, with groupoid structure 
that of a pair groupoid, defines a $\ca{CA}$-groupoid structure over the pair groupoid 
$M\times M\rra M$: 
\[ \xymatrix{ {\ol{\A}\times \A} \ar@<2pt>[r]\ar@<-2pt>[r] \ar[d] &{\A}\ar[d]\\
M\times M \ar@<2pt>[r]\ar@<-2pt>[r] & M
 }\]
If $(\A,E)$ is a Manin pair, then $(\ol{\A}\times \A,E\times E)$ becomes a multiplicative Manin pair. 
If $M=\pt$, so that $\A=\g$ is a 
quadratic Lie algebra, the diagonal $\g_\Delta \subseteq \ol{\g}\oplus \g$ defines a 
Dirac Lie group structure on $H=\{e\}$. 
\end{example}

\begin{example}
The standard Courant algebroid over any Lie groupoid $H\rra H^{(0)}$ is a $\ca{CA}$-groupoid 
$\TH\rra TH^{(0)}\oplus A^*H$, where $AH\to H^{(0)}$ is the Lie algebroid of $H$, and $A^*H$ its dual. 
The $\ca{VB}$-groupoid structure is given as the direct sum of 
the tangent prolongation $TH\rra TH^{(0)}$ and the cotangent groupoid $T^*H\rra A^*H$.
See \cite[Example 3.8]{Mehta:2009js} and  \cite[Example 9]{LiBland:2010wi}. Both 
$(\TH,T^*H)$ and $(\TH,TH)$ are multiplicative Manin pairs. 

In particular, if  $H$ is a Lie group, the $\mathcal{VB}$-groupoid structure on $\TH$ is the 
direct product of the group $TH\rra \pt$ with the symplectic groupoid $T^*H\rra \h^*$:
\[ \xymatrix{ {\TH} \ar@<2pt>[r]\ar@<-2pt>[r] \ar[d] &{\h^*}\ar[d]\\
H  \ar@<2pt>[r] \ar@<-2pt>[r] & \pt
 }\]
If $(\TH,E)$ is a Dirac Lie group structure, then $E\cap TH=0$ since the source and target maps 
$E\to \h^*$ are surjective. Thus $E$ is the graph of a bivector field $\pi\in\Gamma(\wedge^2 TH)$. 
The condition that $E$ is a subgroupoid translates into the condition that $\pi$ is multiplicative, i.e.~ 
a Poisson-Lie group structure. In fact the following was obtained by Ortiz 
\cite{Ortiz:2008bd} and Jotz \cite{Jotz:2009va}, as part of a general classification of multiplicative Manin pairs for $\A=\TH$: 
\begin{proposition}
The Dirac Lie group structures for the standard Courant algebroid  
over a Lie group $H$ are exactly those of the form $(\TH,\on{Gr}_\pi)$ where 
$\pi$ defines a Poisson-Lie group structure on $H$. If $(H,\pi),(H',\pi')$ 
are Poisson Lie groups and $\Phi\colon H\to H'$ is a Lie group homomorphism, 
then the standard lift of $\Phi$ is a Dirac Lie group morphism if and only if $\Phi$ is a 
Poisson Lie group morphism, i.e.~ $\pi\sim_\Phi \pi'$.  
\end{proposition}
As a special case, any Lie group carries a `trivial' Dirac Lie group structure $(\TH,T^*H)$. 
The Manin pair $(\TH,TH)$ is multiplicative, but is not a Dirac Lie group structure
in our sense since $TH$ is not a wide subgroupoid. 
\end{example}

\begin{example}\label{ex:canonicalmorphism1}
For any multiplicative Manin pair $(\A,E)$, the morphism $(\TH,TH)\da (\A,E)$ (cf.~ Example~\ref{ex:canonicalmorphism}) is 
a morphism of multiplicative Manin pairs.  
\end{example}

\begin{example}\label{ex:cardir}
\cite[$\mathsection$ 3.4]{Alekseev:2009tg}
Let $G$ be a Lie group whose  Lie algebra $\g$ carries an invariant metric $B$. 
Then there is a $\ca{CA}$-groupoid structure on $G$, 
 \[ \xymatrix{ G\times (\ol{\g}\oplus \g) \ar@<+4pt>[r]\ar@<0pt>[r] \ar[d] &{\g}\ar[d]\\
G  \ar@<2pt>[r]\ar@<-2pt>[r] & \pt
 }\]
Here the $\mathcal{VB}$-groupoid structure is the direct product of the group $G\rra \pt$ with the 
pair groupoid $\ol{\g}\oplus \g\rra \g$. As a Courant algebroid, $\A=G\times (\ol{\g}\oplus \g)$
is the action Courant algebroid for the following action of $\ol{\g}\oplus \g$ on $G$
\begin{equation*}\varrho(\zeta_1,\zeta_2)=\zeta_2^L-\zeta_1^R\end{equation*}
where $\zeta^L,\zeta^R$ are the left-,right-invariant vector fields defined by $\zeta\in\g$.
Since the action $\varrho$ is transitive, the Courant algebroid $\A$ is exact. In fact there is an explicit 
isomorphism of Courant algebroids $\kappa\colon G\times (\ol{\g}\oplus \g)\to \T G_\eta$, where $\eta=\f{1}{12}B(\theta^L,[\theta^L,\theta^L]) \in\Om^3(G)$
is the Cartan 3-form:
\begin{equation}\label{eq:kappa} 
\kappa(\zeta_1,\zeta_2)=\Big(\zeta_2^L-\zeta_1^R,\ \hh B(\theta^L,\zeta_2)+\hh B(\theta^R,\zeta_1)\Big).
\end{equation}
Here, $\theta^L,\ \theta^R\in\Omega^1(G,\g)$ are the left-,right-invariant Maurer-Cartan forms, defined by the property $$\iota(\zeta^L)\theta^L=\zeta=\iota(\zeta^R)\theta^R,\ \ \zeta\in\g.$$

The subbundle $E=G\times \g_\Delta$ defines a Dirac Lie group
structure on $G$. This is the \emph{Cartan-Dirac structure} on $G$,
found independently by Alekseev, \v{S}evera and Strobl. Its multiplicative properties were noted in \cite{Alekseev:2009tg}.  
\end{example}

\subsection{Constructions with $\ca{CA}$-groupoids}
In this section we will collect some further properties and constructions 
for $\ca{CA}$-groupoids. While we are mainly interested in the case $H^{(0)}=\pt$,  
the general proofs are more conceptual and in any case not harder. 
\subsubsection{Basic properties}
Given a Lie groupoid $H\rra H^{(0)}$, let  $TH\rra TH^{(0)}$ be its  tangent prolongation  
and $T^*H\rra A^*H$ the cotangent groupoid.
\begin{proposition}\label{prop:aastar} 
For any  $\ca{CA}$-groupoid 
$\A\rra \A^{(0)}$ over  $H\rra H^{(0)}$, the anchor map defines a morphism of 
$\mathcal{VB}$-groupoids $\a\colon \A\to TH$, while $\a^*$ defines a morphism of $\mathcal{VB}$-groupoids 
$\a^*\colon T^*H\to \A$. 
\end{proposition}
\begin{proof}
By definition of a $\ca{CA}$-groupoid, the image of $\on{gr}(\Mult_\A)$ under the anchor map lies in $T \on{gr}(\Mult_H)=\on{gr}(\Mult_{TH})$. Hence, the graph of $\a$ is a $\ca{VB}$-subgroupoid of 
$TH\times \A$, proving that $\a$ is a
$\mathcal{VB}$-groupoid homomorphism. Dualizing, $\a^*\colon T^*H\to \A^*$ is a 
$\mathcal{VB}$-groupoid homomorphism. But the isomorphism $\A^*\cong \A$ given by the metric is 
an isomorphism of $\mathcal{VB}$-groupoids. 
\end{proof}

\begin{corollary}\label{cor:diag}
For any $\ca{CA}$-groupoid $\A$ over $H$, the diagonal morphism $\TH\da \ol{\A}\times \A$ 
given by 
\[ v+\alpha\sim (x,y)\Leftrightarrow v=\a(x),\ \ y-x=\a^*(\alpha)\] is a morphism of $\ca{CA}$-groupoids. 
\end{corollary}
\begin{proof}
As shown in \cite[Proposition 1.6]{LiBland:2009ul}, the diagonal morphism is a morphism of Courant algebroids. 
By Proposition~\ref{prop:aastar}, it is also a morphism of $\mathcal{VB}$-groupoids. 
\end{proof}

\subsubsection{Reduction and Pull-backs}
\begin{proposition}[Coisotropic reduction]\label{prop:reduction}
Let $\A\rra \A^{(0)}$ be a $\ca{CA}$-groupoid over $H\rra H^{(0)}$, 
and let $C\subseteq \A$ be a $\mathcal{VB}$-subgroupoid 
along a subgroupoid $K\subseteq H$.  Assume that 
\begin{enumerate}
\item $C$ is co-isotropic, 
\item $C$ is involutive, 
\item $\a(C)\subseteq TK,\ \a(C^\perp)=0$. 
\end{enumerate}
Then 
the quotient $\A_C=C/C^\perp$ defines a $\ca{CA}$-groupoid structure on $K$, 
in such a way that the inclusion map $K\to H$ lifts to a morphism of $\ca{CA}$-groupoids, 
$C/C^\perp\da \A$. If $E\subseteq \A$ defines a multiplicative Manin pair $(\A,E)$, and
$E$ is transverse to $C$ then 
\[ (\A_C,E_C)=(C/C^\perp,\ (E\cap C)/(E\cap C^\perp))\] 
is again a multiplicative Manin pair.
\end{proposition}
Here transversality means $E|_K+C=\A|_K$, or equivalently $E\cap C^\perp=0$. 
\begin{proof}
Since $C$ is a co-isotropic $\mathcal{VB}$-subgroupoid of $\A$, $C^\perp$ is 
a $\mathcal{VB}$-subgroupoid of $C$, and $\A_C=C/C^\perp$ inherits a 
$\mathcal{VB}$-groupoid structure (see Corollary~\ref{cor:david} from Appendix~\ref{app:fiberproducts} for details). 
By \cite[Proposition 2.1]{LiBland:2009ul}, the Courant bracket on $\A$ descends to a Courant bracket 
on the quotient $\A_C$, in such a way that 
\begin{equation}\label{eq:s}
 S=\{(x,[x])|\ x\in C\}\subseteq \A\times \ol{\A}_C\end{equation}
is a Courant morphism $S\colon\ \A_C\da \A$. Here $[x]\in
\mathbb{B}$ denotes the image of $x\in C$. The graph of the groupoid multiplication 
of $\A_C$ is a transverse composition of Courant relations, 
\[ \on{gr}(\Mult_{\A_C})=\on{gr}(\Mult_\A)\circ (S\times S\times S),\]
hence it is itself a Courant relation. 
Thus $\A_C$ carries a $\ca{CA}$-groupoid structure. Since $S$ is a 
Dirac structure along the graph of the inclusion, and also a subgroupoid, 
it defines a $\ca{CA}$-groupoid morphism. 

If $E\subseteq \A$ is a multiplicative Dirac structure transverse to $C$, then $E_C=E\circ S$ 
is a transverse composition, and is a multiplicative Dirac structure in $\A_C$. 
\end{proof}

\begin{proposition}[Pull-backs]
Let $\A\rra \A^{(0)}$ be a $\ca{CA}$-groupoid over $H\rra H^{(0)}$, and $\Phi\colon K\to H$ 
a homomorphism of Lie groupoids. Suppose that $\Phi$ is transverse to the anchor map $\a\colon \A\to TH$. Then the pull-back Courant algebroid $\Phi^!\A\rightrightarrows 
\Phi^* \A^{(0)}$ inherits the structure 
of a $\ca{CA}$-groupoid over $K\rra K^{(0)}$.
\end{proposition}
\begin{proof}
By definition (see \cite[Proposition 2.7]{LiBland:2009ul}), the pull-back Courant algebroid is a reduction
$\Phi^!\A=(\A\times \T K)_C$ relative to the coisotropic subbundle $C$ along 
$\on{gr}(\Phi)\cong K$, 
\[ C=\A\times_{TH} \T K \subseteq \A\times \T K,\]
the fiber product relative to the maps $\a_\A\colon \A\to TH$ and $d\Phi\circ \a_{\T K}\colon \T K\to TH$. 
Proposition~\ref{prop:moerdijk} shows that $C$ is a Lie groupoid.  Its
space of units $C^{(0)}=\A^{(0)}\times_{TH^{(0)}} A^*K$ is a smooth
subbundle of $\A^{(0)}\times A^*K$ along
$\on{gr}(\Phi|_{K^{(0)}})\cong K^{(0)}$.  Corollary~\ref{cor:david}
from Appendix~\ref{app:fiberproducts} shows that $C^\perp\subseteq C$
is a subgroupoid. Hence $C/C^\perp$ inherits a $\ca{CA}$-groupoid
structure.
\end{proof}

\section{Classification of Dirac Lie group structures}\label{sec:classDLG}

In this Section we will give the general classification and construction of 
Dirac Lie group structures over Lie groups $H$. The classification will be given in terms of \emph{$H$-equivariant} Dirac Manin triples $(\dd,\g,\h)_\beta$. 
\subsection{Vacant $\ca{LA}$-groupoids}\label{subsec:vacuum}
Following Mackenzie \cite{Mackenzie:2003wj}, a $\ca{VB}$-groupoid $V\to H$ will be called \emph{vacant}  if
it has the property $V^{(0)}=V|_{H^{(0)}}$.
\begin{lemma}
For any Dirac Lie group structure $(\A,E)$ over a group $H$, the sub-bundle 
$E$ is a vacant $\ca{LA}$-groupoid.
\end{lemma}
\begin{proof}
The Lie algebroid bracket is induced from the Courant bracket on $\A$. Since 
$E^{(0)}\cong \A^{(0)}$ is a Lagrangian subspace of $\A_e$, it must must coincide with $E_e$. 
\end{proof}
As shown by Mackenzie \cite{Mackenzie:2003wj}, vacant
$\ca{LA}$-groupoids over groups are characterized in terms of
Lie-theoretic data. We will review his theory from a mildly different
perspective; further details are given in Appendix~\ref{app:mat}.
\begin{definition}\label{def:triple}
Let $H$ be a Lie group with Lie algebra $\h$. A \emph{Lie algebra triple} 
$(\dd,\g,\h)$ is a Lie algebra $\dd$ with a vector space decomposition $\dd=\g\oplus \h$ into two 
Lie subalgebras $\g,\h$. Given an  action of $H$ on $\dd$ by automorphisms, which integrates the adjoint action of $\h\subseteq \dd$ and restricts to the adjoint 
action of $H$ on $\h$, we refer to $(\dd,\g,\h)$ as an \emph{$H$-equivariant Lie algebra triple}.  
\end{definition}
We will simply write $h\mapsto \Ad_h$ for the action of $H$ on $\dd$. Part (b) of the following 
Proposition associates a vacant $\ca{LA}$-groupoid $E\to H$ to any $H$-equivariant Lie algebra triple $(\dd,\g,\h)$. It is realized as a $\ca{LA}$-subgroupoid of the direct product of $TH\rra 0$ with the 
pair groupoid $\g\oplus \g\rra \g$. 
\begin{proposition}\label{prop:dress}
Let $(\dd,\g,\h)$ be an $H$-equivariant Lie algebra triple.
\begin{enumerate}
\item 
The subset 
\[ V=\{(v,X,X')|\ v\in T_hH,\ X,X'\in\dd,\ \Ad_hX'-X=\iota(v)\theta^R_h\},\]
 is an $\ca{LA}$-subgroupoid of $TH\times (\dd\oplus\dd)\rra \dd$, of rank equal to $\dim\g+2\dim\h$.
Its object space is $V^{(0)}=\dd$.
\item
The subset  
\[ E=\{(v,\xi,\xi')|\ v\in T_hH,\ \xi,\xi'\in\g,\ \Ad_h\xi'-\xi=\iota(v)\theta^R_h\}\]
is a vacant $\ca{LA}$-subgroupoid of $TH\times (\g\oplus\g)\rra \g$, of rank equal to $\dim\g$. 
Its object space is $E^{(0)}=\g$. 
The source map 
$(v,\xi,\xi')\mapsto \xi'$ is a trivialization of $E$, and defines a morphism of Lie algebroids $E\to \g$.
\end{enumerate}
\end{proposition}
\begin{proof}
(a) 
The $\ca{VB}$-groupoid $TH\times(\dd\oplus \dd) \rra \dd$ may be regarded as a direct sum of two 
$\ca{VB}$-subgroupoids $TH\rra 0$ and $H\times (\dd\oplus \dd)\rra \dd$. 
Right trivialization $TH\cong \h\rtimes H$ gives a fiberwise injective group isomorphism 
\begin{equation}\label{eq:map1}
 TH\to \dd\rtimes H,\ \ v\mapsto (\iota_v\theta^R_h,h)
\end{equation}
where $h$ is the base point of $v$, and the semi-direct product is relative to $\Ad$.
On the other hand, the map  
\begin{equation}\label{eq:map2} 
H\times(\dd\oplus \dd)\to \dd\rtimes H,\ \ (h,X,X')\mapsto \Ad_h(X')-X
\end{equation}
is a fiberwise surjective $\ca{VB}$-groupoid homomorphism. 
The fibered product of the two maps \eqref{eq:map1}, \eqref{eq:map2} is equal to $V$, 
which is hence a $\ca{VB}$-subgroupoid
of rank $\dim\h+2\dim\dd-\dim\dd=\dim\g+2\dim\h$. 
 
Let $H\times H$ act on $H$ by $(h_1,h_2).h=h_1hh_2^{-1}$, on $TH$ by the tangent lift of 
this action, and on  $\dd\oplus \dd$ by $(h_1,h_2).(X,X')=(\Ad_{h_1}X,\Ad_{h_2}X')$. 
We obtain a diagonal action on $TH\times (\dd\oplus \dd)$ by $\ca{LA}$-groupoid 
automorphisms. The maps \eqref{eq:map1}, \eqref{eq:map2} are equivariant relative 
to the action $(h_1,h_2).(Y,h)=(\Ad_{h_1}Y,h_1hh_2^{-1})$ on $\dd\rtimes H$, 
hence $V$ is $H\times H$-invariant. To verify that $V$ is a $\ca{LA}$-subgroupoid, it 
hence suffices to check near the group unit. In particular, we may assume that $H$ is connected and simply connected. Let $D$ be a connected Lie group with Lie algebra $\dd$, and with the action of $\dd\oplus\dd$
by $(X,X')\mapsto X'^L-X^R$. The corresponding action Lie algebroid embeds as a Lie subalgebroid 
\begin{equation}\label{eq:big}
\{(v,X,X')|\  v=X'^L|_d-X^R|_d\}\subseteq TD\times(\dd\oplus\dd)\end{equation}
(where $d\in D$ is the base point of $v\in TD$). On a neighborhood of $e\in H$, the group homomorphism 
$H\to D$ exponentiating $\h\to \dd$ is an embedding, and $V$ is simply the intersection of \eqref{eq:big} 
with 
$TH\times (\dd\oplus\dd)$. In particular, it is a Lie subalgebroid of $TH\times (\dd\oplus\dd)$.

%

(b) The same argument as for $V$ shows that $E$ is a subbundle of rank $\dim\g$. Since $E$ is the intersection 
of $V$ with the $\ca{LA}$-subgroupoid $TH\times(\g\oplus\g)$, it is itself an $\ca{LA}$-subgroupoid. 
Since $E$ has trivial intersection with the subbundle 
of elements of the form $(v,\xi,0)$, the source map $TH\times(\g\oplus\g)\to \g,\ (v,\xi,\xi')\mapsto \xi'$ 
defines a trivialization of $E$. Furthermore, since this projection is a Lie algebroid homomorphism, the 
same is true for its restriction to $E$.
\end{proof} 

\begin{proposition}\label{prop:VLAg1}
There is a 1-1 correspondence between 
\begin{itemize}
\item[(i)] Vacant $\ca{LA}$-groupoids $E\rra \g$ over groups $H\rra \pt$, and
\item[(ii)] $H$-equivariant Lie algebra triples $(\dd,\g,\h)$.
\end{itemize}
\end{proposition}
The proof of Proposition \ref{prop:VLAg1} is found in Appendix~\ref{app:mat}, but we summarize the construction here. The direction $(ii)\Rightarrow (i)$ is part (b) of Proposition~\ref{prop:dress}. 
In the opposite direction $(i)\Rightarrow (ii)$, 
let $\h$ be the Lie algebra of $H$, and put $\dd=\g\oplus\h$ as a vector space. 
One finds that $\dd$ carries a unique action $\Ad$ of $H$, extending the adjoint action on $\h\subseteq\dd$ and such that 
\begin{equation}\label{eq:eqE}\iota(\a(z))\theta^R_h=\Ad_h s(z)-t(z)\end{equation} for all 
$z\in E_h$. Furthermore, $\dd$ carries a unique Lie bracket such that $\g,\h$ are Lie subalgebras and 
such that the differential of $\Ad\colon H\to \in \on{Aut}(\dd)$ gives the adjoint action $\ad\colon 
\h\to \on{End}(\dd)$.

\subsection{Dirac Manin triples}\label{subsec:genman}
If $V$ is a vector space with an element $\beta\in S^2V$, we denote by 
$\beta^\sharp\colon V^*\to V$ the map $\beta^\sharp(\mu)=\beta(\mu,\cdot)$. 
A subspace $U\subseteq V$ is called $\beta$-coisotropic if $\beta^\sharp(\on{ann}(U))\subseteq U$. 
\begin{definition}
A  \emph{Dirac Manin triple} $(\dd,\g,\h)_\beta$ is a triple 
$(\dd,\g,\h)$ of Lie algebras, together with an element 
$\beta\in (S^2\dd)^\dd$ such that $\g$ is $\beta$-coisotropic.

If $(\dd,\g,\h)$ is an $H$-equivariant triple, and $\beta$ is $H$-invariant, 
we call $(\dd,\g,\h)_\beta$ an \emph{$H$-equivariant Dirac Manin triple}. 
\end{definition}
If $H$ is simply connected, then the $H$-equivariance conditions are automatic.
If $\beta$ is non-degenerate and $\g,\h$ are both Lagrangian Lie subalgebras, 
the Dirac Manin triple is an ordinary Manin triple. 
 
We will now associate to any Dirac Manin triple 
$(\dd,\g,\h)_\beta$ a new Dirac Manin triple 
$(\q,\g,\mf{r})_\gamma$, where $\gamma$ is non-degenerate and $\g$ is Lagrangian in $\q$. 
Let $\dd^*_\beta$ be the Lie algebra, equal to $\dd^*$ as a vector space, 
with the Lie bracket 
\[ \l [\mu_1,\mu_2],\xi\r=\l\mu_2,[\xi,\beta^\sharp(\mu_1)]\r,\ \ \mu_1,\mu_2\in\dd^*_\beta,\ \xi\in\g.\]
The element $\beta$, viewed as a bilinear form on $\dd^*_\beta$, is invariant under the bracket. 
The co-adjoint action of $\dd$ is by derivations of the bracket, hence we may form the semi-direct product
\[ \wh{\dd}=\dd\ltimes\dd^*_\beta.\]
The bilinear form 
\[ \wh{\beta}((\xi_1,\mu_1),(\xi_2,\mu_2))=\beta(\mu_1,\mu_2)+\l \mu_1,\xi_2\r+\l\mu_2,\xi_1\r\]
on $\wh{\dd}$ is invariant and non-degenerate. Note that $\dd\subseteq \dd\ltimes\dd^*_\beta$
is a Lagrangian Lie subalgebra, and $\dd^*_\beta$ is a Lie algebra ideal. This defines a new 
Dirac Manin triple $(\dd\ltimes\dd^*_\beta,\dd,\dd^*_\beta)_{\wh{\beta}}$. 
\begin{remark}\label{rem:CAgrpOvPt}
As observed by Drinfel'd \cite{Drinfeld:1989tu}, there is in fact a 1-1 correspondence between (i) 
Manin pairs $(\wh{\dd},\dd)$ with a Lie algebra ideal complementary to $\dd$, and (ii)
Lie algebras $\dd$ with invariant elements $\beta\in S^2\dd$.  One may interpret this as a classification of $\ca{CA}$-groupoids over $H=\pt$. Here $\wh{\dd}\equiv \dd\ltimes\dd^*_\beta\rra \dd$ is the action Lie groupoid for the translation action of $\dd^*_\beta$ on $\dd$ via the map $\beta^\sharp\colon \dd^*_\beta\to\dd$.
\end{remark}
The Lie subalgebra $\mf{c}=\g\ltimes\dd^*_\beta$ is coisotropic, since it contains 
the Lagrangian Lie subalgebra $\g\ltimes\on{ann}(\g)$. Hence $\mf{c}^\perp$ is an ideal in 
$\mf{c}$, and the quotient
\[\q=\mf{c}/\mf{c}^\perp\]  
is a Lie algebra with a non-degenerate invariant metric induced from that on $\dd\ltimes\dd^*_\beta$. 
Let $\gamma\in (S^2\q)^\q$ be given by the dual metric on $\q^*$. 
The inclusion $\g\hra \dd\ltimes\dd^*_\beta$ descends to an inclusion $\g\hra \q$ as a 
Lagrangian Lie subalgebra, thus $(\q,\g)$ is a Manin pair. 

Since $\dd^*_\beta$ is an ideal complementary to $\dd$, the same is true of $(\dd^*_\beta)^\perp$. 
Let $\wh{f}\colon \dd\ltimes\dd^*_\beta\to \dd$ be the projection with kernel $(\dd^*_\beta)^\perp$. 
Explicitly, $\wh{f}(\xi,\mu)=\xi+\beta^\sharp(\mu)$. 
This is a Lie algebra homomorphism, and since $\mf{c}^\perp\subseteq (\dd^*_\beta)^\perp$, it descends 
to a Lie algebra homomorphism 
\[ f\colon\q\to \dd,\]
with the important properties $f(\xi)=\xi$ for $\xi\in\g$ and $\beta^\sharp=f\circ f^*.$ 

Finally, $\mf{r}=f^{-1}(\h)$ is a Lie algebra complement to $\g$. We have thus defined a Dirac Manin triple 
\[ (\q,\g,\mf{r})_\gamma,\]
where $\gamma$ is non-degenerate and $\g$ is Lagrangian. We denote by $p_{\mf{r}}\in\on{End}(\q)$ the projection to $\mf{r}$ along $\g$ and by $p_\h\in\on{End}(\dd)$ the projection to $\h$ along $\g$; thus $f\circ p_{\mf{r}}=p_\h\circ f$.

\begin{examples} \label{ex:specialcases}
We describe the triple $(\q,\g,\mf{r})_\gamma$ associated to $(\dd,\g,\h)_\beta$ in some extreme cases. 
\begin{enumerate}
\item[(i)] If $\beta=0$ one obtains (independent of $\h$)
\[ (\q,\g,\mf{r})_\gamma=(\g\ltimes\g^*,\g,\g^*)_\gamma\]
with $\gamma$ the bilinear form given by the pairing.  The map $f\colon \q\to \dd$ is projection to 
$\g\subseteq \dd$. 
\item[(ii)] If $\beta$ is non-degenerate, defining a non-degenerate metric on $\dd$, one finds 
\[ (\q,\g,\mf{r})_\gamma=(\dd\oplus \ol{\g/\g^\perp},\g_\Delta,\h\oplus 0)_\gamma\]
where $\g/\g^\perp$ is the quotient Lie algebra with metric induced from that on $\dd$, and 
$\ol{\g/\g^\perp}$ is the same Lie algebra with the opposite metric. $\g_\Delta$ is embedded 
`diagonally' as $\xi\mapsto (\xi,[\xi])$ (where $[\xi]$ is the image in $\g/\g^\perp$), 
and the homomorphism $f$ is projection to the first summand.  
\item[(iii)]In particular, if $\beta$ is non-degenerate and 
$\g$ is \emph{Lagrangian}, we obtain $(\q,\g,\mf{r})_\gamma=(\dd,\g,\h)_\beta$, with $f$ the identity map. 
\end{enumerate}
\end{examples}

\subsection{From Dirac Manin triples to Dirac Lie group structures}\label{subsec:construct}
Let $(\dd,\g,\h)_\beta$ be an $H$-equivariant Dirac Manin triple, and let $(\q,\g,\mf{r})_\gamma$ 
and $f\colon \q\to \g$ be as in Section \ref{subsec:genman}. We will obtain a Dirac Lie group structure $(\A,E)$ on $H$ by reduction from 
the direct product of the multiplicative Manin pairs 
\[ (\TH,TH)\times (\ol{\q}\oplus \q,\g\oplus\g),\]
where $\TH\rra \h^*$ is the standard $\ca{CA}$-groupoid structure, and 
$\ol{\q}\oplus \q\rra\q$ is the pair groupoid. 
\begin{proposition}\label{prop:ciscois}
The subset $C\subseteq \TH\times (\ol{\q}\oplus \q)$ given as
\begin{equation}\label{eq:defc}
 C=\{(v+\alpha,\zeta,\zeta')|\ \Ad_h
 f(\zeta')-f(\zeta)=\iota(v)\theta^R_h\}\end{equation}
(where $h\in H$ is the base point of $v+\alpha\in \TH$) 
is a coisotropic, involutive $\mathcal{VB}$-subgroupoid, 
with $\a(C^\perp)=0$. 
The reduction of $(\TH,TH)\times (\ol{\q}\oplus \q,\g\oplus\g)$ relative to $C$ is a Dirac Lie group structure $(\A,E)$.
\end{proposition}
\begin{proof}
An argument similar to that given in the proof of Proposition~\ref{prop:dress} shows that 
$C$ is a $\ca{VB}$-subgroupoid of rank $\dim \q+\dim\dd$. Furthermore, 
$C$ is the pre-image of the $\ca{LA}$-subgroupoid $V$ from Proposition~\ref{prop:dress} under the $\ca{VB}$-groupoid homomorphism 
\begin{equation}\label{eq:grhom}
 \TH\times (\ol{\q}\oplus \q)\to TH\times (\dd\oplus\dd),\ (v+\alpha,\zeta,\zeta')\mapsto 
(v,f(\zeta),f(\zeta')).\end{equation}
Since \eqref{eq:grhom} preserves brackets, and since $V$ is a Lie subalgebroid, 
it follows that $C$ is involutive. 
The orthogonal bundle $C^\perp$ has rank equal to $\dim\dd$,
and is spanned by the sections
\[ \psi(\mu)=\Big(-\big\l \mu,\theta^R\big\r,\,f^*(\mu),\,
f^*(\Ad_{h^{-1}}\mu)\Big),\ \ \mu\in \dd^*.\]
Indeed, the pairing with $(v+\alpha,\zeta,\zeta')\in\Gamma(C)$ is
$\l\mu,-\iota(v)\theta^R_h-f(\zeta)+\Ad_{h}  f(\zeta')\r=0$
as required. The property $C^\perp\subseteq C$ follows by checking the definition of $C$ 
on the sections $\psi(\mu)$, 
\[ \Ad_h \big(f(f^*(\Ad_{h^{-1}}\mu))\big)-f(f^*(\mu))=0,\]
using the $H$-equivariance of $f\circ f^*=\beta^\sharp$.  The object space of 
$\TH\times(\ol{\q}\oplus \q)$ is $\h^*\times\q$, embedded as the space of units 
$T^*_eH\times\q_\Delta$. This is contained in $C$, hence
$C^{(0)}=\h^*\times\q$. On the other hand, 
$(C^\perp)^{(0)}\cong \dd^*$, embedded in $C^{(0)}$ by the  
map $\dd^*\to \h^*\times \q,\ \mu\mapsto (p_\h^*(\mu),f^*(\mu))$.  
We next show that $TH\times (\g\oplus \g)$ is transverse to $C$, or equivalently that $TH\times (\g\oplus \g)\cap C^\perp$ is trivial. Indeed, vanishing of 
the $T^*H$-component of $\psi(\mu)$ is equivalent to $\mu\in \on{ann}(\h)$, 
but then the last two components are contained in 
$f^*(\on{ann}(\h))=\mf{r}^\perp$. 
Coisotropic reduction by $C$ 
(cf. Proposition~\ref{prop:reduction}) gives the 
multiplicative Manin pair 
$(\A,E)=((\TH\times(\ol{\q}\oplus \q))_C,\ (TH\times (\g\oplus \g))_C)$. 
We have $\A^{(0)}=C^{(0)}/(C^\perp)^{(0)}=(\h^*\times\q)/ \dd^*\cong\g$
(the last identification is obtained by taking $0\times\g\hra \h^*\times\q$ as a complement to $\dd^*$),  and also $E^{(0)}=\g$ since  
\[ (TH\times (\g\oplus \g))^{(0)}=\{0\}\times\g\subseteq \h^*\times\q.\]
Since $\A^{(0)}=E^{(0)}$, it follows that $(\A,E)$ is a Dirac Lie group structure on $H$. 
\end{proof}

The construction of $\A$ by coisotropic reduction defines a $\ca{CA}$-groupoid morphism 
\begin{equation}\label{eq:S}
 S\colon \A\da \TH\times (\ol{\q}\oplus \q).
\end{equation}

\subsection{From Dirac Lie group structures to Dirac Manin triples}
%
In this Section we will show that any Dirac Lie group structure $(\A,E)$ on $H$
arises by the reduction procedure from the last section, from a unique $H$-equivariant 
Dirac Manin triple $(\dd,\g,\h)_\beta$. 
\subsubsection{Definition of $(\dd,\g,\h)_\beta$}
As remarked in Section~\ref{subsec:vacuum}, the Dirac structure $E$ is a vacant $\ca{LA}$-groupoid over $H$. Hence it corresponds to a unique $H$-equivariant Lie algebra triple $(\dd,\g,\h)$. 
%
%
%
Let $\q:=\A_e$, and let $f\colon \q\to \dd$ be the linear map given as the sum of the target and anchor map 
at the group unit $e\in H$,
\begin{equation}\label{eq:deff} f(\zeta)=t_e(\zeta)+\a_e(\zeta),\ \ \ \zeta\in\mf{q}. \end{equation}
Let $\gamma \in S^2\q$ be dual 
to the given metric on $\A_e$. Write $\q=\g\oplus \mf{r}$, with $\mf{r}$ be the kernel of $t_e\colon \A_e\to \g$, 
and $\g$ embedded as $E_e$. Thus $f(\tau)=\a_e(\tau)$ for $\tau\in\mf{r}$ and $f(\xi)=\xi$ for $\xi\in\g$.

Define $\beta\in S^2\dd$ by
\[ \beta^\sharp=f\circ f^*\colon \dd^*\to \dd.\]
This defines $(\dd,\g,\h)_\beta$, but we will need to show that $\beta$ is $H$-invariant and that 
this triple gives $(\A,E)$. We will also show that $(\q,\g,\mf{r})_\gamma$ is the triple associated 
to $(\dd,\g,\h)_\beta$. (Among other things, we will have to show that $\q$ is a Lie algebra and that $f$ is a Lie algebra homomorphism.) 
As before, we denote by $p_{\mf{r}}\in\End(\q)$ the projection to 
$\mf{r}$ along $\g$ and by $p_\h\in\End(\dd)$ the projection to 
$\h$ along $\g$. Thus 
$t_e=1-p_{\mf{r}}$ and  $p_\h\circ f=f\circ p_{\mf{r}}$.
\subsubsection{Trivialization of $\A$}
Since $t,s\colon E\to \g$ are fiberwise isomorphisms, 
we have $\A=E\oplus \ker(t)=E\oplus \ker(s)$ as vector bundles. Let
\[ j\colon \A\to E\]
be the projection along $\ker(t)$. The trivialization $E=H\times\g$ 
given by the source map $s\colon E\to\g$ extends to a trivialization $\A=H\times \q$, as follows.
\begin{proposition}\label{prop:trivial}
\begin{enumerate}
\item The map 
\[\A\to \q,\ x\mapsto j(x)^{-1}\circ x \] 
defines a trivialization, $\A\cong H\times\q$, compatible with the metric. 
\item 
The constant sections of $\A\cong H\times\q$ form a Lie algebra under Courant bracket. Thus $\q$ inherits 
a Lie algebra structure. 
\item The subspace $\g$ is a Lie subalgebra of $\q$, and the trivialization of $\A$ restricts 
 to the  given trivialization 
$E\cong H\times\g$. 
\item 
The subspace $\mf{r}$ is a Lie subalgebra of $\q$, and the trivialization of $\A$ restricts to a
trivialization $\ker(t)\cong H\times\mf{r}$.
\item\label{it:d}
Restriction of the anchor map to constant sections defines an action $\q\to \mf{X}(H)$
with coisotropic stabilizers, so that $\A$ is the corresponding action Courant algebroid
(cf.~ Equation \eqref{eq:actioncourant}). 
\end{enumerate}
\end{proposition}
\begin{proof}
For each $h\in H$, the map $\A_h\to \q,\ x\mapsto j(x)^{-1}\circ x$ has inverse 
$\q\to\A_h,\ \zeta\mapsto y\circ \zeta$, where $y\in E_h$ is the unique element such that $s(y)=t(\zeta)$. 
It is clear that the resulting trivialization extends that of $E$. The trivialization is compatible with the metric, since $\l j(x)^{-1}\circ x,\,j(x)^{-1}\circ x\r=\l j(x)^{-1},\,j(x)^{-1}\r+\l x,\,x\r=\l x,\,x\r$. 

By definition, a section $\sig\in\Gamma(\A)$ is `constant' relative to the
trivialization of $\A$ if and only if  $\sig_{h_1h_2}\circ \sig_{h_2}^{-1}\in E$, 
for all $h_1,\,h_2$. This can be rephrased in terms of morphisms: 
Let $P_E\colon \A\times \A\da \A$ be the Courant morphism, with underlying map $H\times H\to H$ projection to 
the second factor, where $(x_1,x_2)\sim_{P_E} x$ if and only if $x_1\in E$ and $x=x_2$. Thus $P_E\subseteq \A\times \ol{\A}\times\ol{\A}$ is obtained from $\A_\Delta\times E$ by interchanging the last two components.
We note that $\sig\in\Gamma(\A)$ is constant if and only if and only if there is a section 
$\hat{\sig}\in \Gamma(\A\times \A)$ such that 
\[ \hat{\sig}\sim_{\Mult_\A}\sig,\ \ \hat{\sig}\sim_{P_E}\sig.\] 
Note that $\hat{\sig}$ is uniquely determined by the constant section $\sig$: Its value at $h_1,h_2$
is $\hat{\sig}_{h_1,h_2}=(\sig_{h_1h_2}\circ \sig_{h_2}^{-1},\sig_{h_2})\in E_{h_1}\times \A_{h_2}$. 
Given another constant section $\sig'$, we have 
\[ \Cour{\hat{\sig},\hat{\sig}'}\sim_{\Mult_\A}\Cour{\sig,\sig'},\ \ 
\Cour{\hat{\sig},\hat{\sig}'}\sim_{P_E}\Cour{\sig,\sig'},\]
since Courant morphism preserve Courant brackets. Hence $\Cour{\sig,\sig'}$ is constant. 
%
%
%
It follows that the space of constant sections is closed under Courant bracket. 
Furthermore, if $\sig,\sig'$ are constant, then 
$\Cour{\sig,\sig'}+\Cour{\sig',\sig}= \a^*\d\l\sig,\sig'\r=0$
since $\l\sig,\sig'\r$ is constant.  Hence the resulting bracket on
$\q$ is skew-symmetric, and hence is a Lie bracket.  

It is obvious that the trivialization of $\A$ restricts to the given trivialization of $E$. 
Since $E$ is involutive, the constant sections with values in $E$ form a Lie subalgebra, 
thus $\g$ is a Lie subalgebra of $\q$. On the other hand,  
$t(x)=0\Leftrightarrow t(j(x))=0\Leftrightarrow s(j(x))=0\Leftrightarrow t(j(x)^{-1}\circ x)=0$, 
shows that the trivialization takes $\ker(t)$ to $\mf{r}$.  The $\mf{r}$-valued constant sections $\sig$ are exactly those for which $\wh{\sig}=0\times\sig$. Since this property is preserved under Courant bracket, it follows that $\mf{r}$ is a Lie subalgebra 
of $\q$. 

Since the anchor map takes Courant brackets to Lie brackets, we obtain
a $\q$-action on $H$. By construction, the Courant bracket on $\A$
extends the Lie bracket on constant sections, and the anchor map
extends the action map. As shown in \cite{LiBland:2009ul} (cf.~ also Section \ref{subsec:basic}), this 
implies that the action of $\q$ has coisotropic stabilizers. 
\end{proof}

The first part of the Proposition may be phrased as the statement that the trivializing map 
$\A\to \q$ defines a morphism of Manin pairs 
\begin{equation}\label{eq:trivial} T\colon (\A,E)\da (\q,\g),\end{equation} 
where $x\sim_T \zeta$ if and only if 
$\zeta= j(x)^{-1}\circ x$. 
%

\subsubsection{Construction of the coisotropic subgroupoid $C\subseteq \TH\times (\ol{\q}\oplus \q)$}
In the following discussion, whenever we write a composition of groupoid elements we take it to be implicit that 
the elements are composable.  
\begin{proposition}\label{prop:revred}
Let $(\A,E)$ be a Dirac Lie group structure on $H$, and define a Lie algebra structure on  $\q=\A_e$ as above. 
Then the subset $C\subseteq \TH\times (\ol{\q}\oplus \q)$ given as 
\begin{equation}\label{eq:c}
 C=\{(v+\alpha,\zeta,\zeta')|\ \exists x\in \A\colon v=\a(x),\ 
\zeta\circ x\circ {\zeta'}^{-1}\in E\}\end{equation}
is an involutive co-isotropic $\ca{VB}$-subgroupoid satisfying $\a(C^\perp)=0$. There is a canonical isomorphism 
of $\ca{CA}$-groupoids $\A\to C/C^\perp$, taking $E$ to $(TH\times (\g\oplus \g))\cap C\ /\ (TH\times (\g\oplus \g))\cap C^\perp .$
\end{proposition}
\begin{proof}
Recall the definition of the division morphism $D\colon \ol{\A}\times\A\da \A$ from the proof of Proposition 
\ref{prop:alongunit} where $(x_1,x_2)\sim_D x$ if and only if $x_1^{-1}\circ x_2=x$. Together with the trivialization $T\colon \A\da \q$, we obtain a morphism $K=(T\times T)\circ D^\top\colon \A\da \ol{\q}\oplus \ol{\q}$. Under this morphism, $x\sim_K (\zeta_1,\zeta_2)$ if and only if 
$\zeta_1\circ x\circ \zeta_2^{-1}\in E$.   

Let $R\colon \A\da \TH\times \A$ be the morphism, with underlying map the diagonal inclusion, 
defined by the property that $x\sim_R (v+\alpha,y)$ 
if and only if $v=\a(x)$ and $y-x=\a^*(y)$. Thus $R\subset \TH\times \A\times \ol{\A}$ is obtained from 
the diagonal morphism cf.~ Corollary~\ref{cor:diag}) by permutation of the components and a sign 
change of the metric. The composition of $R$ with $\TH_\Delta\times K\colon \TH\times \A\da 
\TH\times (\ol{\q}\oplus \q)$ is clean, and defines a morphism 
\[ Q=(\TH_\Delta\times K)\circ R\colon \A\da \TH\times (\ol{\q}\oplus \q)\]
with underlying map $H\to H$ the identity map. Explicitly, 
\begin{equation}\label{eq:qrelation}
 y\sim_Q (v+\alpha,\zeta_1,\zeta_2) \Leftrightarrow \exists x\in \A\colon \zeta_1\circ x\circ \zeta_2^{-1}\in E,\ 
v=\a(x),\ y-x=\a^*(\alpha).\end{equation}
Since $R,\ K,\ \TH_\Delta$ are all $\ca{CA}$-groupoid morphism, the same is true of $Q$. 

We claim  that $\ker(Q)=0$. Indeed, suppose $y\sim_Q (0,0,0)$. The condition 
$x-y=\a^*(\alpha)$ with $\alpha=0$ gives $x=y$, and the condition $\zeta_1\circ x\circ \zeta_2^{-1}\in E$ with $\zeta_i=0$ implies $x=0$, as claimed. On the other hand, $\on{ran}(Q)=C$. By Lemma~\ref{lem:kerZeroLem} below, 
there is an isomorphism of Courant algebroids $\A\to C/C^\perp$.  

Finally, we show that $E= (TH\times(\g\oplus\g))\circ Q$. Suppose \eqref{eq:qrelation}
with $\alpha=0$ and $\zeta_1,\zeta_2\in\g$. Then $x=y$, and $\zeta_1\circ y\circ \zeta_2^{-1}\in E$. Since $\zeta_1,\zeta_2\in \g= E^{(0)}$ it follows that $y\in E$. Therefore $(TH\times(\g\oplus\g))\circ Q\subseteq E$, and the conclusion follows, since both sides are Lagrangian.
\end{proof}

\begin{lemma}\label{lem:kerZeroLem}
Let $R\colon \A\da  \A'$ be a Courant morphism, with 
underlying map $\Phi\colon M\to M'$ a diffeomorphism. 
If $\ker(R)=0$, then $C=\on{ran}(R)$ is co-isotropic, with $\a(C^\perp)=0$, and $\A\cong \A'_C$ as Courant algebroids. If $\A,\A'$ are $\ca{CA}$-groupoids, and $R\colon\A\da\A'$ is a $\ca{CA}$-groupoid morphism, then 
$\A\cong \A'_C$ is an isomorphism of $\ca{CA}$-groupoids.  
\end{lemma}
\begin{proof}
The inclusion $R\subseteq (\A'\times C)|_{\on{gr}(\Phi)}$ shows $(0\times C^\perp)|_{\on{gr}(\Phi)}\subseteq 
R^\perp=R$, hence $C^\perp\subseteq C$ so that $C$ is co-isotropic. Furthermore, since $\a(0,y')=(0,\a(y'))$ 
for $y'\in C^\perp$ is 
tangent to  $\on{gr}(\Phi)$, we see $\a(y')=0$, hence $\a(C^\perp)=0$. Let 
$P\colon \A'\da \A'_C$ be the Courant morphism defined by the reduction. 
Thus $y'\sim_P y''$ if and only if 
 $y'\in C$, with $y''$ its image under the quotient map. We will show that $P\circ R\colon \A\da \A'_C$ is an isomorphism. Indeed, let $x\in\ker (P\circ R)$. Then $x\sim_R x',\ x'\sim_P 0$ for some $x'\in \A'$. 
By definition of $P$, we have $x'\in C^\perp$. Since $(0\times C^\perp)|_{\on{gr}(\Phi)}\subseteq R$, $x\sim_R x'$ 
implies $x\sim_R 0$, hence $x=0$. The property $\ker(P\circ R)=0,\ \on{ran}(P\circ R)=\A'$ means 
that $P\circ R$ defines an isomorphism $\A\cong \A'_C$.  If $R\colon \A\da \A'$ is a morphism of 
$\ca{CA}$-groupoids, then so is $P$ and hence $P\circ R$.  
\end{proof}
The co-isotropic subbundle $C$ has an alternative description, similar to Proposition 
\ref{prop:ciscois}. 
\begin{proposition}\label{prop:altCdesc}
The co-isotropic subbundle $C\subseteq \TH\times (\ol{\q}\oplus \q)$ from Proposition 
\ref{prop:revred} may be written, 
\[
C=\{(v+\alpha,\zeta_1,\zeta_2)|\ \iota(v)\theta^R_h=\Ad_h f(\zeta_2)-f(\zeta_1)\}.
\]
\end{proposition}
\begin{proof}
Given $(v+\alpha,\zeta_1,\zeta_2)\in \TH\times (\ol{\q}\oplus \q)$, let 
$z\in E_h$ be the unique element with $s(z)=t(\zeta_2)$. By Equation \eqref{eq:eqE}, we have 
$ \Ad_h t(\zeta_2)=\Ad_h s(z)=\iota(\a(z))\theta^R_h+t(z)$. Together with Equation  
\eqref{eq:deff} we obtain
\[ \begin{split}
\Ad_h f(\zeta_2)-f(\zeta_1)&=\Ad_h(t(\zeta_2)+\a(\zeta_2))-(t(\zeta_1)+\a(\zeta_1))\\
&=\iota(\a(z))\theta^R_h+\Ad_h\a(\zeta_2)-\a(\zeta_1)+t(z)-t(\zeta_1).
\end{split}\]
The first three terms lie in $\h$, the last two in $\g$. Hence the property  
$\iota(v)\theta^R_h=\Ad_h f(\zeta_2)-f(\zeta_1)$ is equivalent to the two 
conditions 
\begin{equation}\label{eq:that}
 \iota(v)\theta^R_h=\iota(\a(z))\theta^R_h+\Ad_h\a(\zeta_2)-\a(\zeta_1),\ \ t(\zeta_1)+t(z)\end{equation}
Equation \eqref{eq:that} is equivalent to the condition that 
$x:=\zeta_1^{-1}\circ z\circ \zeta_2$ is defined and $v=\a(x)$. 
\end{proof}

\begin{remark}\label{rem:trivial}
Define a bundle map $H\times \q\to C,\ (h,\zeta)\mapsto (v,\xi,\zeta)$ where
$\iota(v)\theta^R_h=p(\Ad_h f(\zeta))$ and $\xi=(1-p)\Ad_h f(\zeta)$. 
The sub-bundle given as its image is invariant under left groupoid multiplication by elements of 
$TH\times (\g\oplus\g)$, and is a complement to $C^\perp$. Hence, its composition with the quotient map to $\A=C/C^\perp$ is 
the trivialization $H\times\q\cong \A$ from Proposition \ref{prop:trivial}.
\end{remark}

\subsubsection{Relation between $(\dd,\g,\h)_\beta$ and $(\q,\g,\mf{r})_\gamma$}
We still have to show that $\beta$ is $H$-invariant, and that $(\q,\g,\mf{r})_\gamma$
is the Dirac Manin triple associated to $(\dd,\g,\h)_\beta$ by the construction from 
Section~\ref{subsec:genman}. 
\begin{proposition}\label{prop:atlast}
\begin{enumerate}
\item The map $f\colon \q\to \dd$ is a Lie algebra homomorphism. 
\item The element $\beta\in S^2\dd$ 
defined by $\beta^\sharp=f\circ f^*$ is $\dd$-invariant as well as $\Ad_h$-invariant. 
\item $\g$ is $\beta$-coisotropic. 
\item The Lie subalgebra $\mf{c}=\g\ltimes \dd^*_\beta\subseteq \dd\ltimes\dd^*_\beta$ is coisotropic, and 
the map 
\[ \mf{c}\to \q,\ (\xi,\mu)\mapsto \xi+f^*(\mu)\]
descends to an isometric Lie algebra isomorphism $\mf{c}/\mf{c}^\perp\to\q$. 
\end{enumerate}
\end{proposition}
\begin{proof}
\begin{enumerate} 
\item Since $f$ restricts to a Lie algebra homomorphism $\mf{r}\to \h$, 
and is given by the identity map on $\g$, we need 
only check that $f([\tau,\xi])=[f(\tau),\xi]$ for $\xi\in \g,\ \tau\in\mf{r}$. 
Define sections of $C\subseteq \TH\times (\q\oplus \q)$ by 
\[ (f(\tau)^L,0,\tau),\ \ \ (p_\h(\Ad_h\xi))^R,(1-p_\h)\Ad_h \xi,\xi).\]
Since $C$ is involutive, their Courant bracket 
\[ \big(p_\h(\Ad_h[f(\tau),\xi])^R,(1-p_\h)\Ad_h\,[f(\tau),\xi],[\tau,\xi]\big)
\]
is again a section of $C$. Thus 
\[\begin{split}
 p_\h(\Ad_h[f(\tau),\xi])&=\Ad_h f([\tau,\xi])-(1-p_\h)\Ad_h\,[f(\tau),\xi]\\
&=\Ad_h f([\tau,\xi])-\Ad_h\,[f(\tau),\xi]+p_\h(\Ad_h[f(\tau),\xi]),\end{split}
\]
giving $f([\tau,\xi])=[f(\tau),\xi]$ as desired.  
\item 
By the same argument as in the proof of Proposition~\ref{prop:ciscois}, 
the fiber of $C^\perp$ at $h\in H$ is spanned by the sections $\psi(\mu)$, $\mu\in\dd^*$.
The property $C^\perp\subseteq C$ gives $0=\Ad_h(f(f^*(\mu))-f(f^*(\mu))$ as desired. 
This shows that $\beta$ is invariant under the adjoint action of $H$. 
In particular it is $\h$-invariant. Since $f$ is a Lie algebra homomorphism, it is 
also equivariant under the adjoint action of $\g$. Thus $\beta=f\circ f^*$ is $\g$-invariant as 
well.  
\item 
The dual map $f^*\colon \dd^*\to \q$ 
takes $\on{ann}(\g)\subseteq\dd^*$ to $\g^\perp\subseteq\q$. Hence, for $\mu\in \on{ann}(\g)$, 
$\beta(\mu,\mu)=\l f^*(\mu),f^*(\mu)\r=0$. 
\item
The map $\mf{c}\to \q,\ (\xi,\mu)\mapsto \xi+f^*(\mu)$ is surjective, since its image contains $\g$ as well as the complement
$f^*(\on{ann}(\h))=\mf{r}^\perp$. The map clearly preserves the bilinear
forms, hence its kernel must be $\mf{c}^\perp$.  Using the
identity
\[ [f^*(\mu_1),f^*(\mu_2)]=f^*([\beta^\sharp(\mu_1),\mu_2]),\]
(which is verified by pairing both sides with $\zeta\in\q$), one finds that 
it is a Lie algebra homomorphism.
\end{enumerate}
\end{proof}

It follows that $(\dd,\g,\h)_\beta$ is an $H$-equivariant Dirac Manin triple, and that $(\q,\g,\mf{r})_\gamma$ 
and $f\colon \q\to \g$ result from the construction of Section~\ref{subsec:genman}.  Propositions~\ref{prop:revred} and~\ref{prop:altCdesc} define a canonical isomorphism between $(\A,E)$ and the Dirac Lie group constructed in Section~\ref{subsec:construct}.

\section{Morphisms}\label{sec:mor}
In this section, we will show that the correspondence between Dirac Lie group structures and $H$-equivariant Dirac Manin pairs respects morphisms, thus completing the proof of Theorem \ref{th:main}. 

\subsection{Morphisms of Dirac Manin triples}\label{subsec:MorphDiracMan}
The Dirac Manin triples form a category relative to the following notion of morphism. 
\begin{definition}
A \emph{morphism of Dirac Manin triples} 
\[ \k\colon (\dd_0,\g_0,\h_0)_{\beta_0}\da (\dd_1,\g_1,\h_1)_{\beta_1}\] 
is a $\beta_1-\beta_0$-coisotropic Lie subalgebra $\k\subseteq \dd_1\times\dd_0$ such that 
\[ \g_1=\k\circ \g_0,\ \ \ \h_0=\h_1\circ \k.\] 
\end{definition}
See Appendix~\ref{app:compos} for compositions of linear relations. 
The property $\h_0=\h_1\circ \k$ implies 
$\ker(\k)\subseteq \h_0$, hence $\g_0\cap \ker(\k)=0$. Hence there is a linear map $\psi\colon 
\g_1\to \g_0$, taking $\xi_1\in\g_1$ to the unique element $\xi_0=\psi(\xi_1)\in\g_0$ with 
$\xi_0\sim_\k \xi_1$. By a similar argument, there is a linear map $\phi\colon \h_0\to \h_1$, taking $\nu_0\in \h_0$ to the unique element $\nu_1=\phi(\nu_0)\in\h_1$ with $\nu_0\sim_\k \nu_1$. 

\begin{lemma}
$\k$ is the direct sum of the graphs of $\psi,\phi$. In particular 
\[\dim \k=\dim\g_1+\dim(\dd_0/\g_0).\]
\end{lemma} 
\begin{proof}
Suppose $\xi_0+\nu_0\sim_\k \xi_1+\nu_1$ with $\xi_i\in\g_i$ and $\nu_i\in \h_i$. 
Since $\psi(\xi_1)\sim_\k \xi_1$ and $\nu_0\sim \phi(\nu_0)$, it follows that 
$\xi_0-\psi(\xi_1)\sim_\k \nu_1-\phi(\nu_0)$. Since $\h_0=\h_1\circ \k$, this implies that 
$\xi_0-\psi(\xi_1)\in \h_0$, hence $\xi_0-\psi(\xi_1)=0$. Similarly 
$\nu_1-\phi(\nu_0)=0$. 
\end{proof} 
%
%
\begin{remark}\label{rem:nondeg}
Suppose that the $\beta_i$ are non-degenerate and that the Lie subalgebras $\g_i$ are Lagrangian.
Then  $\k\subseteq \dd_1\times \dd_0$ is $\beta_1-\beta_0$-Lagrangian, for dimensional reasons. 
For all $(\xi_1+\phi(\nu_0),\psi(\xi_1)+\nu_0)\in\k$, we find (denoting the metric on $\dd_i$ 
simply by $\l\cdot,\cdot\r$), 
\[ \begin{split}
0&=\l (\xi_1+\phi(\nu_0),\psi(\xi_1)+\nu_0), (\xi_1+\phi(\nu_0),\psi(\xi_1)+\nu_0)\r\\
&=2\l \phi(\nu_0),\xi_1)\r+\l \phi(\nu_0),\phi(\nu_0)\r-2\l \psi(\xi_1),\nu_0\r-\l\nu_0,\nu_0\r.
\end{split}\]
Since this is true for all $\xi_1\in\g_1$, this shows that $\psi=\phi^*$, for the identification $\h_i=\g_i^*$ given by the pairing, and that 
$\phi$ preserves the induced bilinear forms on $\h_i$. In particular, for ordinary Manin triples 
  $(\dd_i,\g_i,\h_i)$, a morphism is given by 
a pair of Lie algebra homomorphisms $\psi\colon \g_1\to \g_0$ and $\phi\colon \h_0\to \h_1$ such that 
$\phi=\psi^*$. 
\end{remark}

Let $(\dd,\g,\h)_\beta$ be a Dirac Manin triple, and $(\q,\g,\mf{r})_\gamma$ the Dirac Manin triple 
obtained by the construction from Section~\ref{subsec:genman}, with the corresponding map 
$f\colon \q\to \dd$. Then the graph of $f$ defines a morphism of Dirac Manin triples 
\[ \on{gr}(f)\colon (\q,\g,\mf{r})_\gamma\da (\dd,\g,\h)_\beta.\]
The construction of $(\q,\g,\mf{r})_\gamma$ is functorial:   
\begin{proposition}\label{prop:lcons}
Let $\k\colon (\dd_0,\g_0,\h_0)_{\beta_0}\da (\dd_1,\g_1,\h_1)_{\beta_1}$ be a morphism of Dirac Manin triples, 
given as the direct sum of the graphs of $\phi\colon \h_0\to \h_1$  
and $\psi\colon \g_1\to \g_0$. Let $(\q_i,\g_i,\mf{r}_i)_{\gamma_i}$ be the Dirac Manin triples
associated to $(\dd_i,\g_i,\h_i)_{\beta_i}$ as in Section 
\ref{subsec:genman}, with the corresponding maps $f_i\colon \q_i\to \dd_i$. 
Let $\mf{l}\subseteq \q_1\times \q_0$ be the direct sum of the graphs of $\psi$ and 
$\kappa=\psi^*\colon \mf{r}_0=\g_0^*\to \mf{r}_1=\g_1^*$. Then $\mf{l}$ 
defines a morphism of Dirac Manin triples 
\[ \mf{l}\colon (\q_0,\g_0,\mf{r}_0)_{\gamma_0}\da
(\q_1,\g_1,\mf{r}_1)_{\gamma_1},\]
with $f(\mf{l})\subseteq \k$  where $f=f_1\times f_0$. As a consequence, 
\[\ \ \ \ \ \on{gr}(f_1)\circ \mf{l}=\k\circ \on{gr}(f_0).\]
%
%
\end{proposition}
\begin{proof}
We will begin with an alternative construction of $\mf{l}$ by reduction. 
Write $\dd=\dd_1\times\dd_0$, $\g=\g_1\times\g_0$, $\h=\h_1\times\h_0$, and $\beta=\beta_1-\beta_0$ 
so that $\k\subseteq \dd$ is a $\beta$-coisotropic Lie subalgebra. Then 
\[ \wh{\k}=\k\ltimes \on{ann}(\k)\subseteq \dd\ltimes \dd^*_\beta\]
is a Lagrangian Lie subalgebra. Since $\k$ is the direct sum of the graphs 
of $\psi\colon \g_1\to \g_0$ and $\phi\colon \h_0\to \h_1$, 
its annihilator $\on{ann}(\k)$ is the direct sum of the graphs of  
$\phi^*\colon \on{ann}(\g_1)\to \on{ann}(\g_0)$ and 
$\psi^*\colon \on{ann}(\h_0)\to \on{ann}(\h_1)$. Let $\mf{l}\subseteq\q$ be the Lagrangian Lie subalgebra defined by reduction of $\wh{\k}$ with respect to the coisotropic Lie subalgebra $\mf{c}=\g\ltimes \dd^*_\beta$, 
\[ \mf{l}=\wh{\k}\cap \mf{c}/\wh{\k}\cap \mf{c}^\perp\subseteq \q.\]
Since $\mf{k}$ is coisotropic, we have $\beta^\sharp(\on{ann}\mf{k})\subseteq\mf{k}$, and hence $f(\mf{l})\subseteq\mf{k}$, by definition of $f$.
Equivalently, $\mf{k}\supseteq\on{gr}(f_1)\circ\mf{l}\circ\on{gr}(f_0)^\top$. This implies $$\mf{k}\circ\on{gr}(f_0)\supseteq\on{gr}(f_1)\circ\mf{l}\circ\on{gr}(f_0)^\top\circ\on{gr}(f_0)\supseteq\on{gr}(f_1)\circ\mf{l}.$$

The inclusion
$$\mf{r}_1\circ\mf{l}=\h_1\circ\on{gr}(f_1)\circ\mf{l}\subseteq\h_1\circ\k\circ \on{gr}(f_0)=\h_0\circ\on{gr}(f_0)=\mf{r}_0$$
shows $\mf{r}_1\circ\mf{l}=\mf{r}_0$, since both sides are Lagrangian.
On the other hand, since $\wh{\mf{k}}\cap\mf{c}\supseteq\mf{k}\cap\g=\on{gr}(\psi)$, we have $\mf{l}\supseteq\on{gr}(\psi)$, and hence 
\[ \mf{l}\circ\g_0\supseteq \on{gr}(\psi)\circ \g_0\supseteq\g_1.\] 
Since both sides are Lagrangian, it follows that $\mf{l}\circ\g_0=\g_1$.
 This shows that $\mf{l}$ defines a morphism $\mf{l}\colon (\q_0,\g_0,\mf{r}_0)_{\gamma_0}\da
(\q_1,\g_1,\mf{r}_1)_{\gamma_1}$ of Dirac Manin triples. 
By Remark~\ref{rem:nondeg}, $\mf{l}$ is the direct sum of the graph of $\psi$ 
with a Lie algebra homomorphism $\kappa=\psi^*$, recovering the 
description of $\mf{l}$ given in the proposition.

Finally we show that $\on{gr}(f_1)\circ \mf{l}=\k\circ \on{gr}(f_0)$, or equivalently 
\begin{equation}
\phi\circ f_0=f_1\circ \kappa,\ \ \psi\circ f_1=f_0\circ \psi
\end{equation}
For the first equation, let $y\in \mf{r}_0$. Then $\kappa(y)+y\in \mf{r}\cap \mf{l}$, and 
$f_1(\kappa(y))+f_0(y)\in \h\cap \k$. Hence $\phi(f_0(y))=f_1(\kappa(y))$ by definition of 
$\phi$. Similarly, given $x\in \g_1$ we have $x+\psi(x)\in\g\cap \mf{l}$, hence
$f_1(x)+f_0(\psi(x))\in\g\cap \k$. Thus $\psi(f_1(x))=f_0(\psi(x))$. 
\end{proof}

\subsection{The Equivalence Theorem for morphisms} 

A \emph{morphism of $H_i$-equivariant Dirac Manin triples 
$ \k\colon (\dd_0,\g_0,\h_0)_{\beta_0}\da (\dd_1,\g_1,\h_1)_{\beta_1}$} is a morphism of Dirac Manin triples 
where the map $\phi\colon \h_0\to \h_1$ is the differential 
of a Lie group homomorphism $\Phi\colon H_0\to H_1$, and where $\mf{k}$ is 
invariant under the adjoint action of $H_0\cong \on{gr}(\Phi)\subseteq H_1\times H_0$.  
Our goal is to show that such morphisms are in 1-1 correspondence with the morphisms of the 
corresponding Dirac Lie group structures. (Recall that a morphism of Dirac Lie groups is defined to be a morphism of multiplicative Manin pairs.)
\begin{theorem}[Morphisms] \label{thm:morph}
There is a 1-1 correspondence between  morphisms of 
of $H_i$-equivariant Dirac Manin triples 
$\k\colon (\dd_0,\g_0,\h_0)_{\beta_0}\da (\dd_1,\g_1,\h_1)_{\beta_1}$
and morphisms of the corresponding 
Dirac Lie groups,  $L\colon (\A_0,E_0)\da (\A_1,E_1)$. 
This correspondence is compatible with the composition of morphisms. 
\end{theorem}

The proof of this Theorem will be given in the next two subsections. We will use the presentation 
$\A_i=C_i/C_i^\perp$ with $C_i\subset \TH\times(\ol{\q}_i\oplus \q_i)$, as in Section~\ref{subsec:construct}. 
For convenience, let us write $\A=\A_1\times\ol{\A_0}$, $C=C_1\times C_0$, $H=H_1\times H_0$,  
$\q=\ol{\q}_1\times \q_0$, $\dd=\dd_1\times\dd_0$, $\beta=\beta_1-\beta_0$, 
$\g=\g_1\times\g_0$ etc. 

\subsubsection{The direction $\Rightarrow$}\label{sec:dir=>}
Given $\k\colon (\dd_0,\g_0,\h_0)_{\beta_0}\da (\dd_1,\g_1,\h_1)_{\beta_1}$, let 
$\mf{l}\colon (\q_0,\g_0,\mf{r}_0)_{\gamma_0}\da (\q_1,\g_1,\mf{r}_1)_{\gamma_1}$ be the morphism 
constructed in Proposition~\ref{prop:lcons}, and consider the Courant morphism
\[ R_\Phi\times (\mf{l}\oplus \mf{l})\colon \TH_0\times (\ol{\q}_0\oplus \q_0)\da 
\TH_1\times (\ol{\q}_1\oplus \q_1).\]
Let $S_i\colon \A_i\da \TH_i\times (\ol{\q}_i\oplus \q_i)$
be the Courant morphisms defined by the reduction, as in \eqref{eq:S}.
\begin{lemma}\label{lem:Lconstr}
The composition $L=S_1^\top\circ 
(R_\Phi\times (\mf{l}\oplus \mf{l}))\circ S_0$
%
defines a Dirac Lie group morphism 
\begin{equation}\label{eq:LmapA0toA1} L\colon (\A_0,E_0)\da (\A_1,E_1).\end{equation}
The trivialization $\A\to \q$ from Proposition~\ref{prop:trivial} restricts to a trivialization 
$L\to \mf{l}$. 
\end{lemma}
\begin{proof}
By definition, 
\[ L=(R_\Phi\times (\mf{l}\oplus \mf{l}))_C\subseteq \A=(\TH\times (\ol{\dd}\oplus\dd))_C.\]
%
We will show that the trivializing map $\A\to \q$ restricts to $L\to \mf{l}$, defining an isomorphism
$L\xra{\cong}\on{gr}(\Phi)\times \mf{l}$. Since $\A_i\xra{\cong} H_i\times\q_i$ restricts 
to $E_i\xra{\cong} H_i\times\g_i$, this then implies that $L$ gives a morphism 
of Manin pairs $(\A_0,E_0)\da (\A_1,E_1)$. 

We first verify $L_e=\mf{l}$. By \eqref{eq:defc}, the fiber $C_e$ consists 
of elements $(f(\zeta')-f(\zeta)+\mu,\zeta,\zeta')$ with $\mu\in \h^*$ and $\zeta,\zeta'\in\q$.
Given $\zeta\in\mf{l}$, we have  $t(\zeta)+\a(\zeta)=f(\zeta)\in\mf{k}$ by Equation~\eqref{eq:deff}.  
Hence $t(\zeta)\in\mf{k}\cap\g=\mf{l}\cap\g$ and $\a(\zeta)\in \mf{k}\cap\h$. 
It follows that 
\begin{equation}\label{eq:ate}
(\a(\zeta),t(\zeta),\zeta)\in C_e\cap(R_\Phi\times (\mf{l}\oplus \mf{l})),\end{equation}
mapping to $\zeta\in C_e/C_e^\perp$ (cf. Remark \ref{rem:trivial}).

This shows $L_e\supseteq\mf{l}$, and equality follows since both sides are Lagrangian.  

Recall next that on $E\subset \A$, the trivialization is given by the source map $s\colon E\to \g$. 
Suppose $z\in E_h$ satisfies $s(z)\in\mf{l}$. Then $s(z)\in \g\cap \mf{l}=\g\cap \k$. 
Equation~\eqref{eq:eqE} shows that 
\[ \Ad_h s(z)=t(z)+\iota(\a(z))\theta^R_h.\] 
Since $\mf{k}$ is invariant under the action of $\on{gr}(\Phi)\subseteq H_1\times H_0$, 
the left hand side lies in $\mf{k}$, hence so does the right hand side. Since $\k$ is the direct sum of its intersections with $\h,\g$, this shows $t(z)\in \g\cap \mf{k}=\g\cap \mf{l}$ and $\a(z)\in T\on{gr}(\Phi)$. Hence $(\a(z),t(z),s(z))\in C\cap R_\Phi\times (\mf{l}\oplus \mf{l})$ maps to $z\in E$, proving that $z\in E\cap L$. 

Finally, given $h\in\on{gr}(\Phi)$ and $\zeta\in \mf{l}$, let $z\in E_h$ be the unique element such that $t(\zeta)=s(z)$. Then $x:=z\circ \zeta\in \A_h$ maps to $\zeta$ under the trivialization. 
Since $\zeta\in L_e$ and $z\in L_h$, by the above, it follows that $x\in L_h$. This show that the 
image of $L_h$ under $\A\to \q$ contains $\mf{l}$, and equality follows since both sides are Lagrangian. 
\end{proof}

\subsubsection{The direction $\Leftarrow$}
Suppose $L\colon  (\A_0,E_0)\da (\A_1,E_1)$ is a morphism of Dirac Lie groups, with underlying 
Lie group homomorphism $\Phi\colon H_0\to H_1$. We will show how to 
construct $\k$.

\begin{proposition}
Suppose that $$L\colon (\A_0,E_0)\dasharrow(\A_1,E_1)$$ is a morphism of multiplicative Manin pairs over the morphism $\Phi:H_0\to H_1$ of Lie groups. Then there exists a unique $H$-equivariant morphism $$\mf{k}\colon (\mf{d}_0,\g_0,\h_0)_{\beta_0}\dasharrow (\mf{d}_1,\g_1,\h_1)_{\beta_1}$$ of Dirac Manin triples such that $L$ results from the construction of Section~\ref{sec:dir=>}.
\end{proposition}

\begin{proof}
We begin by trivializing the bundle $L$. By definition of a morphism of Manin pairs, any element of $(E_1)_{\Phi(h_0)}$ is $L_{(\Phi(h_0),h_0)}$-related to a unique element of $(E_0)_{h_0}$. This shows that $L\cap E$ is of constant rank equal to $\on{rank}(E_1)$, 
and hence is a $\ca{VB}$-subgroupoid. Let $\mf{l}=L_e$ be the fiber at the identity. The space of 
objects of $L$ is $L^{(0)}=L\cap A^{(0)}=\mf{l}\cap \g$, and coincides with the space of 
objects of $L\cap E$. Thus $L\cap E$ is a wide subgroupoid of $L$, and $L=(L\cap E) \oplus \ker(t|_L)$.  
It follows that the projection $j\colon \A\to E$ along $\ker(t)$ restricts to a projection $j\colon L\to L\cap E$, and that the map $\A\to \q,\ x\mapsto j(x)^{-1}\circ x$ restricts to a trivialization $L\to \mf{l}$. 
Equivalently, the isomorphism $\A \xra{\cong} H\times \q$ restricts to an isomorphism
$L \xra{\cong} \on{gr}(\Phi)\times\mf{l}$, with $L\cap E\xra{\cong} \on{gr}(\Phi)\times (\mf{l}\cap\g)$
and $L\cap \ker(t)\xra{\cong} \on{gr}(\Phi)\times (\mf{l}\cap\mf{r})$. 

Since $L$ is a Dirac structure along $\on{gr}(\Phi)$, it follows that $\mf{l}\subset \q$ is a Lie subalgebra. 
Letting $T_i\colon(\A_i,E_i)\da (\q_i,\g_i)$ be the trivializations described in Proposition~\ref{prop:trivial}, 
we have shown that 
\begin{equation*}\label{eq:trivCompMorph}\xymatrix{
(\A_0,E_0)\ar@{-->}^L[r]\ar@{-->}_{T_0}[d]&(\A_1,E_1)\ar@{-->}^{T_1}[d]\\
(\q_0,\g_0)\ar@{-->}^{\mf{l}}[r]&(\q_1,\g_1)
}\end{equation*}
is a commutative diagram of morphisms of Manin pairs. From $\mf{l}=\mf{l}\cap \g\oplus 
\mf{l}\cap\mf{r}$, and the fact that $\mf{l}\cap \g$ is the graph of a linear map 
$\psi\colon\g_1\to \g_0$, it follows that $\mf{l}\cap\mf{r}$ is the graph of a linear 
map $\kappa=\psi^*\colon \mf{r}_0\to \mf{r}_1$. Hence $\mf{l}$ defines a morphism of Dirac Manin triples 
$\mf{l}\colon (\q_0,\g_0,\mf{r}_0)_{\gamma_0}\da (\q_1,\g_1,\mf{r}_1)_{\gamma_1}$. 

We have to verify that $L$ is recovered from $\mf{l}$ by the construction from the last subsection. 
As in the proof of Proposition~\ref{prop:revred}, consider the multiplicative Courant relations
\[R_i:\A_i\da \T H_i\times \A_i,\ \ \
 D_i^\top:\A_i\da \ol{\A_i}\times \A_i\]
for $i=0,1$,  where $\A_i\times\ol{\A_i}$ is taken to be the pair groupoid. Thus 
\[ v+\alpha\sim_{R_i}(y,x) \Leftrightarrow x-y=\a^*\alpha,\ \  v=\a(x),\ \ \ \ 
x\sim_{D_i^\top} (x_1,x_2) \Leftrightarrow x=x_1^{-1}\circ x_2.\]
We also have the identity morphisms $(\TH_i)_\Delta\colon \TH_i\da \TH_i$, the 
standard lift $R_\Phi:\T H_0\da\T H_1$ of $\Phi$ defined by \eqref{eq:standardLift}, and the Courant morphism $\mf{l}=L_e\colon \q_0\da\q_1$. These compose to form the following diagram of multiplicative Courant relations,

$$\xymatrix{
\A_0\ar@{-->}[rr]^L\ar@{-->}[d]_{R_0}&&\A_1\ar@{-->}[d]^{R_1}&\\
\T H_0\times\A_0\ar@{-->}[rr]^{R_\Phi\times L}\ar@{-->}[d]_{(\T H_0)_\Delta\times D_0^\top}&&
\T H_1\times\A_1\ar@{-->}[d]^{(\T H_1)_\Delta\times D_1^\top}\\
\T H_0\times(\ol{\A_0}\times \A_0)\ar@{-->}[rr]^{R_\Phi\times L\times L}\ar@{-->}[d]_{(\T H_0)_\Delta\times T_0\times T_0}&&
\T H_1\times(\ol{\A_1}\times \A_1)\ar@{-->}[d]^{(\T H_1)_\Delta\times T_1\times T_1}\\
\T H_0\times(\ol{\q}_0\oplus \q_0)\ar@{-->}[rr]^{R_\Phi\times(\mf{l}\oplus\mf{l})}&&\T H_1\times(\ol{\q}_1\oplus{\q_1})
}$$
Let $Q_i\colon \A_i\da \TH_i\times (\ol{\q}_i\oplus \q_i)$ be the vertical composition of relations,
as in Proposition~\ref{prop:revred}. In the diagram above, the bottom square commutes, since the trivialization of $\A$ 
restricts to the trivialization $L\to \on{gr}(\Phi)\times\mf{l}$. 
The middle square commutes since $L:\A_0\da \A_1$ is a multiplicative Courant morphism, and $D_i$ is the graph of division for the groupoid $\A_i$. Finally, the top square commutes since both $R_1\circ L$ and $(R_\Phi\times L)\circ R_0$
are the morphism $\A_0\da \TH_1\times \A_1$, where $x$ is related to $(v+\alpha,y)$ if and only if
$x\sim_L y+\a^*\alpha, v=\a(y)$. We conclude that
\[ (R_\Phi\times(\mf{l}\oplus\mf{l}))\circ Q_0=Q_1\circ L.\] 
Now, as in Proposition~\ref{prop:revred}, $\ker{Q_1}=0$. Consequently, $Q_1^\top\circ Q_1=(\A_1)_\Delta$, the identity relation for $\A_1$. Hence
 \begin{equation}\label{eq:Lred}Q_1^\top\circ \big(R_\Phi\times(\mf{l}\oplus\mf{l})\big)\circ Q_0=Q_1^\top\circ Q_1\circ L= L.\end{equation}
 Let $Q=Q_1\times Q_0:\A\da\T H\times(\ol\q\oplus\q)$, and $C=\on{ran}(Q)$. Since $\ker{Q}=0$, Lemma~\ref{lem:kerZeroLem} and \eqref{eq:Lred} imply that
\[ L=(R_\Phi\times (\mf{l}\oplus \mf{l}))_C\subseteq \A=(\TH\times (\ol{\dd}\oplus\dd))_C,\]
%
as claimed.
 
It remains to construct $\mf{k}\colon (\dd_0,\g_0,\h_0)_{\beta_0}\da(\dd_1,\g_1,\h_1)_{\beta_1}$.
Recall again that $L\cap E$ is a vacant $\ca{LA}$-groupoid. By Proposition~\ref{prop:VLAg1}, it 
corresponds to a $\on{gr}(\Phi)\cong H_0$-equivariant Lie algebra triple $(\mf{k},\g',\h')$.
Since $L\cap E$ is a subgroupoid of the vacant $\ca{LA}$-groupoid $E$, we have $\g'\subseteq\g$, $\h'\subseteq\h$, and $\mf{k}\subseteq\dd$ is a $\on{gr}(\Phi)$-invariant Lie subalgebra. Here $\h'$ is the Lie algebra of the base
$H'=\on{gr}(\Phi)\subseteq H$, i.e.~ $\h'=\on{gr}(\phi)$ with $\phi=\d_e\Phi$. On the other hand, 
we saw that $\g'=\mf{l}\cap \g$ is the graph of a Lie algebra homomorphism $\psi\colon\g_1\to \g_0$. 
Since $L\cap E$ is supported on $\on{gr}(\Phi)$ and $t(L\cap E)=(L\cap E)^{(0)}=\mf{l}\cap\g$, it follows from \eqref{eq:deff} that $f(\mf{l})\subseteq\mf{k}$. Then the fact that $\mf{k}\subseteq\dd$ is co-isotropic 
is a special case of Lemma~\ref{lem:small} below.

We conclude that $\mf{k}\subseteq\dd$ defines an $H$-equivariant morphism of Dirac Manin triples $\mf{k}\colon (\dd_0,\g_0,\h_0)_{\beta_0}\da(\dd_1,\g_1,\h_1)_{\beta_1}$ over $\Phi:H_0\to H_1$.
Finally, we verify that the image of $\mf{k}\ltimes\on{ann}(\mf{k})$ under the projection $\g\ltimes\dd^*_\beta\to\q$ is equal to $\mf{l}$.
Since $f(\mf{l})\subseteq \k$, we have $f^*(\on{ann}(\k))\subseteq \mf{l}^\perp=\mf{l}$. 
Since the projection coincides with $f^*$ on $\dd^*$, this shows that the image 
of $\on{ann}(\k)$ is contained in $\mf{l}$. 
On the other hand, the image of $\k\cap (\g\ltimes\dd^*_\beta)$ equals the image of 
$\on{gr}(\psi)^\top\subseteq \mf{l}$. 
This shows that the reduced Lagrangian is contained in $\mf{l}$, and hence coincides with $\mf{l}$.
\end{proof}

The proof used: 

\begin{lemma}\label{lem:small}
Suppose $\mathsf{l}\colon V\to V'$ is a linear map taking 
$\beta\in S^2V$ to $\beta'\in S^2V'$. Suppose $W\subseteq V,\ W'\subseteq V'$ with 
$\mathsf{l}(W)\subseteq W'$. If $W\subseteq V$ is $\beta$-coisotropic, then 
$W'$ is $\beta'$-coisotropic.\end{lemma}\begin{proof} Indeed, $\mathsf{l}(W)\subseteq W'$ implies
$\mathsf{l}^*(\on{ann}(W'))\subseteq \on{ann}(W)$. Thus, if $\beta$ vanishes on 
$\on{ann}(W)$ then $\beta'=\mathsf{l}(\beta)$ vanishes on $\on{ann}(W')$. \end{proof}

\section{Explicit formulas}\label{sec:explF}
\subsection{The $\ca{CA}$-groupoid structure in terms of the trivialization}
Let $(\A,E)$ be a Dirac Lie group structure on $H$. In this Section we will work out the 
formulas for the Courant algebroid structure and $\ca{VB}$-groupoid structure on $\A$ in terms of the 
trivialization $\A\xra{\cong}H\times\q$, obtained in Proposition~\ref{prop:trivial}. 

We will need some background from the theory of matched pairs. Given a Lie algebra triple $(\dd,\g,\h)$, one obtains actions of of $\h$ on $\g\cong\dd/\h$ and $\g$ on $\h\cong\dd/\g$, satisfying the compatibility conditions of a \emph{matched pair} $\g\bowtie \h$ of Lie algebras.  Moreover, letting $p_\h\in\End(\dd)$ be the projection to $\h$ along $\g$, the $\g$-action on $\h$ extends to a $\dd$-action 
on $\h$ by 
\[ \nu\mapsto p_\h([\xi,\nu]),\ \ \ \xi\in\dd,\ \ \nu\in\h.\]
Similarly, an $H$-equivariant Lie algebra triple $(\dd,\g,\h)$ defines a linear action 
$\bullet$ of the group $H$ on $\g=\dd/\h$ and a Lie algebra action 
$\varrho\colon \g\to \mf{X}(H)$, satisfying the compatibility conditions of a \emph{matched pair $\g\bowtie H$ between a Lie group and 
a Lie algebra}. 
These actions are given by 
\begin{equation}\label{eq:hbulvrho}
  h\bullet\xi=(1-p_\h)\Ad_h \xi.\ \ h\in H,\ \xi\in\g
\end{equation}  
and $\iota(\varrho(\xi))\theta^R_h=p_\h(\Ad_h\xi)$
for $h\in H,\ \xi\in\g$. Furthermore, the $\g$-action on $H$ combines with 
the $\h$-action $\nu\to \nu^L$ to a Lie algebra action $\varrho\colon \dd\to\mf{X}(H)$, 
by the same formula:
\begin{equation}
 \iota(\varrho(\zeta))\theta^R_h=p_\h(\Ad_h\zeta),\ \ h\in H,\ \zeta\in\dd.
 \end{equation}
See Appendix~\ref{app:mat} for more details.
One has the following extension of Proposition~\ref{prop:VLAg1}.
\begin{proposition}\label{prop:VLAg2}
There is a 1-1 correspondence between 
\begin{itemize}
\item[(i)] Vacant $\ca{LA}$-groupoids $E\rra \g$ over groups $H\rra \pt$, 
\item[(ii)] $H$-equivariant Lie algebra triples $(\dd,\g,\h)$, and
\item[(iii)] Matched pairs $\g\bowtie H$.
\end{itemize}
Furthermore, if $x\in E_h$ with $s(x)=\xi$ then $t(x)=h\bullet\xi,\ \ \a(x)=\varrho(\xi)_h$. 
\end{proposition}
Here, the equivalence $(i)\Leftrightarrow (iii)$ was observed by Mackenzie \cite{Mackenzie:2003wj,Mackenzie:2008tz}.
Again, further details are given in Appendix~\ref{app:mat}. 

Let $(\q,\g,\mf{r})_\gamma$ 
and $f\colon \q\to \g$ be as in Section~\ref{subsec:genman}. The action $\bullet$ of $H$ on $\g$
defines an action on the dual space $\g^*\cong \mf{r}^\perp\subseteq\q$, which we again denote by $\bullet$. 
Recalling that 
$f^*$ takes the $\Ad_h$-invariant subspace $\g^*\cong \on{ann}(\h)$ isomorphically to $\mf{r}^\perp$, this action is characterized by 
\begin{equation}\label{eq:bulletonrperp} 
h\bullet f^*(\mu)=f^*(\Ad_h\mu),\ \ \ h\in H,\ \mu\in\on{ann}(\h).
\end{equation} 
The restriction of the metric on $\q$ to the subspace $\mf{r}^\perp$ is invariant 
under the $H$-action:
\begin{equation}\label{eq:note} \l h\bullet\nu,\ h\bullet\nu'\r=\l\nu,\nu'\r,\ \ \ 
\nu,\nu'\in\mf{r}^\perp,\ h\in H.
\end{equation}
This follows by writing $\nu=f^*(\mu),\nu'=f^*(\mu')$ and using the $H$-equivariance of 
$\beta^\sharp=f\circ f^*$. 
Write $p_{\mf{r}}\in\End(\q)$ for the projection to $\mf{r}$ along $\g$. 
We are in a position to give explicit structural formulas for Dirac Lie groups.
\begin{theorem}\label{th:mainB}
The Dirac Lie group structure defined by the $H$-equivariant Dirac Manin triple $(\dd,\g,\h)_{\beta}$ is 
given by 
\[ (\A,\,E)=(H\times\q,\,H\times\g).\]
Here  $\A$ carries the structure of an action Courant algebroid, for the 
action $\varrho^\q=\varrho\circ f\colon \q\to \mf{X}(H)$, 
\begin{equation}\label{eq:qact}\ \ \ \
\iota(\varrho^\q(\zeta))\theta^R_h=p_\h(\Ad_h f(\zeta)) .\end{equation} 
The $\ca{VB}$-groupoid structure has source and target maps $s,t\colon \A\to \g$, 
%
\[ s(h,\zeta)=p_{\mf{r}}^*(\zeta),\ \ \ \ t(h,\zeta)=h\bullet(1-p_{\mf{r}})(\zeta);\]
the inclusion of units is the map $\g\to \A,\ \ \xi\mapsto (e,\xi)$, 
and the multiplication of composable elements is given by 
\[(h_1,\zeta_1)\circ (h_2,\zeta_2)=(h_1h_2,\zeta_2+h_1^{-1}\bullet (1-p_{\mf{r}}^*)\zeta_1 ).\]
\end{theorem}
\begin{proof}
Let $\varrho^\q\colon \q\to \mf{X}(H)$ be the Lie algebra action described in Proposition~\ref{prop:trivial}.
By definition, $\varrho^\q(\zeta)=\a(\sigma)$ where $\sigma\in\Gamma(\A)$ is the constant 
section corresponding to $\zeta$ (i.e. $\sig\sim_T \zeta$). 
For $\zeta=\tau\in \mf{r}$, the constant section $\sig$ takes values in $\ker(t)$, hence 
$\sig_h=0_h\circ \sig_e$ for all $h$. Since the anchor $\a$ is a groupoid homomorphism, 
it follows that $\varrho^\q(\tau)_h=0_h\circ \varrho^\q(\tau)_0$. Equivalently, $\varrho^\q(\tau)$ is left-invariant. Hence  $\iota(\varrho^\q(\tau))\theta^R_h=\Ad_h(\varrho^\q(\tau)_e)=\Ad_h(f(\tau))=p(\Ad_h(f(\tau)))$, i.e.~  
$\varrho^q(\tau)=\varrho(f(\tau))$. On the other hand,  Equation \eqref{eq:hbulvrho} shows  $\iota(\varrho(\xi))\theta^R_h=p_\h(\Ad_h(f(\xi)))$ 
for $\xi\in\g$, hence $\varrho^\q(\xi)=\varrho(f(\xi))$. This proves $\varrho^\q=\varrho\circ f$. 

We next consider the groupoid structure. Use the trivialization to write elements of $\A$ in the form $x=(h,\zeta)$. Recall that on the vacant $\ca{VB}$-subgroupoid $E$, the trivialization is given by the 
source map. Hence, by Proposition~\ref{prop:VLAg2} we have 
\begin{equation}\label{eq:st}
 s(h,\xi)=\xi,\ \ t(h,\xi)=h\bullet\xi\end{equation}
for $(h,\xi)\in E$, and 
\begin{equation}\label{eq:stmE} (h_1,\xi_1)\circ  (h_2,\xi_2)=(h_1h_2,\xi_2)\end{equation}
for $(h_i,\xi_i)\in E$ with $h_2\bullet\xi_2=\xi_1$.
Consider now a general element $(h,\zeta)\in \A$. 
By definition of the trivialization, 
\begin{equation}(h,\zeta)=j(h,\zeta)\circ (e,\zeta).\end{equation}
Since $s(j(h,\zeta))=t(e,\zeta)=(1-p_{\mf{r}})\zeta$ by definition of $p_{\mf{r}}$, 
it follows that $j(h,\zeta)=(h,(1-p_{\mf{r}})\zeta)$. 
%
%
We conclude that $t(h,\zeta)=t(j(h,\zeta))=h\bullet(1-p_{\mf{r}})\zeta$, and $s(h,\zeta)=s(e,\zeta)=p_{\mf{r}}^*\zeta$. 

To find the groupoid multiplication, consider first a product $(h_1,0)\circ (h_2,\nu)$ with 
$(h_2,\nu)\in\ker(t)$, i.e. $\nu\in \mf{r}$. The product lies in $\ker(t)$, hence it is of the form 
$(h_1,0)\circ (h_2,\nu)=(h_1h_2,\nu')$ for some $\nu'\in\mf{r}$. Taking inner products with the identity 
$(h_1,h_2\bullet \xi)\circ (h_2,\xi)=(h_1h_2,\xi)$ for $\xi\in\g$, we obtain 
$\l 0,h_2\bullet \xi\r+\l \nu,\xi\r=\l \nu',\xi\r$, hence $\nu'=\nu$. Thus 
\[  (h_1,0)\circ (h_2,\nu)=(h_1h_2,\nu),\ \ \nu\in\mf{r}.\]
Similarly, consider a product $(h_1,\tau)\circ (h_2,0)=(h_1h_2,\tau')$ with $(h_1,\tau)\in \ker(s)$, thus $\tau\in\mf{r}^\perp$. Then $\tau'\in \mf{r}^\perp$, and $\l\tau,h_2\bullet \xi\r+\l 0,\xi\r=\l \tau',\xi\r$
for $\xi\in\g$, proving $\tau'=(h_2)^{-1}\bullet \tau$. 

For a general product $(h_1,\zeta_1)\circ (h_2,\zeta_2)$ of composable elements, write
$\zeta_2=(1-p_{\mf{r}})\zeta_2+p_{\mf{r}}\zeta_2$ and 
$\zeta_1=p_{\mf{r}}^*\zeta_1+(1-p_{\mf{r}}^*)\zeta_1$. We obtain
\[ \begin{split} (h_1,\zeta_1)\circ (h_2,\zeta_2)
&= (h_1,0)\circ (h_2,p_{\mf{r}}\zeta_2)+(h_1,(1-p_{\mf{r}}^*)\zeta_1)\circ (h_2,0)+
(h_1,p_{\mf{r}}^*\zeta_1)\circ (h_2,(1-p_{\mf{r}})\zeta_2)\\
&= (h_1h_2,p_{\mf{r}}\zeta_2+h_1^{-1}\bullet (1-p_{\mf{r}}^*)\zeta_1+(1-p_{\mf{r}})\zeta_2)\\
&= (h_1h_2,\zeta_2+h_1^{-1}\bullet (1-p_{\mf{r}}^*)\zeta_1).
\end{split}\]
\end{proof}

\subsection{Examples}
\subsubsection{The standard Dirac Lie group structure}
For any Lie group $H$, we have the $H$-equivariant Dirac Manin triple 
$(\h\ltimes\h^*,\h^*,\h)_\beta$, with $\beta$ the symmetric bilinear form given by the pairing.  
Since $\beta$ is non-degenerate and $\g=\h^*$ is Lagrangian, we have (cf.~ Example~\ref{ex:specialcases})
$\q=\dd$, with $f$ the identity. The projections $p_\h$ and $(1-p_\h^*)$ coincide, and 
our formulas specialize to
\[ s(h,\nu,\mu)=\mu,\ \ t(h,\nu,\mu)=\Ad_h\mu,\]
\[ (h_1,\nu_1,\mu_1)\circ (h_2,\nu_2,\mu_2)=(h_1h_2,\nu_2+\Ad_{h_1^{-1}}\nu_1,\mu_2).\]
The action of $\h\ltimes \h^*$ on $H$ is given by the left-invariant vector fields, 
$\varrho(\nu,\mu)=\nu^L$. This is the standard Dirac Lie group structure $(\A,E)=(\TH,T^*H)$, written in left-trivialization. 

\subsubsection{The Cartan-Dirac structure}\label{sec:CartDirStr}
Given a Lie group $G$ with an invariant metric on $\g$, one can form the Dirac
Manin triple $(\ol{\g}\oplus \g,\g_\Delta,0\oplus \g)_\beta$ where $\beta$ is given by the metric on 
$\ol{\g}\oplus \g$, and $\g_\Delta$ is the diagonal. Again $\q=\dd,\ f=\on{id}$. 
For $h\in H=\{1\}\times G$ we have $\Ad_h(\xi,\xi')=(\xi,\Ad_h\xi')$. It follows that the action $\bullet$ on 
$\g$ (hence also on $\mf{r}^\perp$) is the trivial action:
\[ h\bullet(\xi,\xi)=(1-p_\h)(\xi,\Ad_h\xi)=(\xi,\xi).\]
The formulas for the groupoid structure simplify to 
\[ s(h,\xi,\xi')=\xi',\ t(h,\xi,\xi')=\xi,\ \ \ 
(h_1,\xi_1,\xi_1')\circ (h_2,\xi_2,\xi_2')=(h_1h_2,\xi_1,\xi_2').\]
From $\iota(\varrho(\xi,\xi'))\theta^R_h=p_\h \Ad_h(\xi,\xi')
=p_\h(\xi,\Ad_h\xi')=\Ad_h\xi'-\xi$
we obtain
\[ \varrho(\xi,\xi')=(\xi')^L-\xi^R.\]
The resulting Dirac Lie group structure $(\A,E)=(G\times (\ol{\g}\oplus \g),\ G\times \g_\Delta)$ is the 
Cartan-Dirac structure from Example~\ref{ex:cardir}.

\subsubsection{Dirac Lie group structures over $H=\pt$}
If the group $H$ is trivial, then the Dirac Manin triple is of the form $(\dd,\dd,0)_\beta$. 
Dirac Lie group structures over $\pt$ are hence classified by Lie algebras $\dd$ with invariant elements 
$\beta\in S^2\dd$. (The same data also classify the $\ca{CA}$-groupoid structures $\A$ over $H=\pt$, since $\A$
extends uniquely to a Dirac Lie group structure by putting $E=\A^{(0)}$.) 
We find
\[ (\q,\g,\mf{r})_\gamma=(\dd\ltimes\dd^*_\beta,\dd,(\dd^*_\beta)^\perp)_{\wh{\beta}}\]
with $f$ the projection along $(\dd^*_\beta)^\perp$. The $\ca{VB}$-groupoid structure 
on $\dd\ltimes_\beta \dd^*_\beta\rra \dd$ is given by 
$s(\xi,\mu)=\xi, \ \ \ t(\xi,\mu)=\xi+\beta^\sharp(\mu)$,  
and the groupoid multiplication of composable elements is given by 
\[ (\xi_1,\mu_1)\circ (\xi_2,\mu_2)=(\xi_2,\mu_1+\mu_2).\]
Note that $\dd$ is a subgroupoid, as required. 

We can also classify the multiplicative Main pairs 
over $H=\pt$. First, by Drinfel'd's classification of $\ca{CA}$-groupoids over $H=\pt$ (see Remark~\ref{rem:CAgrpOvPt}), we may assume that $\A=\dd\ltimes\dd^*_\beta$.
Suppose $E$ is a multiplicative Dirac structure inside $\dd\ltimes\dd^*_\beta$, and 
let $\g:=E^{(0)}$. Then $E=E^\perp\subseteq \g^\perp=\dd\oplus \on{ann}(\g)$. 
The kernel of the source map $E\to \g$ has dimension  $\dim E-\dim\g=\dim\dd-\dim\g=\dim\on{ann}(\g)$, 
and is contained in $\mf{r}=\dd^*_\beta$. This shows $\on{ann}(\g)\subseteq E$. Thus $E=\g\oplus \on{ann}(\g)$ as a vector space. The form $\beta$ vanishes on $\on{ann}(\g)$ since $E$ is isotropic. 
Finally, $E$ is a Lie subalgebra of $\A$ 
if and only if $\g$ is a Lie subalgebra of $\dd$. Thus $E=\g\ltimes\on{ann}(\g)$ as a Lie algebra. 
Conversely, for any $\beta$-coisotropic Lie subalgebra $\g\subseteq\dd$, the semi-direct product $E=\g\ltimes\on{ann}(\g)$ defines a 
Lagrangian Lie subalgebra. We have $E\cap \dd^*=E\cap (\dd^*)^\perp=\on{ann}(\g)$, hence the restriction of both $s,t$ to $E$ are the projections to $\g$ along $\on{ann}(\g)$. It is then clear that $E$ is a subgroupoid. 
We have shown:
\begin{proposition}There is a 1-1 correspondence between
\begin{itemize}
\item[(i)] Multiplicative Manin pairs $(\A,E)$ over $H=\pt$, 
\item[(ii)] Dirac Manin pairs $(\dd,\g)_\beta$ (i.e. $\g$ is a $\beta$-coisotropic Lie subalgebra of $\dd$).  
\end{itemize}
\end{proposition}

\section{The Lagrangian complement $F$} \label{sec:compl}
Let $(\A,E)$ be a Dirac Lie group structure on $H$. We will show that $E$ has a distinguished 
Lagrangian complement. The splitting $\A=E\oplus F$ defines a bi-vector field $\pi_H$, and 
$(H,\pi_H)$ is a quasi-Poisson $\g$-space in the sense of Alekseev and Kosmann-Schwarzbach \cite{Alekseev99}. 

\subsection{Quasi-Poisson $\g$-manifolds}\label{subsec:quasip}
A \emph{quasi-Manin triple} $(\q,\g,\mf{\n})_\gamma$ is Lie algebra $\q$ with a 
non-degenerate element $\gamma\in (S^2\q)^\q$, together with a 
Lagrangian Lie subalgebra $\g$ and a Lagrangian subspace $\n$ complementary 
to $\g$. The quasi-Manin triple $(\q,\g,\n)_\gamma$ determines a trivector
$\chi\in \wedge^3\g\subseteq \wedge^3\q$ by the equation 
\[ \chi(\zeta,\zeta',\zeta'')=\l [\zeta,\zeta'],\zeta'' \r,\ \ \zeta,\zeta',\zeta'' \in\n,\]
as well as a \emph{cobracket}
\[ \partial\colon \g\to\wedge^2\g,\ \ \partial(\xi)(\zeta,\zeta')=\langle[\zeta,\zeta'],\xi\rangle,\quad \zeta,\zeta'\in\n,\ 
\xi\in\g.\]
Here $\g$ is identified with $\n^*$. Note that $\chi$ measures the failure of $\n$ to define a Lie subalgebra of $\q$. 
A \emph{quasi-Poisson $\g$-space} \cite{Alekseev99} for the quasi-Manin triple $(\q,\g,\n)_\gamma$ is a manifold $M$ with an action $\varrho_M\colon \g\to \mf{X}(M)$ and a bivector field $\pi_M\in\mf{X}^2(M)$ satisfying 
\begin{equation}\label{eq:q-Pspace}\begin{split}
 \hh [\pi_M,\pi_M]&=\varrho_M(\chi), \\
 \ca{L}_{\varrho_M(\xi)}\pi_M&=\varrho_M(\partial \xi),\quad \xi\in\g.
 \end{split}\end{equation}

As shown in \cite{Bursztyn:2009wi}, this definition is equivalent to a morphism of Manin pairs,
\begin{equation}\label{eq:k}
K\colon (\T M,TM)\da (\mf{q},\g).\end{equation}
Here $K$ determines the $\g$-action $\varrho_M\colon\g\to \mf{X}(M)$ by the condition 
$\varrho_M(\xi)\sim_K \xi,\ \xi\in\g$, and the bivector field $\pi_M$ on $M$ is described in terms of its graph as 
$\on{Gr}(\pi_M)=\n\circ K\subseteq \TM$, the `backward image' of $\n$. 
More generally, any morphism of Manin pairs $R\colon (\A,E)\da (\q,\g)$ 
determines a quasi-Poisson structure on $M$, by taking its composition with the morphism 
$(\TM,TM)\da (\A,E)$ from Example ~\ref{ex:canonicalmorphism}. Here the $\g$-action is 
$\varrho_M(\xi)=\a(e(\xi))$ where $\g\to \Gamma(E),\ \xi\mapsto e(\xi)$ is defined by the condition 
$e(\xi)\sim_R \xi$. The bi-vector field $\pi_M$ is determined by the splitting $\A=E\oplus F$, and is locally given by 
the formula $\pi_M=\hh \a(e_i)\wedge \a(f^i)$ where $e_i,f^j$ are sections of $E,F$ with 
$\l e_i,f^j \r=\delta^i_j$. (See e.g. \cite[Theorem 3.16]{LiBland:2009ul}). 

\subsection{Quasi-Poisson structures from Dirac Lie groups}\label{subsec:compl}
Let $(\dd,\g,\h)_\beta$ be an $H$-equivariant Dirac Manin triple, and let 
$(\q,\g,\mf{r})_\gamma$ be the Dirac Manin triple constructed from it. 
The following standard procedure turns the Lie algebra complement $\mf{r}$ into a \emph{Lagrangian} 
complement. As before we denote by $p_{\mf{r}}\in\End(\q)$ the projection to $\mf{r}$ along $\g$, so that $1-p_{\mf{r}}$ and 
$p_{\mf{r}}^*$ are the projections to $\g$ along $\mf{r},\mf{r}^\perp$, respectively. 
Their average $\hh ((1-p_{\mf{r}})+p_{\mf{r}}^*)$ is again a projection to $\g$, and its 
kernel $\n$ is the desired Lagrangian complement. Thus $\n$ is the mid-point between $\mf{r},\mf{r}^\perp$ in 
the affine space of complements to $\g$. If $\eps_i$ is a basis of $\g$, and $\phi^i$ a basis of $\mf{r}^\perp$ with $\l\eps_i,\phi^j\r=\delta_i^j$, the space $\n$ has basis 
\[ \nu^i=\phi^i-\hh \sum_j \l \phi^i,\phi^j\r\eps_j.\]
Note that the `r-matrix' 
\[ \hh \sum_i \eps_i\wedge \nu^i=\hh \sum_i \eps_i \wedge \phi^i\in \wedge^2\q\]
is independent of the choice of basis. Letting $\eps^i\in\g^*$ be the dual basis, we have $\phi^i=f^*(\eps^i)$, 
and using  $\l f^*(\eps^i),f^*(\eps^j)\r=\beta(\eps^i,\eps^j)$ we obtain 
\begin{equation}\label{eq:dualbasis} \nu^i= f^*(\eps^i)-\hh \sum_j \beta(\eps^i,\eps^j)\eps_j.\end{equation}
As explained in the previous section, the `trivializing morphism' $T\colon (\A,E)\da (\q,\g)$ 
gives $H$ the structure of a quasi-Poisson space for the quasi-Manin triple $(\q,\g,\mf{n})$. 
In terms of the trivialization $\A=H\times\q$, the Lagrangian complement $F=\n\circ T$ is simply the trivial bundle $H\times\mf{n}$. 
\begin{proposition}
In the affine space of Lagrangian complements to $E$ in $\A$, the sub-bundle $F$ is the mid-point between 
$\ker(s)$ and $\ker(t)$. One has, 
\begin{equation}\label{eq:defF}F=\{x\in\A|\ h\bullet s(x)+t(x)=0\} \end{equation}
where $h\in H$ indicates the base point of $x$. 
\end{proposition}
Note that $E$ is similarly given by a condition $h\bullet s(x)-t(x)=0$.
\begin{proof}
The first claim follows since the trivialization of $\A$ restricts to isomorphisms $E\cong H\times \g$, 
$\ker(t)\cong H\times \mf{r}$ and $\ker(s)\cong H\times \mf{r}^\perp$. The second part 
follows from the characterization of $\mf{n}$ as the kernel of $\hh((1-p_{\mf{r}})
+p_{\mf{r}}^*)$, since $h\bullet s(h,\zeta)=h\bullet p_{\mf{r}}^*(\zeta)$ and 
$t(h,\zeta)=h\bullet (1-p_{\mf{r}})(\zeta)$ in the trivialization. 
\end{proof}

Since $\a$ is given on constant sections of $\A=H\times\q$ by the action map $\varrho$, 
we obtain:
\begin{proposition}
For any Dirac Lie group structure $(\A,E)$ on $H$, with corresponding $\q$-action $\varrho\colon \q\to 
\mf{X}(H)$, one obtains a quasi-Poisson structure on $H$, with bivector field 
\[ \pi_H=\hh \sum_i \varrho(\eps_i)\wedge \varrho(f^*(\eps^i)).\]
and $\g$-action $\varrho_H=\varrho|_\g$. 
\end{proposition}

\subsection{Multiplicative properties}
We next consider the multiplicative aspects of the quasi-Poisson structure. The composition of morphisms
\[ (\A,E)\times (\A,E)\da (\A,E)\da (\q,\g)\]
gives $H\times H$ the structure of a quasi-Poisson $\g$-space $(H\times H,\pi_{H\times H})$, 
with the property that the underlying map $\on{Mult}_H\colon H\times H\to H$ is a morphism of quasi-Poisson manifolds.
The $\g$-action $\varrho_{H\times H}$ is computed as follows. 
Using the trivialization $E=H\times\g$, the equality $(h_1,\xi_1)\circ (h_2,\xi_2)=(h_1h_2,\xi)$ 
holds if and only if $\xi_2=\xi,\ \xi_1=h_2\bullet\xi$. Thus
\[ \varrho_{H\times H}(\xi)_{(h_1,h_2)}=(\varrho_H(h_2\bullet \xi)_{h_1},\varrho_H(\xi)_{h_2}).\]
Proposition~\ref{prop:VLAisMP} confirms that the multiplication in $H$ is equivariant for this twisted 
action. (More generally this holds true for any matched pair between a Lie group and a Lie algebra.) 

The bivector field $\pi_{H\times H}$ is determined by the splitting $(E\times E)\oplus F'$, 
where $F'$ is the backward image 
$F'=\n\circ (T\circ \Mult_\A)=F\circ \Mult_\A=\{(x_1,x_2)\in \A\times\A|\ x_1\circ x_2\in F\}$. Thus 
\[ \begin{split}
F'&=\{(x_1,x_2)\in \A\times\A|\ s(x_1)=t(x_2),\ \ h_1h_2\bullet s(x_2)+t(x_1)=0\}.
\end{split}\]
Since $F'$ and $F\times F$ are both Lagrangian complements to $E\times E$, there is a unique section $\lambda\in\Gamma(\wedge^2(E\times E))$ with the property that 
\[ F'=(\on{id}+\lambda^\sharp)(F\times F),\]
where $\lambda^\sharp\colon \A\times\A\to \A\times\A$ is the bundle map defined by $\lambda$, and
\[ \pi_{H\times H}=\pi_H^{(1)}+\pi_H^{(2)}+\a(\lambda)\]
where $\pi_H^{(i)},\ i=1,2$ is the bivector field $\pi_H$ on the $i$-th factor of $H\times H$, and $\a$ is 
the anchor map for $\A\times \A$. 
(See \cite[Proposition 1.18]{Alekseev:2009tg}.) 
It remains to compute $\lambda$.
\begin{proposition}  
The section $\lambda\in\Gamma(\wedge^2(E\times E))$ is given in terms of the trivialization $E=H\times\g$ as
\[ \lambda=-\hh \sum_{ij} \beta(\eps^i,\eps^j)\  (\eps_{i},0)\wedge (0,h_2^{-1}\bullet \eps_{j}).\]
Thus, 
\[ \a(\lambda)=-\hh \sum_{ij} \beta(\eps^i,\Ad_{h_2}\eps^j)\  
\varrho_H(\eps_i)^{(1)}\wedge \varrho_H(\eps_j)^{(2)}\in\mf{X}(H\times H),\]
where the superscripts $(1),(2)$ indicate the vector fields operating on the first resp.~ second 
$H$-factor. 
\end{proposition}
\begin{proof}
We will use the trivialization $\A=H\times\q$, and omit base points to simplify notation. 
For all $(\tau_1,\tau_2)\in F\times F$ at a given base point $h_1,h_2$, there is a 
unique element $(\xi_1,\xi_2)\in E\times E$ such that $(\tau_1+\xi_1,\tau_2+\xi_2)\in F'$.
Thus, $(\tau_1+\xi_1)\circ (\tau_2+\xi_2)\in F$, i.e.~ 
\[ s(\tau_1+\xi_1)=t(\tau_2+\xi_2),\ \ \ 
t(\tau_1+\xi_1)=-h_1h_2\bullet s(\tau_2+\xi_2)\]
Using $t(\xi_i)=h_i\bullet s(\xi_i)$ and 
$t(\tau_i)=-h_i\bullet s(\tau_i)$, and solving for $\xi_i=s(\xi_i)$, 
we find 
\[ \xi_1=-h_2\bullet s(\tau_2),\ \ \ \xi_2=h_2^{-1}\bullet s(\tau_1).\]
This shows that $\lambda$ is of the form $\lambda=\sum_{i} (\eps_i,0)\wedge (0,s^i)$
for some $s^i\in\g$ (depending on $h_1,h_2$). Taking $\tau_2=0$ and $\tau_1=\nu^i$ the basis element of $\n$, 
we find
\[ (0,s^i)=\lambda^\sharp(\nu^i,0)=(0,h_2^{-1}\bullet s(\nu^i))
=-\hh \sum_l \beta(\eps^i,\eps^j) (0,h_2^{-1}\bullet \eps_j).\qedhere
\]
\end{proof}


\begin{example} Let us specialize the formulas to the Cartan-Dirac structure from Section~\ref{sec:CartDirStr}, 
given by  the $G$-invariant Dirac Manin triple $(\ol{\g}\oplus \g,\g_\Delta,0\oplus \g)_\beta$. In this case
$\q=\dd,\ \mf{r}=\h$.  We have $\mf{r}^\perp=\g\oplus 0$, and $\n=\{(-\xi,\xi)|\xi\in\g\}$ is the anti-diagonal. 
Letting $e_i$ be a basis of $\g$, with $B$-dual basis $e^i$, the corresponding basis of $\g_\Delta$ is $\eps_i=(e_i,e_i)$, hence $f^*(\eps^i)=(-e^i,0)\in \mf{r}^\perp$, and the dual basis of $\n$ is $\nu^i=\hh(-e^i,e^i)$.  
The resulting bivector field on $G$ is 
\[ \pi_G=\hh \sum_i \varrho(e_i,e_i)\wedge \varrho(-e^i,0)=\hh\sum_i ((e_i)^L-(e_i)^R)\wedge (e^i)^R=
\hh \sum_i (e_i)^L\wedge (e^i)^R.\]
Since the action $\bullet$ is trivial, and $\beta(\eps^i,\eps^j)=-B(e^i,e^j)$, the section $\lambda\in \Gamma(\wedge^2(E\times E))$ is given by the formula $\lambda=\hh \sum_i (e_i,e_i)^{(1)}\wedge (e^i,e^i)^{(2)}$. 
\end{example}

\section{Exact Dirac Lie groups}\label{sec:exact}A Dirac Lie group structure $(\A,E)$ on $H$ is called \emph{exact} if the underlying Courant algebroid $\A$ is exact
(cf.~ Section~\ref{sec:prel}). We will show that in this case, $\A$ has a distinguished isotropic 
splitting, giving an identification $\A\cong \TH_\eta$ for a suitable closed 3-form 
$\eta\in\Om^3(H)$.
\subsection{Characterization in terms of Dirac Manin triples}
Exact Dirac Lie group structures have the following characterization in terms of the associated Dirac Manin triples.
\begin{proposition}
Let $(\A,E)$ be a Dirac Lie group structure on $H$, with corresponding Dirac Manin triple 
$(\dd,\g,\h)_\beta$. Then the Dirac Lie group structure is exact if and only if $\beta$ is non-degenerate and 
$\g$ is Lagrangian with respect to $\beta$.  
\end{proposition}
\begin{proof}
Let $(\q,\g,\mf{r})_\gamma$ be the  Dirac Manin triple constructed from $(\dd,\g,\h)_\beta$. 

$\Rightarrow$. Suppose $\A$ is exact. Then $\a_e\colon \q=\A_e\to \h=T_eH$ is surjective, with kernel $\g$. It follows that $\a_e$ restricts to an isomorphism $\mf{r}\to\h$; hence $f\colon \q\to \dd$ is an 
isomorphism. Since $f(\gamma)=\beta$, we 
conclude that $\beta$ is non-degenerate and $\g$ is Lagrangian. 

$\Leftarrow$. If $\beta$ is non-degenerate and $\g$ is Lagrangian, then (cf.~ Example~\ref{ex:specialcases}) the map $f\colon \q\to \dd$ gives an isomorphism $(\q,\g,\mf{r})_\gamma\cong (\dd,\g,\h)_\beta$. Hence the action 
$\varrho^\q=\varrho\circ f$ of $\q$ on $H$ is transitive.  Hence $\a\colon \A\to TH$ is surjective, and 
by dimension count its kernel is $\a^*(T^*H)$. 
\end{proof}

For the remainder of this section, we assume $(\A,E)$ is an exact Dirac Lie group structure, so that 
$\beta$ is non-degenerate, and $\g$ is Lagrangian in $\dd$. Using the isomorphism $(\q,\g,\mf{r})_\gamma\cong (\dd,\g,\h)_\beta$ we will write $\dd,\h,\beta$ in place of $\q,\mf{r},\gamma$ and we omit the 
letter $f$. 
We will write $p\in\End(\dd)$ for the projection to $\h$ along $\g$.

\subsection{Symmetries} \label{subsec:summ}
\begin{proposition}
Let $(\A,E)$ be an exact Dirac Lie group structure on $H$. The action of $H\times H$ on $H$, given by 
$(h_1,h_2).h=h_1 h h_2^{-1}$ lifts uniquely to an action on the vector bundle $E$, in such a way that:
\begin{enumerate}
\item The lifted action is by $\ca{CA}$ automorphisms: It preserves the metric $\l\cdot,\cdot\r$ and the bracket $\Cour{\cdot,\cdot}$, and the anchor map is equivariant.  
\item The lifted action is compatible with the groupoid structure, in the sense that 
the action map $(H\times H)\times\A\to \A$ is a groupoid homomorphism (using the pair groupoid structure 
$H\times H\rra H$).
\end{enumerate}
The action of the diagonal $H\subseteq H\times H$ on $\ker(\a_e)\cong\g$ is 
given by $h.\xi=p^*(\Ad_h\xi)$.
\end{proposition}
Note that this action of $H$ on $\g$ is different from the action $\bullet$, in general. 
The proposition says in particular that $t((h_1,h_2).x)=h_1.t(x),\ \ s((h_1,h_2).x)=h_2.s(x)$,
and 
\[ ((h_1,h_2).x)\circ ((h_1',h_2').x')=(h_1,h_2').(x\circ x')\]
if $s(x)=t(x')$ and $h_2=h_1'$.
\begin{proof}
Regard $\A$ as the reduction $C/C^\perp$ of $\TH\times (\ol{\dd}\oplus \dd)$, as in 
Section \ref{subsec:construct}. 
The lift of the $H\times H$-action to $\TH$, given as the direct sum of the tangent and cotangent 
lifts, is by Courant algebroid automorphisms, and is compatible with the groupoid structure $\TH\rra \h^*$. 
Similarly the action on $\ol{\dd}\oplus \dd$, where the $i$-th factor of $H$ acts by $\Ad$ on the $i$-th factor of $\dd$, is by Courant algebroid automorphisms, and is compatible with the pair groupoid structure 
$\ol{\dd}\oplus \dd\rra \dd$. The sub-bundle $C\subset \TH\times (\ol{\dd}\oplus \dd)$ given by \eqref{eq:defc} is invariant under this action, 
hence so is $C^\perp$, and we obtain an induced action on $\A=C/C^\perp$ by Courant automorphisms compatible 
with the groupoid structure. It has the desired property that the anchor is equivariant and that $\ker(t)$ 
is invariant. In particular, the action of the diagonal $H\subseteq H\times H$ on $\ker(t_e)\cong T_eH=\h$ is the adjoint action. Since the metric on $\A_e=\dd$ is preserved, we deduce the action on $\ker(\a_e)=\g$: 
For $\xi\in\g$ and $\zeta\in\h$, 
\[ \l h.\xi,\zeta\r=\l \xi,\,h^{-1}.\zeta\r=\l\xi,\,\Ad_{h^{-1}}\zeta\r=\l \Ad_h\xi,
\zeta\r=\l p^*(\Ad_h\xi),\zeta\r,\]
thus $h.\xi=p^*(\Ad_h\xi)$. For the uniqueness properties, suppose we are given any lift of the $H\times 
H$-action to Courant automorphisms of $\A$, compatible with the groupoid structure. By the equivariance properties of the anchor and target maps, the action on  $\ker(\a)\cong T^*H$ and $\ker(t)\cong TH$ are uniquely determined, 
hence so is the action on $\A=\ker(\a)\oplus \ker(t)$. 
\end{proof}
\begin{remark}
For possibly non-exact Dirac Lie group structures $(\A,E)$, a similar construction gives an action of 
$\mf{r}\times\mf{r}$ by infinitesimal Courant automorphisms of $\A$, where $\mf{r}=\ker(t_e)\subseteq \dd=\A_e$. 
\end{remark}

\begin{lemma}\label{lem:rightequiv}
Under the isomorphism $\A=H\times\dd$, the action of $\{e\}\times H\subset H\times H$ is given by 
$(e,h_2).(h,\zeta)=(hh_2^{-1},\Ad_{h_2}\zeta)$. 
\end{lemma}
\begin{proof}
Recall that $\A=C/C^\perp$ may be identified with the sub-bundle  $C\cap (TH\times (\g\oplus\dd))$. Under this identification, the trivialization is the projection to the last factor. Since $\{e\}\times H$ preserves this sub-bundle, the Lemma follows. 
\end{proof}
The description of the $H\times\{e\}$-action in terms of the trivialization is more involved, but will not be needed here. 

\subsection{Splitting} 
We next show that the Courant algebroid $\A$ admits a \emph{canonical} $H\times H$-invariant isotropic 
splitting $\mathsf{l}\colon TH\to \A$. Recall that $\n$ denotes the Lagrangian complement 
to $\g$ in $\dd=\q$, as described in Section~\ref{subsec:compl}. 

\begin{theorem}\label{th:splittt}
Let $(\A,E)$ be an exact Dirac Lie group structure on $H$.
There is a unique $\{e\}\times H$-equivariant Lagrangian splitting $\mathsf{l}\colon TH\to \A$ 
such that $\mathsf{l}(T_eH)=\n$. In fact, the splitting is $H\times H$-invariant. 
In terms of the trivialization $\A=H\times\dd$, the splitting is given on right-invariant vector 
fields by 
\begin{equation}\label{eq:splitright} \mathsf{l}(\nu^R)=(h,\Ad_{h^{-1}} (\nu-\hh p^*(\nu))),\ \ \nu\in\h.\end{equation}
The splitting determines an $H\times H$-equivariant isomorphism $\A\xra{\cong} \TH_\eta$, where 
$\eta$ is the closed bi-invariant \emph{Cartan 3-form}
\[ \eta=\f{1}{12}\l\theta^R,[\theta^R,\theta^R]\r.\]
\end{theorem}
\begin{proof}
Let $\Pi\colon \A\to \A$ denote the projection to $\ker(t)$ along $\ker(\a)$. Since $\ker(t),\ker(\a)$ are 
$H\times H$-invariant sub-bundles, the projection $\Pi$ is $H\times H$-equivariant.  
$1-\Pi,\ \Pi^*$ are projections to $\ker(\a)$ with kernels $\ker(t),\ \ker(s)$ respectively. 
The kernel of the projection $\hh((1-\Pi)+\Pi^*)$ is a Lagrangian sub-bundle complementary to $\ker(\a)$, defining an $H\times H$-invariant isotropic splitting $\mathsf{l}\colon TH\to \A$ having this sub-bundle as its range. 
At the group unit, $\Pi$ coincides with the projection $p\colon \dd=\A_e\to \h=\ker(t_e)$ along $\g=\ker(\a_e)$, 
hence $\mathsf{l}(T_eH)$ is the subspace $\n=\ker((1-p)+p^*)$. 

By Lemma \ref{lem:rightequiv}, the right hand side of Formula \eqref{eq:splitright} defines an $\{e\}\times H$-splitting. To show that it equals the left hand side, it is hence enough to check at $h=e$. Since 
$\nu-\hh p^*(\nu)$ lies in $\n=\ker((1-p)+p^*)$ and maps to $\nu$ under $p$, this proves 
\eqref{eq:splitright}. 

It remains to compute the resulting 3-form, using the formula \eqref{eq:actioncourant} for 
action Courant algebroids. For $\nu\in\h$ let $\ti{\nu}=\nu-\hh p^*(\nu)$ be its projection to $\n$. 
Equation \eqref{eq:actioncourant} gives
\[ \begin{split}
\l \Cour{\mathsf{l}(\nu_1^R),\,\mathsf{l}(\nu_2^R)},\ \mathsf{l}(\nu_3^R)\r
=&\l [\ti{\nu}_1,\ti{\nu}_2],\ti{\nu}_3\r + \l \L(\nu_1^R)\Ad_{h^{-1}}\ti{\nu}_2 ,\ti{\nu}_3\r
-\l \L(\nu_2^R)\Ad_{h^{-1}}\ti{\nu}_3,\Ad_{h^{-1}}\ti{\nu}_3\r
\\&+\l \L(\nu_3^R)\Ad_{h^{-1}}\ti{\nu}_1,\Ad_{h^{-1}}\ti{\nu}_3\r\\
 =&\l [\ti{\nu}_1,\ti{\nu}_2],\ti{\nu}_3\r -\l [\nu_1,\ti{\nu}_2],\ti{\nu}_3\r \\
&+\l [\nu_2,\ti{\nu}_1],\ti{\nu}_3\r -\l [\nu_3,\ti{\nu}_1],\ti{\nu}_2\r,
\end{split}\]
where we used $\L(\nu^R)\Ad_{h^{-1}}\zeta=-\Ad_{h^{-1}}[\nu,\zeta]$ for $\nu\in\h,\ \zeta\in\dd$.
Using the definition $\ti{\nu}_i=\nu_i-\hh p^*\nu_i$, note that terms with exactly two $p^*\nu_i$'s
cancel out. The term with three $p^*\nu_i$'s is zero, since $\g$ is a Lagrangian Lie subalgebra. 
Terms with exactly one $p^*\nu_i$ simplify, e.g. $\l [\nu_1,\nu_2],-\hh p^*\nu_3\r
=-\hh \l p[\nu_1,\nu_2],\nu_3\r=-\hh  \l [\nu_1,\nu_2],\nu_3\r$. Denoting terms with two or more $p^*\nu_i$'s by $\ldots$, we hence have 
\[ \l [\ti{\nu}_1,\ti{\nu}_2],\ti{\nu}_3\r=\l [\nu_1,\nu_2],\nu_3]\r-\f{3}{2} \l [\nu_1,\nu_2],\nu_3\r+\ldots=-\hh \l [\nu_1,\nu_2],\nu_3\r+\ldots,\]
and similarly $\l [\nu_1,\ti{\nu}_2],\ti{\nu}_3\r=0+\ldots$. We conclude
\[ \l \Cour{\mathsf{l}(\nu_1^R),\,\mathsf{l}(\nu_2^R)},\ \mathsf{l}(\nu_3^R)\r=-\hh \l [\nu_1,\nu_2],\nu_3\r.\]
It follows that $\eta=\f{1}{12}\l\theta^R,[\theta^R,\theta^R]\r$. 
\end{proof}

For later reference, note that the isomorphism $\TH_\eta\to \A,\ v+\alpha\mapsto \mathsf{l}(v)+\a^*(\alpha)$
is given in terms of the trivialization $\A=H\times \dd$ by the map $v+\alpha\mapsto (h,\zeta)$, 
with 
\begin{equation}\label{eq:form1}
\zeta=\Ad_{h^{-1}}\big((1-\hh p^*)\iota_v\theta^R_h+t(\alpha)\big),\ \ \Leftrightarrow \ \  
\iota_v\theta^R_h=p(\Ad_h\zeta),\ \ \alpha=\l\theta^R_h,\hh (p^*+(1-p))(\Ad_h\zeta)\r.
\end{equation}

\subsection{Multiplicative properties}
We next discuss the multiplicative aspects of the isomorphism $\A\cong \TH_\eta$. 
Let $\pr_1,\pr_2\colon H\times H\to H$ be the two projections, and let
$\sigma\in \Om^2(H\times H)$ be the 2-form
\begin{equation}\label{eq:varpi} \sigma =-\hh \l\pr_1^*\theta^{L},\pr_2^*\theta^{R}\r.
\end{equation}
Then the pull-back of $\eta$ under group multiplication satisfies
\begin{equation*}\label{eq:varsig} \on{Mult}_H^*\eta-\pr_1^*\eta-\pr_2^*\eta=\d\sigma.\end{equation*}
Consequently, the pair $(\Mult_H,\sigma)$ defines a Courant morphism 
\[ R_{\Mult_H,\sigma}\colon \TH_\eta\times \TH_\eta\da \TH_\eta.\]
\begin{proposition}\label{prop:multip}
The isomorphism $\A\to \TH_\eta$ intertwines the Courant morphisms 
\[ \on{gr}(\Mult_\A)\colon \A\times \A\da \A,\ \ \ 
R_{\on{Mult}_H,\sigma}\colon \TH_\eta\times \TH_\eta\da \TH_\eta.\] 
The $\ca{VB}$-groupoid structure on $\TH_\eta$, defined by this isomorphism, has 
source and target maps
\[ 
s(v+\alpha)=(TR_h)^*(\alpha+\hh \l v,\theta^R_h\r),\ \ 
t(v+\alpha)=(TL_h)^*(\alpha-\hh \l v,\theta^L_h\r),\]
for $v+\alpha\in T_hH_\eta$.  If $v_i+\alpha_i\in \T_{h_i}H_\eta,\ i=1,2$ are composable, then 
their groupoid product $v+\alpha=(v_1+\alpha_1)\circ (v_2+\alpha_2)\in \T_{h_1h_2}H_\eta$ is given by 
$v=v_1\circ v_2$ (product in the group $TH$) and 
\[ \alpha=(TL_{h_1^{-1}})^*(\alpha_2-\f{1}{2}\l \iota(v_1)\theta^L_{h_1},\theta^R_{h_2}\r).
\]  
\end{proposition}
\begin{proof}
The graph of the multiplication morphism for $\TH_\eta$ is obtained from that of $\TH$ as 
\[ R_{\on{Mult}_{H,\sigma}}=(\on{id}+\ti{\sigma}^\sharp) R_{\on{Mult}_H},\]
where $\ti{\sigma}\in \Om^2(G\times G\times G)$ is the pull-back of $\sigma$ under the map 
$(a,a_1,a_2)\mapsto (a_1,a_2)$. Let $H\times H$ act on $H\times H\times H$ by 
$(h_1,h_2)(a,a_1,a_2)=(h_1ah_2^{-1},h_1a_1,a_2 h_2^{-1})$. Since $\ti{\sigma}$ is invariant under this action, 
and $R_{\on{Mult}_H}$ is invariant under its lift to the Courant algebroid, $R_{\on{Mult}_{H,\sigma}}$ is again $H\times H$-equivariant. On the other hand, by Section~\ref{subsec:summ} the graph $\on{gr}(\Mult_\A)$ of the multiplication morphism 
is invariant under the $H\times H$-action on $\A\times \ol{\A}\times\ol{\A}$, given by 
\[ (h_1,h_2).(x,x_1,x_2)=((h_1,h_2).x,\ (h_1,e).x_1,\ (e,h_2).x_2).\]
It hence suffices to check that the isomorphism $\A_e\to \T_eH_\eta$ takes $\on{gr}(\Mult_\A)_{(e,e,e)}$ to 
$R_{\Mult_H,\sigma}|_{(e,e,e)}$. For $i=1,2$ let $\zeta_i\in\dd$ the elements corresponding to $v_i+\alpha_i\in \T_eH$ under the isomorphism $\A_e\cong \T_eH$, as in \eqref{eq:form1}.
The elements $\zeta_1,\zeta_2$ are composable if and only if $p^*\zeta_1=s(\zeta_1)=t(\zeta_2)=(1-p)\zeta_2$, 
that is (cf.~ \eqref{eq:form1}), 
\[ \hh p^*(v_1)+\alpha_1=-\hh p^*(v_2)+\alpha_2.\]
In this case, the composition $\zeta=\zeta_1\circ \zeta_2=\zeta_2+(1-p^*)\zeta_1$
corresponds, under \eqref{eq:form1},  to the element $v+\alpha$ with $v=v_1+v_2$ and 
$\alpha=\alpha_2-\hh p^*(v_1)$. The resulting equations read
\[ v=v_1+v_2,\ \ \ \alpha_1=\alpha-\hh p^*(v_2),\ \ \ \alpha_2=\alpha+\hh p^*(v_1). \]
Since 
\[ \begin{split}
\iota(v_1,v_2)\sigma_{e,e}&=-\hh \l v_1, \pr_2^*\theta^R_e\r+\hh \l\pr_1^*\theta^L_e,\ v_2\r
\\&=-\hh \l p^*(v_1),\pr_2^*\theta^R_e\r+\hh \l\pr_1^*\theta^L_e,\ p^*(v_2)\r \\
&=(\hh p^*(v_2),-\hh p^*(v_1)),
\end{split}\]
these are exactly the conditions for  
$(v+\alpha,\,v_1+\alpha_1,\,v_2+\alpha_2)\in R_{\Mult_H,\sigma}$. 
Consider the resulting groupoid structure on $\TH_\eta$. For $h=e$ and $v+\alpha\in \T_eH=\h\oplus \h^*$ we had found 
\[ s_e(v+\alpha)=\alpha+\hh p^*(v)=\alpha+\hh \l v,\theta^R_e\r.\]
The formula at general group elements follows by right translation, using the equivariance property 
of the source map. The argument for the target map is similar, and 
the formula for groupoid multiplication is just spelling out the definition of $R_{\Mult_H,\sigma}$.
\end{proof}

\begin{appendix}
\section{Composition of relations}\label{app:compos}
For more details on the theory summarized in this section, with particular emphasis on the symplectic 
setting, see Guillemin-Sternberg \cite{Guillemin:ww}.

A (linear) \emph{relation} $R\colon V_1\da V_2$ 
between vector spaces $V_1,\ V_2$ is a subspace $R\subseteq V_2\times V_1$. Write $v_1\sim_R v_2$ if $(v_2,v_1)\in R$. Any linear map $A\colon V_1\to V_2$ defines a relation $\on{gr}(A)$. In particular, the identity map 
of $V$ defines the diagonal relation $\on{gr}(\on{id}_V)=V_\Delta\subseteq V\times V$. 

The \emph{transpose relation} $R^\top \colon V_2\to V_1$ consists of all $(v_1,v_2)$ such that 
$(v_2,v_1)\in R$. We define 
\[\ker(R)=\{v_1\in V_1|\ v_1\sim 0\},\ \ \  
\on{ran}(R)=\{v_2\in V_2|\ \exists v_1\in V_1\colon (v_2,v_1)\in R\}\]
Given another relation $R'\colon V_2\da V_3$, the composition $R'\circ R\colon V_1\da V_3$ consists of all $(v_3,v_1)$ such that $v_1\sim_{R}v_2$ and $v_2\sim_{R'} v_3$ for some $v_2\in V_2$. 

We let $\on{ann}^\natural(R)\colon V_1^*\to V_2^*$ be the relation such that $\mu_1\sim_{\on{ann}^\natural(R)}\mu_2$ if $\l\mu_1,\,v_1\r=\l\mu_2,\,v_2\r$ whenever $v_1\sim_R v_2$. 
Thus $(\mu_2,\mu_1)\in \on{ann}^\natural(R)\Leftrightarrow (\mu_2,-\mu_1)\in \on{ann}(R)$. Note 
$\on{ann}^\natural(V_\Delta)=(V^*)_\Delta$, and more generally 
\begin{equation}\label{eq:annRel} \on{ann}^\natural(\on{gr}(A))=\on{gr}(A^*)^\top\end{equation}
for linear maps $A\colon V_1\to V_2$. 
Suppose $W_1,\ W_2$ are vector spaces with non-degenerate symmetric bilinear forms. A relation 
$L\colon W_1\da W_2$ is called \emph{Lagrangian} if $L\subseteq W_2\times\ol{W_1}$ is a Lagrangian 
subspace, where $\ol{W_1}$ indicates $W_1$ with the opposite bilinear form. 
\begin{lemma}\label{lem:lagRel}
If $L\colon W_1\da W_2$ and $L'\colon W_2\da W_3$ are Lagrangian relations, then $L'\circ L\colon W_1\da W_3$ 
is a Lagrangian relation.  
\end{lemma}
The analogous result for symplectic vector spaces is proved in detail in \cite{Guillemin:ww}; this proof carries over to Lagrangian spaces for vector spaces with split bilinear form.
\begin{lemma}\label{lem:ann}
For any relations $R\colon V_1\to V_2$ and $R'\colon V_2\to V_3$, 
one has  
$\on{ann}^\natural(R'\circ R)=\on{ann}^\natural(R')\circ \on{ann}^\natural(R)$. 
\end{lemma}
\begin{proof}
Let $W_i=V_i\oplus V_i^*$ with the metric given by the pairing $\l(v,\alpha),(v',\alpha')\r
=\l\alpha,v'\r+\l\alpha',v\r$. By Lemma~\ref{lem:lagRel}, the composition of Lagrangian relations  
\[ (R'\oplus \on{ann}^\natural(R'))\circ (R\oplus \on{ann}^\natural(R))
=(R'\circ R)\oplus (\on{ann}^\natural(R')\circ \on{ann}^\natural(R)).\]
is again a Lagrangian relation. This means that $\on{ann}^\natural$ of the first summand is equal to the second summand.
\end{proof}
The composition $R'\circ R$ can be regarded as the image of 
\[ R'\diamond R:=(R'\times R)\cap (V_3\times (V_2)_\Delta\times V_1)\]
under the projection to $V_3\times V_1$. 
\begin{lemma}
\[ \begin{split}
\dim(R'\diamond R)&=\dim R'+\dim R-\dim V_2+\dim (\ker(\on{ann}^\natural(R'))\cap \ker (\on{ann}^\natural(R)^\top)),\\
 \dim(R'\circ R)&=\dim (R'\diamond R)-\dim(\ker(R')\cap \ker(R^\top)).\end{split}\]
\end{lemma}
\begin{proof}
The codimension of $(R'\times R)+(V_3\times (V_2)_\Delta\times V_1)$ equals the 
dimension of its annihilator. It is thus equal to the dimension of 
\[ (\on{ann}^\natural(R')\times \on{ann}^\natural(R))\cap (0\times (V_2^*)_\Delta\times 0)
\cong \ker(\on{ann}^\natural(R'))\cap \ker (\on{ann}^\natural(R)^\top).\]
This gives the formula for $\dim(R'\diamond R)$. On the other hand, the projection 
$R'\diamond R\to R'\circ R$ has kernel the intersection 
$(R'\times R)\cap (0\times (V_2)_\Delta \times 0)\cong \ker(R')\cap \ker(R^\top)$.
\end{proof}
The composition of linear relations $R,R'$ is called \emph{transverse} if 
\[ \ker(R')\cap \ker(R^\top)=0,\  \ \ \ \ker(\on{ann}^\natural(R'))\cap \ker(\on{ann}^\natural(R)^\top)=0.\] 
The first condition is equivalent to the claim that for $(v_3,v_1)\in R'\circ R$, there is a unique $v_2\in V_2$ such that $(v_3,v_2)\in R'$ and $(v_2,v_1)\in R$. The second condition is equivalent to 
the transversality of $R'\times R$ with $V_3\times (V_2)_\Delta \times V_1$.
For transverse compositions, $R'\circ R$ varies smoothly 
with $R',R$.
Either of the two conditions in the transversality condition can be replaced with the dimension formula
$\dim(R'\circ R)=\dim(R')+\dim(R)-\dim V_2$. For Lagrangian relations, the dimension formula is automatic. 

More generally, consider (non-linear) relations between manifolds. Here, `clean composition' hypotheses are 
needed. Recall that the intersection of submanifolds $S_1,S_2\subseteq M$ is \emph{clean} (in the sense of Bott) 
if $S_1\cap S_2$ is a submanifold, and $T(S_1\cap S_2)=TS_1\cap TS_2$. Equivalently, the intersection is clean 
if at all points $x\in S_1\cap S_2$, there are local coordinates in which both $S_1,S_2$
are given as subspaces \cite[page 491]{ho:an3}. We say that the composition $R'\circ R$ of submanifolds 
$R \subseteq M_2\times M_1$ and $R'\subseteq M_3\times M_2$ is \emph{clean} if 
\begin{equation}\label{eq:inters}
 R'\diamond R=(R'\times R)\cap (M_3\times (M_2)_\Delta \times M_1)\end{equation}
is a clean intersection, and the map $R'\diamond R\to M_3\times M_1$ (forgetting the 
$M_2$-component) has constant rank. Thus  $R'\circ R$ is an (immersed) 
submanifold, and the map $R'\diamond R\to R'\circ R$ is a submersion. 

The composition is called transverse if the composition of tangent spaces is transverse everywhere. 
In this case $R'\diamond R$ is a smooth submanifold of dimension 
$\dim R'+\dim R-\dim M_2$, and the map $R'\diamond R\to R'\circ R$ is a covering.

\section{Matched pairs and $\ca{LA}$-groupoids}\label{app:mat}
A $\ca{VB}$-groupoid $V$  over $H$ is \emph{vacant} if $V^{(0)}=V|_{H^{(0)}}$. Equivalently, the source map is a fiberwise isomorphism. In \cite{Mackenzie:2003wj,Mackenzie:2008tz}, 
Mackenzie interpreted
a vacant $\ca{LA}$-groupoid $E\rra \g$ over a group $H\rra \pt$  
as a matched pair between a Lie algebra $\g$ and a Lie group $H$.\footnote{Note that in \cite{Mackenzie:2003wj}, Mackenzie uses the terminology of an \emph{interaction} rather than a matched pair.} In this Section we review and 
elaborate these results, proving Proposition~\ref{prop:VLAg2}, in particular.
\begin{lemma}
Suppose $E\rra \g$ is an $\ca{LA}$-groupoid over $H\rra \pt$. Then the map 
\[ (\a,t,s)\colon E\to TH\times(\g\oplus\g),\ x\mapsto (\a(x),t(x),s(x))\]
is a homomorphism of $\ca{LA}$-groupoids. If $E$ is vacant, then $(\a,t,s)$ is 
an embedding as a subbundle.
\end{lemma}
\begin{proof} 
Since $\a\colon E\to TH,\ s\colon E\to \g,\ t\colon E\to \g$ are morphisms
of Lie algebroids, the map $(\a,t,s)$ is one also.  Furthermore, if $s(x_1)=t(x_2)$, 
\[ \begin{split}
(\a(x_1\circ x_2),t(x_1\circ x_2),s(x_1\circ x_2))
&=(\a(x_1)\circ \a(x_2),\ t(x_1),s(x_1))\\
&=(\a(x_1),t(x_1),s(x_1))\circ (\a(x_1),t(x_2),s(x_2))
\end{split}\]
shows that $(\a,t,s)$ is a $\ca{VB}$-groupoid homomorphism. If $E$ is vacant, so that $t$ is a fiberwise isomorphism, the map $E\to H\times \g$ taking $x\in E_h$ to $(h,t(x))$ is an isomorphism. 
In particular, $(\a,t,s)$ is an embedding as a subbundle.  
\end{proof}

\begin{proposition}\label{prop:VLAisMP} 
Let $E\rra \g$ be a vacant $\ca{LA}$-groupoid. For any $\xi\in\g$ and
$h \in H$, let $h^{-1} \bullet\xi \in \g$ and $\varrho(\xi)_h \in
T_hH$ be defined by the condition that there exist $x\in E_h$ with
\begin{equation}\label{eq:atstransform}
 (\a(x),\,t(x),\,s(x))=(\varrho(\xi)_h,\,  h\bullet \xi,\,  \xi ).\end{equation}
Then the map $(h,\xi)\mapsto h\bullet\xi$ defines an action of $H$ on
$\g$, while $\varrho\colon \g\to \mf{X}(H)$ is an action of $\g$ on
$H$. These actions satisfy the compatibility conditions,
\begin{equation}\label{eq:comp1*} 
h\bullet[\xi_1,\xi_2]=[h\bullet \xi_1,\,h\bullet\xi_2]
+\L_{\varrho(\xi_1)}(h\bullet \xi_2)-\L_{\varrho(\xi_2)}(h\bullet \xi_1)
\end{equation}
and 
\begin{equation}\label{eq:comp2*} 
\varrho(\xi)_{h_1h_2}=(\Mult_H)_*(\varrho(h_2\bullet \xi)_{h_1},
\varrho(\xi)_{h_2}).\end{equation}
Conversely, given a pair of actions of $H$ on $\g$ and of $\g$ on $H$,
satisfying \eqref{eq:comp1*} and \eqref{eq:comp2*}, the span of the
sections $\alpha(\xi)$ is a vacant $\ca{LA}$-subgroupoid.
\end{proposition}
\begin{proof}
For $h\in H,\ \xi\in\g$ let 
$\alpha(\xi)_h=(\varrho(\xi)_h,\,  h\bullet \xi,\,  \xi)$ be the right hand side 
of \eqref{eq:atstransform}. 
Since the image of $E$ under $(\a,t,s)$ is a subgroupoid, we have 
\[ \alpha(h_2\bullet \xi)_{h_1}\circ \alpha(\xi)_{h_2}=\alpha(\xi)_{h_1h_2}.\]
Applying $\a,s$ to this identity gives \eqref{eq:comp2*} and the
action property $(h_1h_2)\bullet\xi=h_1\bullet h_2\bullet\xi$.  On the
other hand, since $E$ is a Lie subalgebroid,
$[\alpha(\xi_1),\alpha(\xi_2)]=\alpha([\xi_1,\xi_2])$.

Application of
$s,\a$ gives \eqref{eq:comp1*} and the property
$[\varrho(\xi_1),\varrho(\xi_2)]=\varrho([\xi_1,\xi_2])$. Conversely,
given actions $\varrho$ and $\bullet$ satisfying \eqref{eq:comp1*} and
\eqref{eq:comp2*}, let $E$ be the subbundle of $TH\times (\g\oplus\g)$
spanned by the sections $\alpha(\xi)$. The compatibility conditions
guarantee that it is a $\ca{VB}$-subgroupoid and also a Lie subalgebroid.
\end{proof} 

We note that \eqref{eq:comp1*} and \eqref{eq:comp2*} are exactly the
compatibility conditions for a \emph{matched pair $\g\bowtie H$
between a Lie group and a Lie algebra} as given in
\cite{Mackenzie:2003wj}. Therefore
Proposition~\ref{prop:VLAisMP} can be interpreted as proving a 1-1
correspondence between such matched pairs and vacant
$\ca{LA}$-groupoids over a Lie group.

By differentiating the action of $H$ on $\g$, we obtain a linear
representation of $\h$ on $\g$ (still denoted $\bullet$). Similarly,
since $\varrho(\xi)_e=0$, we may linearize the action of $\g$ on $H$
to obtain a linear representation $\dot\varrho$ of $\g$ on $\h$. 
Concretely, 
\[ \dot{\varrho}(\xi)(\tau)=[\varrho(\xi),\wt{\tau}]|_e\]
where $\wt{\tau}\in\mf{X}(H)$ with $\wt{\tau}|_e=\tau$.  
By
linearizing \eqref{eq:comp1*} and \eqref{eq:comp2*}, one obtains the
following conditions, for all $\xi,\xi_1,\xi_2\in\g$ and
$\tau,\tau_1,\tau_2\in\h$:
\begin{equation}\label{eq:compA}
\tau\bullet[\xi_1,\xi_2]=[\tau\bullet\xi_1,\xi_2]-[\tau\bullet\xi_2,\xi_1]-\dot{\varrho}(\xi_1)(\tau)\bullet\xi_2
+\dot{\varrho}(\xi_2)(\tau)\bullet\xi_1,
\end{equation}
\begin{equation}\label{eq:compB}
\dot{\varrho}(\xi)([\tau_1,\tau_2])=[\dot{\varrho}(\xi)(\tau_1),\tau_2]-[\dot{\varrho}(\xi)(\tau_2),\tau_1]
-\varrho(\tau_1\bullet\xi)(\tau_2)+\varrho(\tau_2\bullet\xi)(\tau_1).
\end{equation}
%
%
These are exactly the compatibility conditions for a \emph{matched
pair} of Lie algebras, as studied in
\cite{Majid:1990tba,KosmannSchwarzbach:1988vja}. The conditions are
equivalent to the statement that $\dd=\g\oplus \h$ carries a Lie
bracket, with $\g,\h$ as Lie subalgebras and such that
\begin{equation}\label{eq:xitau}
 [\xi,\tau]=\dot{\varrho}_{\xi}(\tau)-\tau\bullet\xi,\ \ \ \xi\in\g,\ \tau\in\h.\end{equation}
Let $p=p_\h\colon \dd\to \h$ be the projection to the second summand, and put $q=1-p$. 
\begin{proposition}\label{prop:ext}
\begin{enumerate}
\item 
The adjoint action $\Ad\colon H\to \End(\h)$ admits a unique extension
$\Ad\colon H\to \End(\dd)$ with the property
\begin{equation}\label{eq:action} \Ad_{h}\xi=h\bullet\xi+
\iota(\varrho(\xi))\theta^R_h\end{equation}
for all $h\in H$, $\xi\in\g$. Its derivative is the adjoint action of
$\h$ on $\dd$.
\item The action of $\g$ on $H$ combines with the $\h$-action 
$\nu\mapsto \nu^L$ to an action of the Lie algebra $\dd$. We have
\[\iota(\varrho(\zeta))\theta^R_h=p(\Ad_h\zeta),\ \ \zeta\in\dd.\]
\item 
The action $\Ad_h$ on $\dd$ is a Lie algebra automorphism of 
$\dd$.
\end{enumerate}
\end{proposition}
\begin{proof}
\begin{enumerate}
\item
For $h_1,h_2\in H$, 
\[ \begin{split}\Ad_{h_1}\Ad_{h_2}\xi&=
\Ad_{h_1}(h_2\bullet\xi+\iota(\varrho(\xi))\theta^R_{h_2})\\
&=h_1h_2\bullet\xi+\iota(\varrho(h_2\bullet\xi))\theta^R_{h_1}
+\iota(\varrho(\xi))\Ad_{h_1}\theta^R_{h_2})\\
&=h_1h_2\bullet\xi+\iota(\varrho(h_2\bullet\xi),\varrho(\xi))
(\Mult^*\theta^R)_{h_1,h_2}\\
&=h_1h_2\bullet\xi+\iota(\varrho(\xi))\theta^R_{h_1h_2}
\end{split}\]
where we used \eqref{eq:comp1*}.  This shows that \eqref{eq:action}
extends the adjoint action on $\h$ to an action on $\dd$. It is clear
that $\Ad_h$ extends the adjoint action on $\h$, and that $\f{\p}{\p
t}\Ad_{t\tau}\xi=[\tau,\xi]$ for $\xi\in\g,\ \tau\in\h$ (cf.~
\eqref{eq:xitau}).
\item 
Note first that for $\tau\in\h$, $\varrho(\tau)=\tau^L$. On the other
hand, for $\xi\in\h$ the formula
$\iota(\varrho(\xi))\theta^R_h=p(\Ad_h\xi)$ follows from
\eqref{eq:action}, by applying $p$ to both sides. 
For $\tau,\in\h,\ \xi\in\g$ we compute: 
\[ \iota([\varrho(\tau),\varrho(\xi)])\theta^R=
\L(\tau^L)\iota(\varrho(\xi))\theta^R=\L(\tau^L)(p(\Ad_h\xi))=p(\Ad_h[\tau,\xi])=\iota(\varrho([\tau,\xi]))\theta^R
\]
hence $[\varrho(\tau),\varrho(\xi)]=\varrho([\tau,\xi])$.
\item 
We have
$
\Ad_h[\tau,\xi]=[\Ad_h \tau,\Ad_h\xi],\ \ \xi\in\g,\ \tau\in\h,
$
by taking the derivative of $\Ad_h \Ad_{\exp(t\tau)}\xi=\Ad_{\exp(t\Ad_h\tau)}\Ad_h\xi$.  
For $\xi_1,\xi_2\in\g$, we have 
\[ \L_{\varrho(\xi_1)}(\Ad_h\xi_2)=[\iota(\varrho(\xi_1))\theta^R,\Ad_h\xi_2]=[p(\Ad_h\xi),\,\Ad_h\xi_2].\]
The $\g$-component of the desired equation $\Ad_h([\xi_1,\xi_2])=[\Ad_h\xi_1,\Ad_h\xi_2]$ now follows from 
Equation \eqref{eq:comp1*}:
\[\begin{split}
q(\Ad_h[\xi_1,\xi_2])
&=[q(\Ad_h\xi_1),q(\Ad_h\xi_2)]+\L_{\varrho(\xi_1)}q(\Ad_h\xi_2)-
\L_{\varrho(\xi_2)}q(\Ad_h\xi_1)\\
&=q\Big([q(\Ad_h\xi_1),q(\Ad_h\xi_2)]+[p(\Ad_h\xi_1),\,\Ad_h\xi_2]-[p(\Ad_h\xi_2),\,\Ad_h\xi_1]\Big)\\
&=q([\Ad_h\xi_1,\Ad_h\xi_2]).
\end{split}
\]
Similarly, the $\h$-component is obtained by the calculation,  
\[  \begin{split}
p(\Ad_h[\xi_1,\xi_2])&=\iota(\varrho([\xi_1,\xi_2]))\theta^R\\
&=\L_{\varrho(\xi_1)}\iota(\varrho(\xi_2))\theta^R-\L_{\varrho(\xi_2)}\iota(\varrho(\xi_1))\theta^R
-[\iota(\varrho(\xi_1))\theta^R,\  \iota(\varrho(\xi_2))\theta^R]\\
&=\L_{\varrho(\xi_1)} p(\Ad_h\xi_2)-\L_{\varrho(\xi_2)} p(\Ad_h\xi_1)
-[p(\Ad_h\xi_1),\,p(\Ad_h\xi_2)]\\
&=p([p(\Ad_h\xi_1),\,\Ad_h\xi_2])-p([p(\Ad_h\xi_2),\,\Ad_h\xi_1])-[p(\Ad_h\xi_1),\,p(\Ad_h\xi_2)]\\
&=p([\Ad_h\xi_1,\Ad_h\xi_2]).
\end{split}
\]
\end{enumerate}
\end{proof}
Proposition~\ref{prop:ext} shows that a vacant $\ca{LA}$-groupoid over $H\rra\pt$ 
determines an $H$-equivariant Lie algebra triple $(\dd,\g,\h)$, as in Definition 
\ref{def:triple}. The converse was established in Proposition~\ref{prop:dress}. 

Using the $H$-action on $\dd$ we can now characterize $E$ directly in terms of the 
$H$-equivariant triple. 

\begin{proposition}\label{prop:Echar}
Suppose $E\to H$ is a vacant $\ca{LA}$-groupoid over a group $H$, and define 
the $H$-equivariant triple $(\dd,\g,\h)$ as explained above. Then 
\[ \begin{split}
(\a,t,s)(E)=\{(v,\xi,\xi')|\ v\in T_hH,\ \xi,\xi'\in\g,\ \Ad_h\xi'-\xi=\iota(v)\theta^R_h\}
\end{split}\]
\end{proposition}
\begin{proof}
The condition $\Ad_h\xi'-\xi=\iota(v)\theta^R_h$ is just \eqref{eq:action}, proving the 
inclusion $\subseteq$. The opposite inclusion follows by dimension count: Given $\xi'\in\g$, 
the elements $\xi,\,v$ are determined as $\xi=q(\Ad_h\xi'),\ \iota(v)(\theta^R_h)=p(\Ad_h\xi)$.  
\end{proof}

The correspondence between $\ca{LA}$-groupoids and $H$-equivariant triples is compatible with 
morphisms. Suppose $E_i\rra \g_i,\ i=0,1$ are vacant $\ca{LA}$-groupoids over groups $H_i$. 
A \emph{morphism} (resp.~ \emph{comorphism}) from $E_0$ to $E_1$ is a Lie group homomorphism 
$\Phi\colon H_0\to H_1$, together with a vector bundle map 
$E_0\to E_1$ (resp.~ $\Phi^*E_1\to E_0$) whose graph is an 
$\ca{LA}$-subgroupoid of $E_1\times E_0$ along the graph of $\Phi$. If $E_i$ are vacant, 
so that $E_i|_e\cong \g_i$, such a morphism (resp.~ comorphism) defines a pair of Lie algebra 
homomorphisms $\d_e\Phi\colon \h_0\to \h_1$ and $\g_0\to \g_1$ (resp. $\g_1\to \g_0$). 
Let $\dd_i=\g_i\oplus\h_i$. 
\begin{proposition}\label{prop:vacmor}
\begin{enumerate}
\item If $E_0\to E_1$ is a morphism of $\ca{LA}$-groupoids, then the linear map 
$\dd_0\to \dd_1$ given as the direct sum of the maps $\g_0\to \g_1$ and $\h_0\to \h_1$ 
is a Lie algebra homomorphism, equivariant relative to the underlying group homomorphism 
$\Phi\colon H_0\to H_1$. 
\item If $\Phi^*E_1\to E_0$ is a comorphism of $\ca{LA}$-groupoids, then the subspace 
$\mf{r}\subseteq \dd_1\times\dd_0$, given as the direct sum of the graphs of the maps $\g_1\to \g_0$ and $\h_0\to \h_1$, is a Lie subalgebra, invariant under the action of $H_0$ (via its inclusion as $\on{gr}(\Phi)\subseteq H_1\times H_0$). 
\end{enumerate}
\end{proposition}
\begin{proof}
\begin{enumerate}
\item The statement is obvious if $E_0\to E_1$ is an inclusion. The general case reduces to that of an 
inclusion, by letting $E_1'=E_1\times E_0,\ H_1'=H_1\times H_0$, $\Phi'(h_0)=(\Phi(h_0),h_0)$, 
and with the inclusion $E_0\to E_1'$ the direct sum of the identity map with the map $E_0\to E_1$. 
\item Similar to (a), let $E_0'\hra E_1'=E_1\times E_0$ be the
inclusion of the graph of the map $\Phi^*E_1\to E_0$. By (a), applied
to inclusions one obtains an $H_0\cong \on{gr}(\Phi)$-equivariant Lie
algebra homomorphism $\dd_0':=\g_1\times
\h_0\hra \dd_1':=\dd_1\times\dd_0$. Its range is $\mf{r}$, which is
hence an $H_0$-invariant Lie subalgebra.
\end{enumerate}
\end{proof}

Conversely, given an $H_0\cong \on{gr}(\Phi)$-invariant Lie subalgebra $\mf{r}\subseteq \dd_1\times\dd_0$ given as the direct sum of the graphs of $\d_e\Phi$ and the graph of a Lie algebra homomorphism $\g_0\to \g_1$ (resp.~ $\g_1\to \g_0$), the resulting $\ca{LA}$-subgroupoid of $E_1\times E_0$ defines a  
morphism (resp.~ comorphism) of $\ca{LA}$-groupoids.

\section{Some constructions with $\ca{VB}$-groupoids}\label{app:fiberproducts}
We will base our discussion on the  following result 
%
\begin{proposition}\label{prop:moerdijk} \cite[$\mathsection$ 5.3]{moerdijk03}
Suppose $H,G,K$ are Lie groupoids, and $\phi\colon G\to K$ and $\psi\colon H\to K$ are morphisms of Lie groupoids.  
 If $\phi$ and $\psi$ are transverse, 
 then the fibered product 
$H\times_K G$ is a Lie groupoid, with $H^{(0)}\times_{K^{(0)}} G^{(0)}$ as its space of units. 
\end{proposition}
\begin{remark}
Note that, \cite[$\mathsection$ 5.3]{moerdijk03} makes the additional assumption that 
the restrictions $\phi^{(0)}=\phi|_{G^{(0)}}$ and $\psi^{(0)}=\psi|_{H^{(0)}}$ to the 
units are transverse. However, this property is automatic: Suppose $x\in G^{(0)}$, $y\in H^{(0)}$ 
are units with $w:=\phi(x)=\psi(y)$. Then  
$$T_xG=\ker(T_xs)\oplus T_xG^{0},\ \ \ T_yH=\ker(T_ys)\oplus T_yH^{0}, \ \ \ T_wK=\ker(T_ws)\oplus T_wK^{0}.$$
Since the tangent maps to $\phi,\ \psi$ respect these decompositions, their transversality implies that 
of $\phi^{(0)},\ \psi^{(0)}$.
\end{remark}
\begin{corollary}\label{cor:keragrp}
Suppose $\phi\colon V\to W$ is a fiberwise surjective homomorphism of $\mathcal{VB}$-groupoids over $G\to H$. 
 Then $\ker(\phi)$ is a $\mathcal{VB}$-subgroupoid of $V$.  
\end{corollary}
\begin{proof}
Since $\phi$ is fiberwise surjective, it is transverse to the zero section $H\to W$§. We may view $\ker(\phi)$ as the fibered product $V\times_W H$, where $H\to W$ is the inclusion of the 
zero section. By Proposition~\ref{prop:moerdijk}, it is a Lie groupoid. By Definition~\ref{def:CALAVBgroupoid}
it is a $\mathcal{VB}$-groupoid.  
\end{proof}

For the next result, we recall Pradines' observation 
\cite{Pradines:1988td} (see also \cite[$\mathsection$ 11.2]{Mackenzie05}) 
that the dual of a $\mathcal{VB}$-groupoid $V\to H$ carries a natural
structure of $\mathcal{VB}$-groupoid,
\[ \xymatrix{ {V^*} \ar@<2pt>[r]\ar@<-2pt>[r] \ar[d] &{\on{ann}(V^{(0)})}
\ar[d]\\
H \ar@<2pt>[r]\ar@<-2pt>[r]& H^{(0)} }\]
where $\on{ann}(V^{(0)})$ is the annihilator of $V^{(0)}$ in $V^*|_{H^{(0)}}$. 
The groupoid structure is given by $\l\alpha_1\circ \alpha_2,\ v_1\circ v_2\r=\l\alpha_1,v_1\r+\l\alpha_2,v_2\r$, for
composable elements $\alpha_1,\alpha_2\in V^*$ and $v_1,v_2\in V^*$, with
$\alpha_i$ having the same base points as $v_i$. 

Alternatively, one can define the groupoid multiplication in terms of its graph by
\begin{equation}\label{eq:grMultDual} \on{gr}(\Mult_{V^*})=\on{ann}^\natural(\on{gr}(\Mult_V))\end{equation}
(using the notation from Appendix~\ref{app:compos}). Writing the groupoid axioms in terms of compositions of relations, it then follows from the vector bundle version of Lemma~\ref{lem:ann}, that the $\ca{VB}$-groupoid axioms of $V$ imply those for $V^*$. 

Suppose now that  $\Phi:\colon V\to W$ is a morphism of $\ca{VB}$-groupoids, i.e.~ 
$$\on{gr}(\Phi)\circ\Mult_V\subseteq\Mult_W\circ\on{gr}(\Phi\times \Phi).$$
By application of Lemma~\ref{lem:ann} and \eqref{eq:grMultDual} one obtains the corresponding equation for $\Phi^*\colon W^*\to V^*$ holds. Thus we have proven 
\cite[Proposition~11.2.6]{Mackenzie05}, that the dual bundle map $\Phi^*\colon W^*\to V^*$ is again 
a morphism of $\ca{VB}$-groupoids.
%

\begin{corollary}
Suppose $C\subseteq V$ is a $\mathcal{VB}$-subgroupoid over groupoids $K\subseteq H$. Then 
$\on{ann}(C)\subseteq V^*$ is a  $\mathcal{VB}$-subgroupoid. 
Its space of objects is $\on{ann}(C)\cap \on{ann}(V^{(0)})$. 
\end{corollary}

\begin{proof}
Let $i\colon K\hra H$ be inclusion. 
By Proposition~\ref{prop:moerdijk}, the pull-back 
$i^* V^*\to K$ is a $\mathcal{VB}$-groupoid. 
It comes with a fiberwise surjective 
Lie groupoid homomorphism $i^*V^*\to C^*$, where the map on units is again fiberwise surjective. Its kernel is $\on{ann}(C)$. 
\end{proof}

A non-degenerate fiber metric $\l\cdot,\cdot\r$ on a
$\ca{VB}$-groupoid $V$ is \emph{multiplicative} if it satisfies
\begin{equation}\label{eq:multForm} \l v_1\circ v_2,v_1'\circ v_2'\r=\l v_1,v_1'\r+\l v_2,v_2'\r\end{equation}
for composable elements $v_1,v_2$, resp.~ $v_1',v_2'$, with $v_i$
having the same base points as $v_i'$. Equivalently, the graph
$\on{gr}(\on{Mult}_V)\subset V\times \ol{V}\times \ol{V}$ is an isotropic subbundle, where $\ol{V}$ denotes 
$V$ with the opposite fiber metric.  

The fiber metric $\l\cdot,\cdot\r$ defines a map $\Psi:V\to V^*$, and \eqref{eq:multForm} shows that  \begin{equation}\label{eq:Psiinside}\Psi\big(\on{gr}(\on{Mult}_V)\big)\subseteq \on{gr}(\on{Mult}_{V^*})=\on{ann}^\natural\big(\on{gr}(\on{Mult}_V)\big).\end{equation} Since, in addition, the fiber metric $\l\cdot,\cdot\r$ is non-degenerate, $\Psi$ defines an isomorphism of $\ca{VB}$-groupoids. This shows that \eqref{eq:Psiinside} is an equality and $$\Psi(V^{(0)})=(V^*)^{(0)}=\on{ann}(V^{(0)}).$$ Therefore both $\on{gr}(\on{Mult}_V)$ and $V^{(0)}$ are Lagrangian.

\begin{corollary}\label{cor:david}
Let $V\to H$ be a $\mathcal{VB}$-groupoid, equipped with a
multiplicative non-degenerate fiber metric.  Let $C\rra C\cap V^{(0)}$ be a
co-isotropic $\mathcal{VB}$-subgroupoid. 
Then $C^\perp\rra C^\perp\cap V^{(0)}$ 
is a $\mathcal{VB}$-subgroupoid of $C$, and hence the quotient inherits a
$\mathcal{VB}$-groupoid structure $C/C^\perp\rra (C\cap V^{(0)})/(C^\perp\cap V^{(0)})$. Moreover, the natural non-degenerate fiber metric on $C/C^\perp$ is multiplicative.
\end{corollary}
\begin{proof}
The identification $V^*\cong V$ identifies $\on{ann}(C)\cong C^\perp\subseteq 
C$. By the previous Corollary, this is a $\ca{VB}$-subgroupoid of $V$.
Hence  $C^\perp\to C$ is an inclusion of $\ca{VB}$-groupoids. Therefore, the dual morphism, \begin{equation}\label{eq:C*toCpstar}C^*\to (C^\perp)^*,\end{equation} is a surjective submersion of $\ca{VB}$-groupoids.
Thus, by Corollary~\ref{cor:keragrp}, the kernel $(C/C^\perp)^*\cong C/C^\perp$ of \eqref{eq:C*toCpstar} carries a natural $\ca{VB}$-groupoid structure.
Finally, it is clear that the restriction of the fiber metric $\l\cdot,\cdot\r$ to $C/C^\perp$ satisfies \eqref{eq:multForm}, since it does so for $V$.
\end{proof}

\end{appendix}

\providecommand{\bysame}{\leavevmode\hbox to3em{\hrulefill}\thinspace}
\providecommand{\MR}{\relax\ifhmode\unskip\space\fi MR }
\providecommand{\MRhref}[2]{%
  \href{http://www.ams.org/mathscinet-getitem?mr=#1}{#2}
}
\providecommand{\href}[2]{#2}

\end{document}